\documentclass[10pt,reqno,a4paper,twoside]{extarticle}

\usepackage{charter}

\usepackage[left=1in,right=1in,top=1in,bottom=1in]{geometry}

\usepackage{graphicx, varioref, amsfonts}
\usepackage{amsthm,amssymb,amsmath,mathrsfs,mathtools,stmaryrd} 
\usepackage{dsfont}
\usepackage{upgreek} 
\usepackage{esint} 
\usepackage{breakurl}
\usepackage{empheq}
\usepackage{enumerate,enumitem}
\usepackage{hyperref}
\usepackage{cleveref}
\usepackage{autonum}
\usepackage{soul}
\usepackage{cancel}
\usepackage{amscd}
\usepackage{srcltx}
\usepackage{ifthen}
\usepackage{float}
\usepackage{color}
\usepackage{comment}
\usepackage{chngcntr}
\usepackage{etoolbox}
\usepackage{fancyhdr}
\usepackage{breqn}
\usepackage{cite}
\usepackage[T1]{fontenc}

\usepackage{titlefoot}

\usepackage{authblk} 
\theoremstyle{plain}
\newtheorem{lemma}{Lemma}[section]

\newtheorem{proposition}[lemma]{Proposition}

\newtheorem{theorem}[lemma]{Theorem}

\newtheorem{assumption}[lemma]{Assumption}

\theoremstyle{definition}

\newtheorem{definition}[lemma]{Definition}

\newtheorem{example}[lemma]{Example}
\newtheorem{remark}[lemma]{Remark}

\newlist{todolist}{itemize}{2}
\setlist[todolist]{label=$\square$}
\usepackage{pifont}

\numberwithin{equation}{section}
\begin{document}

\title{Finite Dimensional Projections of HJB Equations in the Wasserstein Space
}
\newcommand\shorttitle{Finite Dimensional Projections of HJB Equations in the Wasserstein Space}
\date{June 23, 2025}

\author{Andrzej {\'{S}}wi{\k{e}}ch\footnote{School of Mathematics, Georgia Institute of Technology, Atlanta, GA 30332, USA; Email: swiech@math.gatech.edu}}
\author{Lukas Wessels\footnote{School of Mathematics, Georgia Institute of Technology, Atlanta, GA 30332, USA; Email: wessels@gatech.edu}}
\newcommand\authors{Andrzej {\'{S}}wi{\k{e}}ch and Lukas Wessels}

\affil{}

\maketitle

\unmarkedfntext{\textit{Mathematics Subject Classification (2020) ---} 28A33, 35D40, 35R15, 49L12, 49L25, 49N80, 93E20}


\unmarkedfntext{\textit{Keywords and phrases ---} Hamilton--Jacobi--Bellman equations in infinite dimension, viscosity solutions, mean field control, stochastic particle systems, Wasserstein space}


\begin{abstract}
This paper continues the study of controlled interacting particle systems with common noise started in [W. Gangbo, S. Mayorga and A. {\'{S}}wi{\k{e}}ch, \textit{SIAM J. Math. Anal.} 53 (2021), no. 2, 1320--1356] and [S. Mayorga and A. {\'{S}}wi{\k{e}}ch, \textit{SIAM J. Control Optim.} 61 (2023), no. 2, 820--851]. First, we extend the following results of the previously mentioned works to the case of multiplicative noise: (i) We generalize the convergence of the value functions $u_n$ corresponding to control problems of $n$ particles to the value function $V$ corresponding to an appropriately defined infinite dimensional control problem; (ii) we prove, under certain additional assumptions, $C^{1,1}$ regularity of $V$ in the spatial variable. The second main contribution of the present work is the proof that if $DV$ is continuous (which, in particular, includes the previously proven case of $C^{1,1}$ regularity in the spatial variable), the value function $V$ projects precisely onto the value functions $u_n$. Using this projection property, we show that optimal controls of the finite dimensional problem correspond to optimal controls of the infinite dimensional problem and vice versa. In the case of a linear state equation, we are able to prove that $V$ projects precisely onto the value functions $u_n$ under relaxed assumptions on the coefficients of the cost functional by using approximation techniques in the Wasserstein space, thus covering cases where $V$ may not be differentiable.
\end{abstract}


\section{Introduction}


This work is a continuation of the study of optimal control problems governed by weakly interacting particle systems with common noise started in \cite{gangbo_mayorga_swiech_2021,mayorga_swiech_2023}. In these works, the authors proved, among other results, that the value functions for the $n$ particle systems with additive common noise converge, as $n$ tends to infinity, to a value function corresponding to a control problem in the Wasserstein space. In the present work, we generalize this convergence to the case of multiplicative noise, and build on these results to further enhance the understanding of controlled particle systems.

More precisely, we fix an initial and terminal time $0\leq t <T < \infty$ as well as an initial state $\mathbf{x} = (x_1,\dots,x_n)\in (\mathbb{R}^d)^n$, and consider the optimal control problem of minimizing the cost functional
\begin{equation}
	J_n(t,\mathbf{x};\mathbf{a}(\cdot)) := \mathbb{E} \left [ \int_t^T \left ( \frac1n \sum_{i=1}^n \left ( l_1(X_i(s),\mu_{\mathbf{X}(s)}) + l_2(a_i(s)) \right ) \right ) \mathrm{d}s + \mathcal{U}_T(\mu_{\mathbf{X}(T)}) \right ]
\end{equation}
over some class of admissible controls $\mathbf{a}(\cdot) = (a_1(\cdot),\dots,a_n(\cdot)) : [t,T]\times \Omega^{\prime} \to (\mathbb{R}^d)^n$ on some probability space $\Omega'$, subject to the system of state equations
\begin{equation}\label{intro:state_equation}
	\begin{cases}
		\mathrm{d}X_i(s) = [-a_i(s) + b(X_i(s),\mu_{\mathbf{X}(s)})] \mathrm{d}s + \sigma(X_i(s),\mu_{\mathbf{X}(s)}) \mathrm{d}W(s)\\
		X_i(t) = x_i\in \mathbb{R}^d,
	\end{cases}
\end{equation}
$i=1,\dots,n$. Here, $l_1: \mathbb{R}^d \times \mathcal{P}_2(\mathbb{R}^d) \to \mathbb{R}$, $l_2: \mathbb{R}^d \to \mathbb{R}$ and $\mathcal{U}_T:\mathcal{P}_2(\mathbb{R}^d) \to \mathbb{R}$ denote the running and terminal cost functions, respectively, defined on the Wasserstein space $\mathcal{P}_2(\mathbb{R}^d)$. Moreover, the vector $\mathbf{X}(s) = (X_1(s),\dots,X_n(s)) \in (\mathbb{R}^d)^n$ denotes the state of the $n\in \mathbb{N}$ particles which are interacting through their empirical measure $\mu_{\mathbf{X}(s)} = \frac{1}{n} \sum_{i=1}^n \delta_{X_i(s)}$. Finally, the state equation is driven by an $\mathbb{R}^{d^{\prime}}$-valued Wiener process $(W(s))_{s\in [t,T]}$, and $b:\mathbb{R}^d \times \mathcal{P}_2(\mathbb{R}^d) \to \mathbb{R}^d$ and $\sigma: \mathbb{R}^d \times \mathcal{P}_2(\mathbb{R}^d)\to \mathbb{R}^{d\times d^{\prime}}$ are the drift and diffusion coefficients, respectively.

One of the major approaches to these kinds of control problems is the dynamic programming approach introduced by Richard Bellman in the 1950s, see \cite{bellman_1957}. The central object of study in this approach is the value function defined as
\begin{equation}
	u_n(t,\mathbf{x}) := \inf_{\mathbf{a}(\cdot)} J_n(t,\mathbf{x};\mathbf{a}(\cdot)).
\end{equation}
Assuming sufficient regularity, this value function can be used to derive necessary and sufficient optimality conditions, and to construct optimal feedback controls. It is well known that, under appropriate assumptions on the coefficients of the control problem, $u_n$ is the unique viscosity solution of the Hamilton--Jacobi--Bellman (HJB) equation
\begin{equation}\label{intro:finite_dimensional_HJB}
	\begin{cases}
		\partial_t u_n + \frac12 \text{Tr}(A_n(\mathbf{x},\mu_{\mathbf{x}}) D^2u_n) - \frac{1}{n} \sum_{i=1}^n H(x_i,\mu_{\mathbf{x}},nD_{x_i}u_n) = 0,\\
		\qquad\qquad\qquad\qquad\qquad\qquad\qquad\qquad\qquad\qquad\qquad (t,\mathbf{x})\in (0,T)\times (\mathbb{R}^d)^n\\
		u_n(T,\mathbf{x}) = \mathcal{U}_T(\mu_{\mathbf{x}}), \quad \mathbf{x}\in (\mathbb{R}^d)^n,
	\end{cases}
\end{equation}
where $A_n(\mathbf{x},\mu_{\mathbf{x}})$ is an $nd\times nd$-matrix consisting of $n^2$ block matrices $(A_n)_{ij}(\mathbf{x},\mu_{\mathbf{x}})=\sigma(x_i,\mu_{\mathbf{x}}) \sigma^{\top}(x_j,\mu_{\mathbf{x}})$, $i,j=1,\dots,n$, of dimensions $d\times d$, and $H:\mathbb{R}^d \times \mathcal{P}_2(\mathbb{R}^d) \times \mathbb{R}^d \to \mathbb{R}$ denotes the Hamiltonian given by
\begin{equation}
	H(x,\mu,p) = -b(x,\mu) \cdot p - l_1(x,\mu) + \sup_{q\in\mathbb{R}^d} (q\cdot p - l_2(q)).
\end{equation}

We are interested in the behavior of the value functions $u_n$ as the number of particles $n$ tends to infinity. A formal computation suggests that the value functions $u_n$ converge to the value function associated with an optimal control problem in $\mathcal{P}_2(\mathbb{R}^d)$, see \cite[Section 3.3]{djete_possamai_tan_2022}, which solves the limiting equation
\begin{equation}\label{intro:HJB_on_Wasserstein_space}
	\begin{cases}
		\partial_t \mathcal{U}(t,\mu) + \frac12 \int_{\mathbb{R}^d} \text{Tr} \left [ D_x \partial_{\mu} \mathcal{U}(t,\mu)(x)\sigma(x,\mu)\sigma^{\top}(x,\mu) \right ] \mu(\mathrm{d}x)\\
		\quad + \frac12 \int_{(\mathbb{R}^d)^2} \text{Tr} \left [ \partial_{\mu}^2 \mathcal{U}(t,\mu)(x,x^{\prime}) \sigma (x,\mu) \sigma^{\top}(x^{\prime},\mu) \right ] \mu(\mathrm{d}x) \mu(\mathrm{d}x^{\prime}) \\
		\quad - \int_{\mathbb{R}^d} H(x,\mu,\partial_{\mu} \mathcal{U}(t,\mu)(x)) \mu(\mathrm{d}x) =0, \quad (t,\mu)\in (0,T) \times \mathcal{P}_2(\mathbb{R}^d)\\
		\mathcal{U}(T,\mu) = \mathcal{U}_T(\mu),\quad \mu \in \mathcal{P}_2(\mathbb{R}^d).
	\end{cases}
\end{equation}
Since this equation is posed on the Wasserstein space $\mathcal{P}_2(\mathbb{R}^d)$, which is not a linear space, it is notoriously hard to solve. In \cite{cardaliaguet_2013,lions_2007-2011}, P.-L. Lions introduced the notion of the so-called L-viscosity solution, in which such equations are lifted to a space of square-integrable random variables $E$ on some non-atomic measure space $\Omega$. In the present work, we take $E:= L^2(\Omega;\mathbb{R}^d)$, $\Omega = (0,1)$. In this approach, the lift $V:[0,T]\times E \to \mathbb{R}$ of a function $\mathcal{V}:[0,T]\times\mathcal{P}_2(\mathbb{R}^d) \to \mathbb{R}$ is defined as $V(t,X) := \mathcal{V}(t,X_{\texttt{\#}} \mathcal{L}^1)$, where $X_{\texttt{\#}}\mathcal{L}^1$ denotes the push-forward of the one-dimensional Lebesgue measure $\mathcal{L}^1$ through the random variable $X\in E$. Now, $\mathcal{V}$ is called an L-viscosity solution of equation \eqref{intro:HJB_on_Wasserstein_space}, if its lift $V$ is a viscosity solution\footnote{Here, the notion of viscosity solution for equations in Hilbert spaces is the natural extension of the one used in finite dimensional spaces, see also Appendix A.} of the lifted HJB equation\footnote{Note that the second order term coincides with the term $\text{Tr}(\Sigma(X)(\Sigma(X))^{\ast} D^2V)$, where the trace is taken in $E$.}
\begin{equation}\label{intro:lifted_HJB}
	\begin{cases}
		\partial_t V + \frac12 \sum_{m=1}^{d^{\prime}} \langle D^2 V \Sigma(X) e^{\prime}_m, \Sigma(X) e^{\prime}_m \rangle_E - \tilde{H}(X,X_{\texttt{\#}}\mathcal{L}^1, DV) =0,\\
		\qquad\qquad\qquad\qquad\qquad\qquad\qquad\qquad\qquad\qquad\qquad (t,X)\in (0,T)\times E\\
		V(T,X) = U_T(X), \quad X\in E,
	\end{cases}
\end{equation}
where $(e_m^{\prime})_{m=1,\dots, d^{\prime}}$ denotes the standard basis of $\mathbb{R}^{d^{\prime}}$, $\Sigma(X)e_m^{\prime}(\omega) := \sigma(X(\omega),X_{\texttt{\#}}\mathcal{L}^1)e_m^{\prime}$, the lifted Hamiltonian $\tilde{H}: E\times \mathcal{P}_2(\mathbb{R}^d)\times E \to \mathbb{R}$ is given by
\begin{equation}\label{definition_H_tilde}
	\tilde{H}(X,\beta,P) := \int_{\Omega} H(X(\omega),\beta,P(\omega)) \mathrm{d}\omega
\end{equation}
and $U_T$ is the lift of $\mathcal{U}_T$. Since $E$ is a Hilbert space, the existing theory of viscosity solutions can be applied to the lifted equation. We refer to \cite{fabbri_gozzi_swiech_2017} for the theory of HJB equations in Hilbert spaces.


Differential equations on the Wasserstein space such as \eqref{intro:HJB_on_Wasserstein_space} and other similar equations in abstract spaces arise not only in mean-field control problems, but also in variational problems in spaces of probability measures, control problems with partial observation, differential games, large deviations and fluid dynamic problems. The theory of such equations is relatively new, however it has been growing rapidly and accordingly, the literature on the subject is quite extensive, see for example \cite{AmFe,badreddine_frankowska_2022,bandini_cosso_fuhrman_pham_2019,bayraktar_cosso_pham_2018,bayraktar_ekren_zhang_2023,bensoussan_graber_yam_2024,bensoussan_yam_2019,bertucci_2024,burzoni_ignazio_reppen_soner_2020,CKT,CKTT,cosso_gozzi_kharroubi_pham_rosestolato_2023,cosso_gozzi_kharroubi_pham_rosestolato_2024,cosso_pham_2019,daudin_jackson_seeger_2023,daudin_seeger_2024,feng_katsoulakis_2009,feng_nguyen_2012,feng_swiech_2013,feng_mikami_zimmer_2021,gangbo_nguyen_tudorascu_2008,gangbo_mayorga_swiech_2021,gangbo_tudorascu_2019,GHN,hynd_kim_2015,jerhaoui_prost_zidani_2024,jimenez_marigonda_quincampoix_2020,jimenez_marigonda_quincampoix_2023,jimenez_quincampoix_2018,kraaij_2022,kraaij_2020,LSZ,mimikos-stamatopoulos_2023,pham_wei_2017,pham_wei_2018,soner_yan_2024,soner_yan_2024_2}. 
For partial differential equations (PDEs) on finite dimensional spaces, it is well-known that the notion of viscosity solution based on sub- and superjets is equivalent with the notion based on test functions. This equivalence carries over to equations on (sufficiently regular) infinite dimensional linear spaces, however subtle issues arise when equations have unbounded terms and, generally, definitions of viscosity solutions are based on the proper choice of test functions. The situation is even more complicated for equations on spaces of probability measures or more general metric spaces. Several notions of viscosity solutions have been proposed, based on various notions of derivatives, differentials or jets, different and special test functions, various interpretations of terms in the equations, use of upper and lower Hamiltonians, etc. These are dictated by things like the type of space, special structure of the PDE, the ability to create local maxima and minima, and the need to be able to carry out a version of the doubling technique to prove comparison principles.
In the case of the Wasserstein space, the lifting procedure introduced by P.-L. Lions in \cite{cardaliaguet_2013,lions_2007-2011} and briefly recalled above opens the door to the application of the theory for Hilbert spaces, which we are going to use. In fact, once our equation  \eqref{intro:HJB_on_Wasserstein_space} is lifted, the whole analysis is carried out in our Hilbert space $E$. We also refer to \mbox{\cite{cardaliaguet_delarue_lasry_lions_2019,carmona_delarue_2018,carmona_delarue_2018_2}} where this lifting approach is applied to mean-field games.


Since the definition of a viscosity solution comes with low regularity, but differentiability is needed to construct optimal feedback controls, proving higher regularity is an important issue. The first result in this direction can be found in \cite{lions_1988}. More recently, this issue has been addressed in \cite{bensoussan_graber_yam_2024,bensoussan_yam_2019,defeo_swiech_wessels_2023,gangbo_meszaros_2022,mayorga_swiech_2023}.



It has been known for a long time that the solution of an uncontrolled particle system of the type \eqref{intro:state_equation} converges to the solution of a mean-field stochastic differential equation, see \cite{sznitman_1991}. This phenomenon is known as propagation of chaos. It is not obvious, however, that this convergence ``commutes'' with the optimization in the control problem. The first general result in this direction in the case of idiosyncratic noise seems to be \cite{lacker_2017}, see also \cite{budhiraja_dupuis_fischer_2012,carmona_delarue_2015,fornasier_solobmrino_2014,fischer_livieri_2016} for earlier contributions in this direction. Subsequently, this convergence was generalized to the case of common noise, see \cite[Theorem 2]{djete_possamai_tan_2022}, and even to the so-called extended mean-field control problems, i.e., control problems that involve the law of the control, see \cite{djete_2022}.


Once the convergence is established, the question about the rate of convergence arises. The first result in this direction for equations with both, idiosyncratic and common noise, was obtained in \cite{germain_pham_warin_2022}, assuming the existence of a classical solution of the infinite dimensional HJB equation. It was proved there that the rate of convergence is of order $\mathcal{O}(1/n)$. A simpler argument to obtain this was pointed out in \mbox{\cite[Section 1.3]{cardaliaguet_daudin_jackson_souganidis_2023}}, see also \mbox{\cite{cardaliaguet_delarue_lasry_lions_2019,carmona_delarue_2018_2}}, where a similar method was applied in the context of mean-field games. It should also be noted that, even though no explicit statement was made there, it was already noticed in  \cite{germain_pham_warin_2022} that in the absence of idiosyncratic noise, if $\mathcal{U}$ is a smooth solution of \eqref{intro:HJB_on_Wasserstein_space} then $u_n(t,\mathbf{x}):=\mathcal{U}(t,\mu_{\mathbf{x}})$ is a solution of \eqref{intro:finite_dimensional_HJB}. If the infinite dimensional value function is not smooth, the situation becomes much more delicate and has recently received much attention, see \cite{cardaliaguet_daudin_jackson_souganidis_2023,cardaliaguet_jackson_mimikos-stamatopoulos_souganidis_2023,cardaliaguet_souganidis_2023,daudin_delarue_jackson_2024}. In all of these works, the presence of the idiosyncratic noise introduces a regularization of the finite dimensional value functions.

Very recently, in \cite{liao_meszaros_mou_zhou_2024}, the authors considered a closely related problem and obtained several results that are similar to ours. Let us highlight some similarities and differences to our work. In the mentioned work, the authors investigate so-called neural stochastic differential equations (SDEs), which can be viewed as a system of controlled coupled SDEs, where the control is the same for each particle. In their model, the SDEs are driven solely by an additive common noise. Introducing several regularizations, in particular idiosyncratic noise at the level of the finite dimensional SDEs, and using backward SDE methods as well as the stochastic maximum principle, the authors show that the finite dimensional value functions and the optimal feedback controls converge to functions on the space of probability measures. Moreover, it is shown that the limits project precisely onto the finite dimensional objects. Note that these limits are not identified as the value function and optimal control of an associated infinite dimensional control problem. The analysis in \cite{liao_meszaros_mou_zhou_2024} is based on finite dimensional problems while in our manuscript the infinite dimensional HJB equation \eqref{intro:HJB_on_Wasserstein_space} plays the central role.


In \cite{gangbo_mayorga_swiech_2021,mayorga_swiech_2023}, in the case of a constant diffusion coefficient $A_n(\mathbf{x},\mu_{\mathbf{x}}) \equiv A_n$, it was shown that the solutions $u_n$ of \eqref{intro:finite_dimensional_HJB} converge uniformly on bounded sets in the Wasserstein space to a function $\mathcal{V}:[0,T]\times\mathcal{P}_2(\mathbb{R}^d) \to \mathbb{R}$. Moreover, it was shown that $\mathcal{V}$ is an L-viscosity solution of equation \eqref{intro:HJB_on_Wasserstein_space}, and in \cite{mayorga_swiech_2023} the authors introduced an optimal control problem in $E$ for which equation \eqref{intro:lifted_HJB} is the associated HJB equation. Under certain additional assumptions, it was proved in \cite{mayorga_swiech_2023} that the lift $V$ of $\mathcal{V}$ has $C^{1,1}$ regularity in the spatial variable. In the present work, first, we generalize the convergence and the characterization of the limit to the case of multiplicative noise. The proof relies on uniform continuity estimates for the finite dimensional value functions $u_n$. In contrast to \cite{mayorga_swiech_2023}, these estimates are obtained using the stochastic representation of the value function. Then, we show that the $C^{1,1}$ regularity of the lift $V$ carries over to the case of multiplicative noise under essentially the same assumptions. As in \cite{mayorga_swiech_2023}, the proof uses the fact that $V(t,\cdot)$ is of class $C^{1,1}$ if it is semiconcave and semiconvex, see \cite{lasry_lions_1986}. The semiconcavity and semiconvexity are obtained using stochastic methods.

In the second part of the paper, we investigate the relationship between the value functions $u_n$ corresponding to the controlled $n$ particle systems and the projection $V_n:[0,T]\times (\mathbb{R}^d)^n \to \mathbb{R}$ of the value function $V$ defined as $V_n(t,\mathbf{x}) := V(t,\sum_{i=1}^n x_i \mathbf{1}_{A^n_i})$, where $\mathbf{1}_{A^n_i}$ denotes the indicator function on $A^n_i = (\frac{i-1}{n}, \frac{i}{n})\subset (0,1) = \Omega$, $i=1,\dots,n$. Notice that $V(t,\sum_{i=1}^n x_i \mathbf{1}_{A^n_i})=\mathcal{V}(t,\mu_{\mathbf{x}}).$ Using the stochastic representation of the value function $V$, we are able to show the inequality $V_n\leq u_n$. The opposite inequality is proved using viscosity methods under the assumption that $DV$ is continuous. This required regularity of $V$ in particular follows from the results of Section \ref{section:C11_regularity}, see Remark \ref{remark_joint_continuity}(ii), however we point out that the $C^{1,1}$ regularity in the spatial variable is not needed to prove this inequality. Taking both inequalities together, we obtain that, under the mentioned regularity assumption, the value function $V$ projects precisely onto the finite dimensional value functions $u_n$, which in particular conclusively answers the question about the convergence rate in this setting. Note that in contrast to \cite{liao_meszaros_mou_zhou_2024}, we also give a characterization of the limit $V$ as the value function of a control problem in a Hilbert space. Based on this projection property, we then show that optimal controls of the finite dimensional control problems lift to optimal controls of the infinite dimensional control problem. Moreover, we show that piecewise constant optimal controls of the infinite dimensional control problem project onto optimal controls of the finite dimensional control problem. This refines the corresponding result from \cite{liao_meszaros_mou_zhou_2024}.

In the last part, we show that in the case of a linear state equation, the previously discussed projection property can be generalized to cases where the value function may not be differentiable. This proof relies on an approximation of functions on the Wasserstein space. While such approximations are known, the existing methods in the literature (see e.g. \cite{cosso_martini_2023,daudin_delarue_jackson_2024,daudin_jackson_seeger_2023}) either require stronger assumptions on the function that is to be approximated, or they do not preserve convexity, which is crucial for our analysis. Therefore, we modify the technique developed in \cite{cosso_martini_2023}.

After the completion of the present manuscript, the manuscript \cite{cecchin_daudin_jackson_martini_2024} on the convergence problem for mean-field control was posted on arXiv. While the main part of that work considers the case of common and idiosyncratic noise, Section 6 is concerned with the case of zero idiosyncratic noise. In particular, in \cite[Proposition 6.1]{cecchin_daudin_jackson_martini_2024}, they make a similar observation as ours in Proposition \ref{proposition:V_n_leq_u_n} here. Moreover, in \cite[Proposition 6.11]{cecchin_daudin_jackson_martini_2024}, the authors prove the exact projection property in their setting and under different assumptions.


The remainder of the paper is organized as follows. In Section \ref{section:Preliminaries}, we fix the notation and state the assumptions we will be working under. Section \ref{section:Convergence} discusses the convergence of the solution of the finite dimensional HJB equations to the L-viscosity solution of the infinite dimensional HJB equation. We also identify the limit as the value function of a control problem in a Hilbert space. This proves in particular that the finite dimensional value functions converge to the infinite dimensional value function. In Section \ref{section:C11_regularity}, we prove, under two different sets of assumptions, $C^{1,1}$ regularity in the spatial variable of the infinite dimensional value function $V$. Assuming that $DV$ is continuous, which in particular is satisfied in the previously discussed case of $C^{1,1}$ regularity, we prove in Section \ref{section:projection_C11} that the infinite dimensional value function projects precisely onto the finite dimensional value functions. In Section \ref{section:feedbacks}, we investigate the relationship between optimal controls for the finite dimensional control problems and optimal controls for the infinite dimensional control problem. Finally, in Section \ref{section:projection_no_C11}, we extend the projection property from Section \ref{section:projection_C11} to cases where $V$ may not be differentiable.

\section{Preliminaries}\label{section:Preliminaries}
\noindent
In this section, we are going to fix the notation, pose the finite and infinite dimensional control problems that we will be working with, and state the main assumptions.


\subsection{Notation}

Throughout the paper, we are going to use the following notation.

\begin{itemize}
	\item For $d,d^{\prime} \in \mathbb{N}$, let $(e_k)_{k=1,\dots, d}$ and $(e^{\prime}_m)_{m=1,\dots,d^{\prime}}$ denote the standard basis of $\mathbb{R}^d$ and $\mathbb{R}^{d^{\prime}}$, respectively. Let $|\,\cdot\,|$ denote the Euclidean norm. For $x,y\in \mathbb{R}^d$, let $x\cdot y$ denote the inner product in $\mathbb{R}^d$.
	\item For $x\in \mathbb{R}^d$, let $\delta_{x}$ denote the Dirac measure at the point $x$. For $n \in \mathbb{N}$, and $\mathbf{x}= (x_1,\dots,x_n)\in (\mathbb{R}^d)^n$, let $\mu_{\mathbf{x}} := \frac1n \sum_{i=1}^n \delta_{x_i}$, and $|\mathbf{x}|_r := \frac{1}{n^{1/r}} \left ( \sum_{i=1}^n |x_i|^r \right )^{1/r}$.
	\item Let $\Omega = (0,1)$ and for $r\geq 1$, let $L^r(\Omega;\mathbb{R}^d)$ denote the Lebesgue space with norm $\|X\|_{L^r(\Omega;\mathbb{R}^d)} := ( \int_{\Omega} |X(\omega)|^r \mathrm{d}\omega )^{1/r}$, $X\in L^r(\Omega;\mathbb{R}^d)$. Let $E:= L^2(\Omega;\mathbb{R}^d)$ denote the space of all square-integrable $\mathbb{R}^d$-valued random variables. The inner product and the norm will be denoted by $\langle \cdot, \cdot \rangle_E$ and $\|\,\cdot\,\|_E$, respectively. Let $\mathcal{L}^1$ denote the Lebesgue measure on $\mathbb{R}$, and for $X\in E$, let $X_{\texttt{\#}} \mathcal{L}^1$ denote the pushforward measure on $\mathbb{R}^d$.
	\item For a subset $B$ of some set, let $\mathbf{1}_B$ denote the characteristic function of $B$.
	\item For $n\in \mathbb{N}$, let $A_i^n = (\frac{i-1}{n},\frac{i}{n}) \subset (0,1) =\Omega$, $i=1,\dots,n$. For $\mathbf{x} = (x_1,\dots,x_n)\in (\mathbb{R}^d)^n$, let
	\begin{equation}\label{E_n}
		X^{\mathbf{x}}_n = \sum_{i=1}^n x_i \mathbf{1}_{A^n_i},
	\end{equation}
	and let $E_n$ denote the $nd$-dimensional subspace of $E$ consisting of random variables of the form \eqref{E_n}. Note that $(X^{\mathbf{x}}_n)_{\texttt{\#}} \mathcal{L}^1 = \mu_{\mathbf{x}}$.
	\item For $r\in [1,2]$, let $\mathcal{P}_r(\mathbb{R}^d)$ denote the set of all Borel probability measures on $\mathbb{R}^d$ with finite $r$-th moment, i.e., $\mathcal{M}_r(\mu):= \int_{\mathbb{R}^d} |x|^r \mu(\mathrm{d}x) < \infty$. Let $d_r:\mathcal{P}_r(\mathbb{R}^d) \times \mathcal{P}_r(\mathbb{R}^d) \to \mathbb{R}$ denote the Wasserstein distance, i.e.,
	\begin{equation}\label{definition_wasserstein_distance}
		\begin{split}
			&d_r(\mu,\beta)\\
			&= \inf_{\gamma\in \Gamma(\mu,\beta)} \left (\int_{\mathbb{R}^d\times \mathbb{R}^d} |x-y|^r \gamma(\mathrm{d}x,\mathrm{d}y) \right )^{\frac1r}\\
			&= \inf \left \{ \left ( \int_{\Omega} |X(\omega)-Y(\omega)|^r \mathrm{d}\omega \right )^{\frac1r}: \; X,Y\in L^r(\Omega;\mathbb{R}^d), \; X_{\texttt{\#}} \mathcal{L}^1 = \mu, \; Y_{\texttt{\#}} \mathcal{L}^1 = \beta \right \}.
		\end{split}
	\end{equation}
	Here, $\Gamma(\mu,\beta)$ is the set of all probability measures on $\mathbb{R}^d \times \mathbb{R}^d$ with first and second marginals $\mu$ and $\beta$, respectively. Note that $d^r_r(\mu,\delta_0) = \mathcal{M}_r(\mu)$. For more details on the Wasserstein space, see e.g. \cite{ambrosio_gigli_savare_2005,villani_2009}. The proof of the second equality in \eqref{definition_wasserstein_distance} can be found for instance in \cite[Theorem 3.9]{Gangbo-notes}.
	\item For $R>0$ and $\bar{r}\in [r,2]$, we denote $\mathfrak{M}^{\bar{r}}_{R} := \{ \mu\in \mathcal{P}_r(\mathbb{R}^d) : \mathcal{M}_{\bar r}(\mu) \leq R \}$.
\end{itemize}
Regarding mappings on Hilbert spaces and their derivatives, we are going to use the following notation.
\begin{itemize}
	\item For a matrix or a trace-class operator on a Hilbert space, we write $\text{Tr}(A)$ for the trace of $A$.
	\item For a Hilbert space $\mathcal{E}$ and $R>0$, $B_R(X)$ denotes the open ball of radius $R$ centered at $X\in \mathcal{E}$.
	\item For Hilbert spaces $\mathcal{E},\mathcal{E}_1,\mathcal{E}_2$, let $L(\mathcal{E}_1,\mathcal{E}_2)$ denote the space of bounded linear operators from $\mathcal{E}_1$ to $\mathcal{E}_2$, and let $L_2(\mathcal{E}_1,\mathcal{E}_2)$ denote the space of Hilbert--Schmidt operators from $\mathcal{E}_1$ to $\mathcal{E}_2$. For $k\in \mathbb{N}$, let $C^k(\mathcal{E})$ denote the set of all real-valued $k$-times continuously Fr\'echet differentiable functions on $\mathcal{E}$. Let $C^{\infty}(\mathcal{E})$ denote the set of all real-valued functions on $\mathcal{E}$ having continuous Fr\'echet derivatives of any order. Let $C^{1,1}(\mathcal{E}_1,\mathcal{E}_2)$ denote the set of all $\mathcal{E}_2$-valued Fr\'echet differentiable functions on $\mathcal{E}_1$ with Lipschitz continuous derivative. In the case $\mathcal{E}_2=\mathbb{R}$, we write $C^{1,1}(\mathcal{E}_1)$. For a function $\varphi:(0,T)\times \mathcal{E}\to\mathbb{R}$ we will write $\partial_t\varphi$ to denote the partial derivative of $\varphi$ with respect to the first variable, and $D\varphi$ and $D^2\varphi$ to denote the first and second order Fr\'echet derivatives of $\varphi$ with respect to the second variable, respectively. We denote by $C^{1,2}((0,T)\times \mathcal{E})$ the set of all real-valued functions on $(0,T)\times\mathcal{E}$ such that $\partial_t\varphi,D\varphi,D^2\varphi$ exist and are continuous.
	\item For $\varphi_n \in C^2((\mathbb{R}^d)^n)$, let $D_{x_i}$, $i=1,\dots,n$, denote the gradient with respect to $x_i\in\mathbb{R}^d$, and let $D_{x_ix_j}^2$ denote the second derivative with respect to $x_i\in \mathbb{R}^d$ and $x_j\in\mathbb{R}^d$.
	\item For $k\in\mathbb{N}$, let $C^k(\mathbb{R}^d\times\mathcal{P}_2(\mathbb{R}^d))$ denote the set of all real-valued functions on $\mathbb{R}^d\times \mathcal{P}_2(\mathbb{R}^d)$ having $k$ derivatives which are continuous. Here, in the second variable, we consider Lions derivatives, called $L$-derivatives, denoted by $\partial^k_{\mu} u$, where $u\in C^k(\mathbb{R}^d \times \mathcal{P}_2(\mathbb{R}^d))$. Note that $\partial_{\mu} u : \mathbb{R}^d\times \mathcal{P}_2(\mathbb{R}^d) \times \mathbb{R}^d \to \mathbb{R}^d$, and higher order derivatives are taken component-wise. For more details on $L$-derivatives, see e.g. \cite{carmona_delarue_2018}.
\end{itemize}

\subsection{The Control Problems}

Throughout the manuscript, $T>0$ denotes the fixed finite time horizon. We use the standard setting for stochastic control problems, which can be found in, e.g., \cite{fabbri_gozzi_swiech_2017,fleming_soner_2006,yong_zhou_1999}. A five-tuple $\tau := (\Omega^{\prime},\mathcal{F},\mathcal{F}^t_s, \mathbb{P},W)$ is a reference probability space, if
\begin{itemize}
	\item $(\Omega^{\prime},\mathcal{F},\mathbb{P})$ is a complete probability space;
	\item $W$ is a standard $\mathbb{R}^{d^{\prime}}$-valued Wiener process on $(\Omega^{\prime},\mathcal{F},\mathbb{P})$ such that $W(t)=0$, $\mathbb{P}$-a.s.;
	\item $\mathcal{F}^t_s$ is the filtration generated by $W$, augmented by all $\mathbb{P}$-null sets of $\mathcal{F}$.
\end{itemize}

\subsubsection{Finite Dimensional Problem}

Let $\mathbf{X}(s) = (X_1(s),\dots,X_n(s)) \in (\mathbb{R}^d)^n$ be the solution
\begin{equation}\label{state_equation}
	\begin{cases}
		\mathrm{d}X_i(s) = [-a_i(s) + b(X_i(s),\mu_{\mathbf{X}(s)})] \mathrm{d}s + \sigma(X_i(s),\mu_{\mathbf{X}(s)}) \mathrm{d}W(s)\\
		X_i(t) = x_i\in \mathbb{R}^d,
	\end{cases}
\end{equation}
$i=1,\dots,n$. Consider the cost functional
\begin{equation}\label{cost_functional}
	J_n(t,\mathbf{x};\mathbf{a}(\cdot)) := \mathbb{E} \left [ \int_t^T \left ( \frac1n \sum_{i=1}^n \left ( l_1(X_i(s),\mu_{\mathbf{X}(s)}) + l_2(a_i(s)) \right ) \right ) \mathrm{d}s + \mathcal{U}_T(\mu_{\mathbf{X}(T)}) \right ]
\end{equation}
for $(t,\mathbf{x}) \in [0,T] \times (\mathbb{R}^d)^n$. The value function $u_n$ is defined as
\begin{equation}\label{finite_dimensional_value_function}
	u_n(t,\mathbf{x}) = \inf_{\mathbf{a}(\cdot)\in \mathcal{A}^n_t} J_n(t,\mathbf{x};\mathbf{a}(\cdot)),
\end{equation}
where the class of admissible controls $\mathcal{A}^n_t$ consists of all processes $\mathbf{a}(\cdot)$ such that
\begin{itemize}
	\item there is a reference probability space $\tau := (\Omega^{\prime}, \mathcal{F}, \mathcal{F}^t_s,\mathbb{P},W)$ such that $\mathbf{a}(\cdot):[t,T]\times \Omega^{\prime} \to (\mathbb{R}^d)^n$ is $\mathcal{F}^t_s$-progressively measurable, and
	\item $\| \mathbf{a}(\cdot) \|_{M^2(t,T;(\mathbb{R}^d)^n)}^2 := \mathbb{E} [ \int_t^T |\mathbf{a}(s)|^2 \mathrm{d}s ] < \infty$.
\end{itemize}

\subsubsection{Infinite Dimensional Problem}\label{Section_Infinite_Dimensional_Problem}

Recall that $E=L^2(\Omega;\mathbb{R}^d)$, where $\Omega=(0,1)$. We now consider the infinite dimensional SDE\footnote{Throughout this manuscript, we consider strong solutions in the probabilistic sense, see e.g. \cite[Section 5.2]{karatzas_shreve_1991} for SDEs in finite dimensional spaces and \cite[Section 1.4.1]{fabbri_gozzi_swiech_2017} for SDEs in Hilbert spaces.}
\begin{equation}\label{Lifted_State_Equation}
	\begin{cases}
		\mathrm{d}X(s) = [-a(s) + B(X(s))] \mathrm{d}s + \Sigma(X(s)) \mathrm{d}W(s)\\
		X(t) = X \in E,
	\end{cases}
\end{equation}
where $B:E\to E$ is defined as $B(X)(\omega) := b(X(\omega),X_{\texttt{\#}}\mathcal{L}^1)$, and $\Sigma : E \to L_2(\mathbb{R}^{d^{\prime}} , E)$ is defined as $\Sigma(X)(\omega) := \sigma(X(\omega),X_{\texttt{\#}}\mathcal{L}^1)$. We consider the cost functional
\begin{equation}
	J(t,X;a(\cdot)) = \mathbb{E} \left [ \int_t^T \left ( L_1(X(s)) + L_2(a(s)) \right ) \mathrm{d}s + U_T(X(T)) \right ],
\end{equation}
where $L_1:E\to \mathbb{R}$ is given by
\begin{equation}
	L_1(X) := \int_{\Omega} l_1(X(\omega),X_{\texttt{\#}}\mathcal{L}^1) \mathrm{d}\omega
\end{equation}
and $L_2:E\to \mathbb{R}$ is given by
\begin{equation}
	L_2(X) := \int_{\Omega} l_2(X(\omega))\mathrm{d}\omega.
\end{equation}
Recall that $U_T:E\to \mathbb{R}$ is the lift of $\mathcal{U}_T: \mathcal{P}_2(\mathbb{R}^d) \to \mathbb{R}$, i.e., $U_T(X) := \mathcal{U}_T(X_{\texttt{\#}}\mathcal{L}^1)$.

In this case, the value function is defined as
\begin{equation}\label{Lifted_Value_Function}
	U(t,X) := \inf_{a(\cdot)\in \mathcal{A}_t} J(t,X,a(\cdot)),
\end{equation}
where the class of admissible controls $\mathcal{A}_t$ consists of all processes $a(\cdot)$ such that
\begin{itemize}
	\item there is a reference probability space $\tau := (\Omega^{\prime}, \mathcal{F}, \mathcal{F}^t_s,\mathbb{P},W)$ such that $a(\cdot):[t,T]\times \Omega^{\prime} \to E$ is $\mathcal{F}^t_s$-progressively measurable, and
	\item $\| a(\cdot) \|_{M^2(t,T;E)}^2 := \mathbb{E} [ \int_t^T \|a(s)\|_E^2 \mathrm{d}s ] < \infty$. 
\end{itemize}

\subsection{Assumptions}

Let us now state the assumptions which will be used in the paper.

\begin{assumption}\label{Assumption_b_sigma_lipschitz}
	Let $r \in [1,2)$.
	\begin{enumerate}[label=(\roman*)]
		\item Let $b:\mathbb{R}^d \times \mathcal{P}_2(\mathbb{R}^d) \to \mathbb{R}^d$ be Lipschitz continuous with respect to the $|\,\cdot\,|\times d_r(\cdot,\cdot)$-distance, i.e., there is a constant $C$ such that
		\begin{equation}\label{Lipschitz_b}
			|b(x,\mu) - b(y,\beta)| \leq C( |x-y| + d_r(\mu,\beta))
		\end{equation}
		for all $x,y\in \mathbb{R}^d$ and $\mu,\beta\in \mathcal{P}_2(\mathbb{R}^d)$. Moreover, let the function $(x,X) \mapsto \tilde{b}(x,X) := b(x,X_{\texttt{\#}}\mathcal{L}^1)$ be in $C^{1,1}(\mathbb{R}^d\times E,\mathbb{R}^d)$.
		\item Let $\sigma:\mathbb{R}^d \times \mathcal{P}_2(\mathbb{R}^d) \to \mathbb{R}^{d\times d^{\prime}}$ be Lipschitz continuous with respect to the $|\,\cdot\,|\times d_r(\cdot,\cdot)$-distance, i.e., there is a constant $C$ such that		\begin{equation}\label{Lipschitz_sigma}
			|\sigma(x,\mu) - \sigma(y,\beta)| \leq C( |x-y| + d_r(\mu,\beta)),
		\end{equation}
		for all $x,y\in \mathbb{R}^d$ and $\mu,\beta\in \mathcal{P}_2(\mathbb{R}^d)$. Moreover, let the function $(x,X) \mapsto \tilde{\sigma}(x,X) := \sigma(x,X_{\texttt{\#}}\mathcal{L}^1)$ be in $C^{1,1}(\mathbb{R}^d\times E,\mathbb{R}^{d\times d^{\prime}})$.
	\end{enumerate}
\end{assumption}

Note that the Lipschitz continuity of $b$ and $\sigma$ in equations \eqref{Lipschitz_b} and \eqref{Lipschitz_sigma}, implies the Lipschitz continuity of the lifted operators $B$ and $\Sigma$ defined in Section \ref{Section_Infinite_Dimensional_Problem}. Indeed, we have
\begin{equation}
	\begin{split}
		\|\Sigma(X) - \Sigma(Y) \|^2_{L_2(\mathbb{R}^{d^{\prime}},E)} &= \int_{\Omega} \left | \sigma(X(\omega),X_{\texttt{\#}}\mathcal{L}^1) - \sigma(Y(\omega),Y_{\texttt{\#}}\mathcal{L}^1) \right |_{\mathbb{R}^{d\times d^{\prime}}}^2 \mathrm{d}\omega\\
		&\leq C \int_{\Omega} \left ( |X(\omega)-Y(\omega)|^2 + d_r^2(X_{\texttt{\#}}\mathcal{L}^1,Y_{\texttt{\#}}\mathcal{L}^1) \right ) \mathrm{d}\omega\\
		& \leq C \|X-Y\|_E^2.
	\end{split}
\end{equation}
The proof for $B$ is similar. Thus, since the coefficient functions $B$ and $\Sigma$ are Lipschitz continuous in $E$, and $W$ is finite dimensional, it is standard that the state equation \eqref{Lifted_State_Equation} has a unique (strong) solution which has continuous trajectories.

\begin{assumption}\label{Assumption_running_terminal_cost}
	Let $r\in [1,2)$.
	\begin{enumerate}[label=(\roman*)]
		\item Let $l_1:\mathbb{R}^d \times \mathcal{P}_2(\mathbb{R}^d)\to \mathbb{R}$ be Lipschitz continuous with respect to the $|\,\cdot\,|\times d_r(\cdot,\cdot)$-distance, i.e., there is a constant $C$ such that
		\begin{equation}
			|l_1(x,\mu) - l_1(y,\beta)| \leq C(|x-y|+d_r(\mu,\beta))
		\end{equation}
		for all $x,y\in \mathbb{R}^d$ and $\mu,\beta\in \mathcal{P}_2(\mathbb{R}^d)$. Moreover, let $(x,X) \mapsto \tilde{l}_1(x,X) := l_1(x,X_{\texttt{\#}}\mathcal{L}^1)$ be in $C^{1,1}(\mathbb{R}^d\times E)$.
		\item Let there be constants $C_1, C_2,C_3>0$ such that $l_2\in C^{1,1}(\mathbb{R}^d)$ satisfies
		\begin{equation}\label{assumption_growth_l_2}
			-C_1 + C_2|p|^2 \leq l_2(p) \leq C_1 + C_3|p|^2
		\end{equation}
		for all $p\in \mathbb{R}^d$. Moreover, let the function
		\begin{equation}
			p \mapsto l_2(p) - \nu|p|^2
		\end{equation}
		be convex for some constant $\nu \geq 0$.
		\item Let $\mathcal{U}_T$ be Lipschitz continuous on $\mathcal{P}_r(\mathbb{R}^d)$, and be such that its lift $U_T$ belongs to $C^{1,1}(E)$.
	\end{enumerate}
\end{assumption}

\begin{assumption}\label{Assumption_Linear_State_Equation}
	\begin{enumerate}[label=(\roman*)]
		\item Let $b:\mathbb{R}^d \times \mathcal{P}_2(\mathbb{R}^d) \to \mathbb{R}^d$ and $\sigma:\mathbb{R}^d \times \mathcal{P}_2(\mathbb{R}^d) \to \mathbb{R}^{d\times d^{\prime}}$ be such that their lifts $B:E\to E$ and $\Sigma:E \to L_2(\mathbb{R}^{d^{\prime}},E)$ are affine linear and continuous.
		\item Let $l_1:\mathbb{R}^d \times \mathcal{P}_2(\mathbb{R}^d) \to \mathbb{R}$ and $\mathcal{U}_T :\mathcal{P}_2(\mathbb{R}^d) \to \mathbb{R}$ be such that $\tilde l_1:\mathbb{R}^d \times E \to \mathbb{R}$ and $U_T: E \to \mathbb{R}$ are convex.
	\end{enumerate}
\end{assumption}

We remark that since $l_2\in C^{1,1}(\mathbb{R}^d)$, we have
\begin{equation}
	|l_2(p)-l_2(q)|\leq C(1+|p|+|q|)|p-q|\quad\text{for all}\,\,p,q\in \mathbb{R}^d.
\end{equation}

\begin{remark}
	Let $g:\mathcal{P}_2(\mathbb{R}^d) \to \mathbb{R}$ and let $G:E \to \mathbb{R}, G(X):=g(X_{\texttt{\#}} \mathcal{L}^1)$ be its lift. We recall that if $G$ is continuously Fr\'echet differentiable (that is $g$ is continuously $L$-differentiable) then we have
	\[
	DG(X)(\omega)=\partial_\mu g(X_{\texttt{\#}} \mathcal{L}^1)(X(\omega)).
	\]
	Moreover, if $DG$ is globally Lipschitz, then $\partial_\mu g$ can be modified so that for every $X$, $\partial_\mu g(X_{\texttt{\#}} \mathcal{L}^1)(\cdot)$ is Lipschitz with the same Lipschitz constant. We refer to \cite[Sections 5.2, 5.3]{carmona_delarue_2018} for more on this and $L$-differentiability. In particular, the $C^{1,1}$ regularity of $G$ can be rephrased in terms of the properties of $\partial_\mu g$. For instance, if we assume that $\partial_\mu g$ is Lipschitz continuous with respect to the $d_2(\cdot,\cdot)\times |\cdot|$-metric, then
	\begin{equation}
		\begin{split}
			\| DG(X) - DG(Y) \|^2_E &= \| \partial_{\mu} g(X_{\texttt{\#}} \mathcal{L}^1)(X) - \partial_{\mu}g(Y_{\texttt{\#}} \mathcal{L}^1)(Y) \|^2_E\\
			&= \int_{\Omega} | \partial_{\mu}g(X_{\texttt{\#}} \mathcal{L}^1)(X(\omega)) - \partial_{\mu}g(Y_{\texttt{\#}} \mathcal{L}^1)(Y(\omega)) |^2 \mathrm{d}\omega\\
			&\leq C \int_{\Omega} \left ( d^2_2(X_{\texttt{\#}} \mathcal{L}^1,Y_{\texttt{\#}} \mathcal{L}^1) + | X(\omega) - Y(\omega) |^2 \right ) \mathrm{d}\omega\\
			&\leq C \|X-Y \|_E^2.
		\end{split}
	\end{equation}
	Regarding the convexity of $G$, we note that the $L$-convexity of $g$ in the sense of \cite[Definition 5.70]{carmona_delarue_2018} implies convexity of its lift $G$. Moreover, by \cite[Proposition 5.79]{carmona_delarue_2018}, if $g$ is continuously $L$-differentiable, $g$ is $L$-convex if and only if $g$ is displacement convex. We refer to \cite[Section 5.5]{carmona_delarue_2018} for more.
\end{remark}

In the following, we discuss some more specific examples of functions $\sigma$ that satisfy the assumptions above. For possible extensions of these examples as well as examples of the remaining coefficients, see \cite[Examples 2.4 and 2.5]{mayorga_swiech_2023}.

\begin{example}
	\begin{enumerate}[label=(\roman*)]
		\item 
		Let $\sigma: \mathbb{R}^d\times \mathcal{P}_2(\mathbb{R}^d) \to \mathbb{R}^{d\times d^{\prime}}$ be given by $\sigma(x,\mu) = \sigma^1(x) + \sigma^2(\mu)$, where $\sigma^1:\mathbb{R}^d \to \mathbb{R}^{d\times d^{\prime}}$ and $\sigma^2:\mathcal{P}_2(\mathbb{R}^d) \to \mathbb{R}^{d\times d^{\prime}}$. Now, assume that each entry of $\sigma^1$ is $C^{1,1}$ with bounded derivative. Moreover, let each entry $\sigma^2_{km}$, $k=1,\dots,d$, $m=1,\dots,d^{\prime}$ of $\sigma^2$ be of the following type. Let there be functions $g_{k,m}:\mathbb{R} \to \mathbb{R}$ which are $C^{1,1}$ with bounded derivatives, $\zeta_{k,m}:\mathbb{R}^d \to \mathbb{R}$ which are $C^{1,1}$ with bounded derivatives, and let $\sigma_{km}^2(\mu) := g_{k,m} \left ( \int_{\mathbb{R}^d} \zeta_{k,m}(x) \mu(\mathrm{d}x) \right )$, $\mu\in \mathcal{P}_r(\mathbb{R}^d)$. Then, the lift is given by $\tilde{\sigma}_{km}^2(X) = g_{k,m} \left ( \int_{\Omega} \zeta_{k,m} \circ X \right )$, $X\in E$. Moreover, $D\zeta_{k,m}(X) \in E$ and the Frech\'et derivative is $D\tilde{\sigma}_{km}^2(X) = g^{\prime}_{k,m} \left ( \int_{\Omega} \zeta_{k,m} \circ X \right ) D\zeta_{k,m}(X)$, so $\tilde{\sigma}_{k,m}^2$ is $C^{1,1}(E)$. The Lipschitz continuity on $\mathcal{P}_r(\mathbb{R}^d)$ follows as in \cite[Example 2.4]{mayorga_swiech_2023}. 
		
		\item For the affine linear case, let $\sigma:\mathbb{R}^d \times \mathcal{P}_2(\mathbb{R}^d) \to \mathbb{R}^{d\times d^{\prime}}$ be given by $\sigma(x,\mu) := \sigma^1 x + \int_{\mathbb{R}^d} \sigma^2y \mu(\mathrm{d}y) + \sigma^3$, where $\sigma^1,\sigma^2\in L(\mathbb{R}^d, \mathbb{R}^{d\times d^{\prime}})$, $\sigma^3\in \mathbb{R}^{d\times d^{\prime}}$. In this case, the lift $\Sigma : E \to L_2(\mathbb{R}^{d^{\prime}}, E)$ is given by $\Sigma(X) := \sigma^1 X + \int_{\Omega} \sigma^2 X + \sigma^3$, which is linear and continuous.
		
	\end{enumerate}
\end{example}

\begin{remark}
	\begin{enumerate}[label=(\roman*)]
		\item Throughout the paper, we consider functions $u:\mathbb{R}^d \times \mathcal{P}_2(\mathbb{R}^d)\to \mathbb{R}^d$. We are going to write $\tilde u: \mathbb{R}^d \times E \to \mathbb{R}^d$ to denote the lift of $u$, i.e., $\tilde u(x,X) := u(x,X_{\texttt{\#}}\mathcal{L}^1)$, $(x,X)\in \mathbb{R}^d \times E$. Moreover, we are going to use capital letters to denote the Nemytskii operator associated with $u$, i.e., $U:E\to E$, $U(X)(\omega) := u(X(\omega),X_{\texttt{\#}}\mathcal{L}^1)$, $X\in E$, $\omega \in \Omega$. For functions $l: \mathbb{R}^d \times \mathcal{P}_2(\mathbb{R}^d) \to \mathbb{R}$, we are also going to write $\tilde{l}:\mathbb{R}^d \times E \to \mathbb{R}$ to denote the lift. However, in this case, the function denoted by the capital letter is defined by $L:E\to \mathbb{R}$, $L(X) := \int_{\Omega} l(X(\omega),X_{\texttt{\#}}\mathcal{L}^1) \mathrm{d} \omega$.
		\item We denote by $l_2^{\ast}(p) := \sup_{q\in \mathbb{R}^d} \{ q \cdot p - l_2(q) \}$ the convex conjugate of $l_2$.
		\item In order to link the finite dimensional control problem with the infinite dimensional control problem, we need the following identity regarding the Hamiltonian of the lifted HJB equation:
		\begin{equation}\label{eq:equalitycosts}
			\int_{\Omega} \sup_{q\in \mathbb{R}^d} (q\cdot P(\omega) - l_2(q)) \mathrm{d}\omega = \sup_{a\in E} \left \{ \langle a,P \rangle_E - \int_{\Omega} l_2(a(\omega)) \mathrm{d}\omega \right \}.
		\end{equation}
		This identity follows from the measurable selection theorem \cite[Theorem 18.19]{aliprantis_border_2006} and equation \eqref{assumption_growth_l_2}. Alternatively, this can be proved directly by considering first the uniformly convex perturbed costs $l_2^{\varepsilon}(a) := l_2(a) + \varepsilon |a|^2$, $\varepsilon>0$, for which the maximizer in the supremum is unique and is equal to $q(\omega)=(Dl_2^\varepsilon)^{-1}(P(\omega))$, which belongs to $E$ and hence the equality holds. One then easily shows that \eqref{eq:equalitycosts} follows in the limit as $\varepsilon \to 0$ since the maximizers are bounded in $E$.
	\end{enumerate}
\end{remark}

\section{Convergence of \texorpdfstring{$u_n$}{un} to \texorpdfstring{$V$}{V}}\label{section:Convergence}

\subsection{Estimates for the Finite Dimensional Problem}

First, we prove the several estimates for the solution of the state equation.
\begin{proposition}\label{Prop_Finite_A_Priori}
	Let Assumption \ref{Assumption_b_sigma_lipschitz} be satisfied. Let $\mathbf{X},\mathbf{X}^0$, and $\mathbf{X}^1$ be the solutions of equation \eqref{state_equation} with initial conditions $\mathbf{x},\mathbf{x}^0$, and $\mathbf{x}^1$, respectively, and control $\mathbf{a}(\cdot)\in \mathcal{A}^n_t$. Then, the following estimates hold:
	\begin{enumerate}[label=(\roman*)]
		\item There is a constant $C$, independent of $n\in\mathbb{N}$, such that
		\begin{equation}\label{a_priori_estimate_for_state_equation}
			\mathbb{E} \left [ \sup_{s\in [t,T]} | \mathbf{X}(s) |_r \right ] \leq C \left ( 1 + | \mathbf{x}|_r + \frac{1}{\sqrt{n}} \| \mathbf{a}(\cdot) \|_{M^2(t,T;(\mathbb{R}^d)^n)} \right )
		\end{equation}
		for all $\mathbf{x}\in (\mathbb{R}^d)^n$, and $\mathbf{a}(\cdot)\in \mathcal{A}^n_t$.
		\item There is a constant $C$, independent of $n\in \mathbb{N}$, such that
		\begin{equation}\label{time_continuity_for_state_equation}
			\mathbb{E} \left [ \sup_{s^{\prime} \in [t,s]} |\mathbf{X}(s^{\prime}) - \mathbf{x} |_r \right ] \leq \frac{C}{\sqrt{n}} \mathbb{E} \left [ \int_t^s |\mathbf{a}(s^{\prime})|^2 \mathrm{d}s^{\prime} \right ]^{\frac12} + C \left (1+ |\mathbf{x}|_r \right )(s-t)^{\frac12}
		\end{equation}
		for all $s\in [t,T]$, $\mathbf{x}\in (\mathbb{R}^d)^n$, and $\mathbf{a}(\cdot)\in \mathcal{A}^n_t$.
		\item There is a constant $C$, independent of $n\in \mathbb{N}$, such that
		\begin{equation}
			\mathbb{E} \left [ \sup_{s\in[t,T]} | \mathbf{X}^1(s) - \mathbf{X}^0(s)|_r \right ]\leq C| \mathbf{x}^1 - \mathbf{x}^0|_r
		\end{equation}
		for all $\mathbf{x}^1, \mathbf{x}^0 \in (\mathbb{R}^d)^n$, and $\mathbf{a}(\cdot)\in \mathcal{A}^n_t$.
	\end{enumerate}
\end{proposition}

\begin{proof}
	$(i)$ From the state equation \eqref{state_equation}, it follows that for every $s^{\prime}\in [t,s]$, we have
	\begin{equation}\label{Estimate_State_Equation_01}
		\begin{split}
			&|X_i(s^{\prime})|^r\\
			&\leq C \Bigg ( |x_i|^r + \int_t^s |a_i(t^{\prime})|^r \mathrm{d}t^{\prime} + \int_t^s | b(X_i(t^{\prime}),\mu_{\mathbf{X}(t^{\prime})}) |^r \mathrm{d}t^{\prime}\\
			&\qquad\qquad\qquad\qquad\qquad\qquad\qquad\qquad + \left | \int_t^{s^{\prime}} \sigma(X_i(t^{\prime}),\mu_{\mathbf{X}(t^{\prime})}) \mathrm{d}W(t^{\prime}) \right |^r \Bigg ).
		\end{split}
	\end{equation}
	Due to the linear growth of $b$, we obtain for the second integral
	\begin{equation}
		\mathbb{E} \left [ \int_t^{s} | b(X_i(t^{\prime}),\mu_{\mathbf{X}(t^{\prime})}) |^r \mathrm{d}t^{\prime} \right ] \leq C \mathbb{E} \left [ \int_t^s (1+ |X_i(t^{\prime})|^r + \mathcal{M}_r(\mu_{\mathbf{X}(t^{\prime})}) ) \mathrm{d}t^{\prime} \right ].
	\end{equation}
	Moreover, by Burkholder--Davis--Gundy inequality and due to the linear growth of $\sigma$, we obtain for the third integral
	\begin{equation}
		\begin{split}
			&\mathbb{E} \left [ \sup_{s^{\prime}\in [t,s]} \left | \int_t^{s^{\prime}} \sigma(X_i(t^{\prime}),\mu_{\mathbf{X}(t^{\prime})}) \mathrm{d}W(t^{\prime}) \right |^r \right ]\\
			&\leq C \mathbb{E} \left [ \left ( \int_t^s | \sigma(X_i(t^{\prime}),\mu_{\mathbf{X}(t^{\prime})}) |^2 \mathrm{d}t^{\prime} \right )^{\frac{r}{2}} \right ]\\
			&\leq C \mathbb{E} \left [ \left ( \int_t^s (1+ |X_i(t^{\prime})|^2 + \mathcal{M}^{\frac{2}{r}}_r(\mu_{\mathbf{X}(t^{\prime})}) ) \mathrm{d}t^{\prime} \right )^{\frac{r}{2}} \right ]\\
			&\leq \frac{1}{4} \mathbb{E} \left [ \sup_{t^{\prime}\in [t,s]} |X_i(t^{\prime})|^r + \sup_{t^{\prime}\in [t,s]} \mathcal{M}_r(\mu_{\mathbf{X}(t^{\prime})}) \right ] \\
			&\quad + C \mathbb{E} \left [ 1 + \int_t^s | X_i(t^{\prime})|^r \mathrm{d}t^{\prime} + \int_t^s \mathcal{M}_r(\mu_{\mathbf{X}(t^{\prime})}) \mathrm{d}t^{\prime} \right ].
		\end{split}
	\end{equation}
	Now, we note that
	\begin{equation}\label{estimate_M_r}
		\mathbb{E} \left [ \sup_{t^{\prime}\in [t,s]} \mathcal{M}_r(\mu_{\mathbf{X}(t^{\prime})}) \right ] = \mathbb{E} \left [ \sup_{t^{\prime}\in [t,s]} \left ( \frac{1}{n} \sum_{i=1}^n |X_i(t^{\prime})|^r \right ) \right ] \leq \frac{1}{n} \sum_{i=1}^n \mathbb{E} \left [ \sup_{t^{\prime}\in [t,s]} |X_i(t^{\prime})|^r \right ]
	\end{equation}
	as well as
	\begin{equation}
		\mathbb{E} \left [ \int_t^s \mathcal{M}_r(\mu_{\mathbf{X}(t^{\prime})}) \mathrm{d}t^{\prime} \right ] \leq \int_t^s \frac{1}{n} \sum_{i=1}^n \mathbb{E} \left [ \sup_{t^{\prime}\in [t,s^{\prime}]} | X_i(t^{\prime})|^r \right ] \mathrm{d}s^{\prime}.
	\end{equation}
	Moreover, we have for some constant $C$, independent of $n$,
	\begin{equation}\label{control_estimtate}
		\frac1n \sum_{i=1}^n \mathbb{E} \left [ \int_t^s |a_i(t^{\prime})|^r \mathrm{d}t^{\prime} \right ] \leq \frac{C}{n^{r/2}} \mathbb{E} \left [ \int_t^s |\mathbf{a}(t^{\prime})|^2 \mathrm{d}t^{\prime} \right ]^{\frac{r}{2}}.
	\end{equation}
	Thus, taking the supremum over $[t,s]$ in equation \eqref{Estimate_State_Equation_01}, the expectation, the sum over $i=1,\dots,n$, and then dividing both sides by $n$, we obtain
	\begin{equation}
		\begin{split}
			& \frac1n \sum_{i=1}^n \mathbb{E} \left [ \sup_{s^{\prime}\in [t,s]} |X_i(s^{\prime})|^r \right ]
			\leq \frac{1}{2n} \sum_{i=1}^n \mathbb{E} \left [ \sup_{s^{\prime}\in [t,s]} |X_i(s^{\prime})|^r \right ] 
			\\
			&\qquad\qquad
			+ C \left ( 1 + |\mathbf{x}|_r^r + \frac{1}{n^{r/2}} \mathbb{E} \left [ \int_t^s |\mathbf{a}(s^{\prime})|^2 \mathrm{d}s^{\prime} \right ]^{\frac{r}{2}} + \int_t^s \frac{1}{n} \sum_{i=1}^n \mathbb{E} \left [ \sup_{t^{\prime}\in [t,s^{\prime}]} |X_i(t^{\prime})|^r \right ] \mathrm{d}s^{\prime} \right ).
		\end{split}
	\end{equation}
	Applying Gr\"onwall's inequality, we obtain
	\begin{equation}\label{Gronwall}
		\frac1n \sum_{i=1}^n \mathbb{E} \left [ \sup_{s \in [t,T]} |X_i(s)|^r \right ] \leq C \left ( 1 + |\mathbf{x}|_r^r + \frac{1}{n^{r/2}} \mathbb{E} \left [ \int_t^T |\mathbf{a}(s)|^2 \mathrm{d}s \right ]^{\frac{r}{2}} \right ).
	\end{equation}
	Noting that
	\begin{equation}
		\mathbb{E} \left [ \sup_{s\in [t,T]} | \mathbf{X}(s)|_r \right ] \leq \left ( \frac{1}{n} \sum_{i=1}^n \mathbb{E} \left [ \sup_{s\in [t,T]} |X_i(s)|^r \right ] \right )^{\frac{1}{r}}
	\end{equation}
	concludes the proof of part $(i)$.
	
	$(ii)$ First, we observe that for every $s^{\prime}\in [t,s]$, we have
	\begin{equation}\label{Estimate_State_Equation_02}
		\begin{split}
			&|X_i(s^{\prime})-x_i|^r \\
			&\leq C \left ( \int_t^s |a_i(t^{\prime})|^r \mathrm{d}t^{\prime} + \int_t^s | b(X_i(t^{\prime}),\mu_{\mathbf{X}(t^{\prime})}) |^r \mathrm{d}t^{\prime} + \left | \int_t^{s^{\prime}} \sigma(X_i(t^{\prime}),\mu_{\mathbf{X}(t^{\prime})}) \mathrm{d}W(t^{\prime}) \right |^r \right ).
		\end{split}
	\end{equation}
	For the second integral, due to the linear growth of $b$, we have
	\begin{equation}
		\mathbb{E} \left [ \int_t^s | b(X_i(t^{\prime}),\mu_{\mathbf{X}(t^{\prime})}) |^r \mathrm{d}t^{\prime} \right ]
		\leq C \left ( 1 + \mathbb{E} \left [ \sup_{t^{\prime}\in [t,s]} |X_i(t^{\prime})|^r + \sup_{t^{\prime}\in [t,s]} \mathcal{M}_r(\mu_{\mathbf{X}(t^{\prime})}) \right ] \right ) (s-t).
	\end{equation}
	For the stochastic integral, we obtain using Burkholder--Davis--Gundy inequality and the linear growth of $\sigma$,
	\begin{equation}
		\begin{split}
			&\mathbb{E} \left [ \sup_{s^{\prime}\in [t,s]} \left | \int_t^{s^{\prime}} \sigma(X_i(t^{\prime}),\mu_{\mathbf{X}(t^{\prime})}) \mathrm{d}W(t^{\prime}) \right |^r \right ]\\
			&\leq C \left ( 1 + \mathbb{E} \left [ \sup_{s^{\prime}\in [t,s]} |X_i(s^{\prime})|^r + \sup_{s^{\prime}\in [t,s]}\mathcal{M}_r(\mu_{\mathbf{X}(s^{\prime})}) \right ]  \right ) (s-t)^{\frac{r}{2}}.
		\end{split}
	\end{equation}
	Taking the supremum over $[t,s]$ in equation \eqref{Estimate_State_Equation_02}, the expectation, the sum over $i=1,\dots,n$, and then dividing both sides by $n$, we obtain using \eqref{estimate_M_r} and \eqref{control_estimtate}
	\begin{equation}
		\begin{split}
			&\frac1n \sum_{i=1}^n \mathbb{E} \left [ \sup_{s^{\prime}\in [t,s]} |X_i(s^{\prime})-x_i|^r \right ]\\
			&\leq \frac{C}{n^{r/2}} \mathbb{E} \left [ \int_t^s |\mathbf{a}(s^{\prime})|^2 \mathrm{d}s^{\prime} \right ]^{\frac{r}{2}} + C \left ( 1 + \frac{1}{n} \sum_{i=1}^n \mathbb{E} \left [ \sup_{s^{\prime}\in [t,s]} |X_i(s^{\prime})|^r \right ] \right ) (s-t)^{\frac{r}{2}}.
		\end{split}
	\end{equation}
	Applying inequality \eqref{Gronwall} and noting that
	\begin{equation}
		\mathbb{E} \left [ \sup_{s^{\prime} \in [t,s]} |\mathbf{X}(s^{\prime}) - \mathbf{x} |_r \right ] \leq \left ( \frac1n \sum_{i=1}^n \mathbb{E} \left [ \sup_{s^{\prime}\in [t,s]} |X_i(s^{\prime})-x_i|^r \right ] \right )^{\frac{1}{r}}
	\end{equation}
	completes the proof.
	
	$(iii)$ First note that for every $s^{\prime} \in [t,s]$,
	\begin{equation}\label{partiii}
		\begin{split}
			|X^1_i(s^{\prime}) - X^0_i(s^{\prime})|^r&\leq C |x^1_i - x^0_i|^r + C \int_t^{s} | b(X^1_i(t^{\prime}),\mu_{\mathbf{X}^1(t^{\prime})}) - b(X^0_i(t^{\prime}),\mu_{\mathbf{X}^0(t^{\prime})}) |^r \mathrm{d}t^{\prime}\\
			&\quad + C \left | \int_t^{s^{\prime}} ( \sigma(X^1_i(t^{\prime}),\mu_{\mathbf{X}^1(t^{\prime})}) - \sigma(X^0_i(t^{\prime}),\mu_{\mathbf{X}^0(t^{\prime})})) \mathrm{d}W(t^{\prime}) \right |^r.
		\end{split}
	\end{equation}
	Due to the Lipschitz continuity of $b$, we have
	\begin{equation}
		\begin{split}
			&\int_t^s | b(X^1_i(t^{\prime}),\mu_{\mathbf{X}^1(t^{\prime})}) - b(X^0_i(t^{\prime}),\mu_{\mathbf{X}^0(t^{\prime})}) |^r \mathrm{d}t^{\prime}\\
			&\leq C \int_t^s \left ( | X^1_i(t^{\prime}) - X^0_i(t^{\prime})|^r + d^r_r(\mu_{\mathbf{X}^1(t^{\prime})},\mu_{\mathbf{X}^0(t^{\prime})} ) \right ) \mathrm{d}t^{\prime} \\
			&\leq C \int_t^s \left ( | X^1_i(t^{\prime}) - X^0_i(t^{\prime})|^r + \frac{1}{n} \sum_{i=1}^n |X^1_i(t^{\prime}) - X^0_i(t^{\prime})|^r \right ) \mathrm{d}t^{\prime}.
		\end{split}
	\end{equation}
	Furthermore, due to the Lipschitz continuity of $\sigma$, we have
	\begin{equation}
		\begin{split}
			&\mathbb{E} \left [ \sup_{s^{\prime}\in [t,s]} \left | \int_t^{s^{\prime}} ( \sigma(X^1_i(t^{\prime}),\mu_{\mathbf{X}^1(t^{\prime})}) - \sigma(X^0_i(t^{\prime}),\mu_{\mathbf{X}^0(t^{\prime})})) \mathrm{d}W(t^{\prime}) \right |^r \right ]\\
			&\leq C \mathbb{E} \left [ \left ( \int_t^{s} | \sigma(X^1_i(t^{\prime}),\mu_{\mathbf{X}^1(t^{\prime})}) - \sigma(X^0_i(t^{\prime}),\mu_{\mathbf{X}^0(t^{\prime})}) |^2 \mathrm{d}t^{\prime} \right )^{\frac{r}{2}} \right ]\\
			&\leq C \mathbb{E} \left [ \left ( \int_t^{s} \left ( | X^1_i(t^{\prime}) - X^0_i(t^{\prime})|^2 + d_r^2(\mu_{\mathbf{X}^1(t^{\prime})},\mu_{\mathbf{X}^0(t^{\prime})} ) \right ) \mathrm{d}t^{\prime} \right )^{\frac{r}{2}} \right ]\\
			&\leq \frac14 \mathbb{E} \left [ \sup_{t^{\prime} \in [t,s]} |X^1_i(t^{\prime}) - X^0_i(t^{\prime})|^r + \sup_{t^{\prime}\in [t,s]} d_r^r (\mu_{\mathbf{X}^1(t^{\prime})},\mu_{\mathbf{X}^0(t^{\prime})}) \right ]\\
			&\quad + C \mathbb{E} \left [ \int_t^s |X_i^1(t^{\prime}) - X_i^0(t^{\prime})|^r \mathrm{d}t^{\prime} + \int_t^s d_r^r (\mu_{\mathbf{X}^1(t^{\prime})},\mu_{\mathbf{X}^0(t^{\prime})}) \mathrm{d}t^{\prime} \right ].
		\end{split}
	\end{equation}
	Now, we note that
	\begin{equation}
		\mathbb{E} \left [ \sup_{t^{\prime}\in [t,s]} d_r^r (\mu_{\mathbf{X}^1(t^{\prime})},\mu_{\mathbf{X}^0(t^{\prime})}) \right ] \leq \frac{1}{n} \sum_{i=1}^n \mathbb{E} \left [ \sup_{t^{\prime}\in [t,s]} | X_i^1(t^{\prime}) - X_i^0(t^{\prime})|^r \right ]
	\end{equation}
	as well as
	\begin{equation}
		\mathbb{E} \left [ \int_t^s d_r^r (\mu_{\mathbf{X}^1(t^{\prime})},\mu_{\mathbf{X}^0(t^{\prime})}) \mathrm{d}t^{\prime} \right ] \leq \frac{1}{n} \sum_{i=1}^n \mathbb{E} \left [ \int_t^s | X_i^1(t^{\prime}) - X_i^0(t^{\prime})|^r \mathrm{d}t^{\prime} \right ].
	\end{equation}
	Thus, taking the supremum over $[t,s]$ in equation \eqref{partiii}, the expectation, the sum over $i=1,\dots,n$ and then dividing both sides by $n$, we obtain
	\begin{equation}
		\begin{split}
			&\frac{1}{n} \sum_{i=1}^n \mathbb{E} \left [ \sup_{s^{\prime} \in [t,s]} |X_i^1(s^{\prime}) - X^0_i(s^{\prime})|^r \right ]\\
			&\leq C \left ( |\mathbf{x}^1 - \mathbf{x}^0 |_r^r + \int_t^s \frac{1}{n} \sum_{i=1}^n \mathbb{E} \left [ \sup_{t^{\prime}\in [t,s^{\prime}]} |X_i^1(t^{\prime}) - X^0_i(t^{\prime})|^r \right ] \mathrm{d}s^{\prime} \right ).
		\end{split}
	\end{equation}
	Applying Gr\"onwall's inequality and noting that
	\begin{equation}
		\mathbb{E} \left [ \sup_{s\in [t,T]} | \mathbf{X}^1(s) - \mathbf{X}^0(s) |_r \right ] \leq \left ( \frac{1}{n} \sum_{i=1}^n \mathbb{E} \left [ \sup_{s\in [t,T]} |X_i^1(s) - X^0_i(s)|^r \right ] \right )^{\frac{1}{r}}
	\end{equation}
	yields the claim.
\end{proof}

We now prove the Lipschitz continuity of the value functions in the spatial variable.

\begin{proposition}\label{Lipschitz_Continuity_u_n}
	Let Assumptions \ref{Assumption_b_sigma_lipschitz} and \ref{Assumption_running_terminal_cost} be satisfied. Then, there is a constant $C>0$, independent of $n\in \mathbb{N}$, such that
	\begin{equation}
		|u_n(t,\mathbf{x})| \leq C (1+ |\mathbf{x}|_r ),
	\end{equation}
	and
	\begin{equation}
		|u_n(t,\mathbf{x}) - u_n(t,\mathbf{y})| \leq C | \mathbf{x} - \mathbf{y} |_r
	\end{equation}
	for all $t\in [0,T]$ and $\mathbf{x},\mathbf{y}\in (\mathbb{R}^d)^n$.
\end{proposition}

\begin{proof}
	The first estimate follows easily from the growth assumptions on $l_1,l_2,\mathcal{U}_T$ and Proposition \ref{Prop_Finite_A_Priori}(i). The second estimate follows from the assumptions about the Lipschitz continuity of $l_1,\mathcal{U}_T$ and Proposition \ref{Prop_Finite_A_Priori}(iii).
\end{proof}

Finally, let us prove an estimate regarding the uniform continuity of the value functions $u_n$ in the time variable.

\begin{proposition}\label{Lipschitz_Time_u_n}
	Let Assumptions \ref{Assumption_b_sigma_lipschitz} and \ref{Assumption_running_terminal_cost} be satisfied. Then, there is a constant $C>0$, independent of $n\in \mathbb{N}$, such that
	\begin{equation}\label{Lipschitz_Time_u_n_first_estimate}
		|u_n(s,\mathbf{x}) - u_n(t,\mathbf{x})| \leq C(1+|\mathbf{x}|_r) |s-t|^{\frac12}
	\end{equation}
	for all $s,t\in [0,T]$ and $\mathbf{x}\in (\mathbb{R}^d)^n$.
\end{proposition}

\begin{proof}
	We first note that arguing exactly as in the proof of \cite[Proposition 3.1]{mayorga_swiech_2023} as well as the discussion following it, we obtain that the functions $u_n$ are continuous and for some sufficiently large constant $K$, which is independent of $n$,
	\begin{equation}
		u_n(t,\mathbf{x}) = \inf_{\mathbf{a}(\cdot) \in \mathcal{A}_t^{n,K\sqrt{n}}} \mathbb{E} \left [ \int_t^T \left ( \frac{1}{n} \sum_{i=1}^n l_1(X_i(s^{\prime}),\mu_{\mathbf{X}(s^{\prime})}) + l_2(a_i(s^{\prime})) \right ) \mathrm{d}s^{\prime} + \mathcal{U}_T(\mu_{\mathbf{X}(T)}) \right ],
	\end{equation}
	where $\mathcal{A}^{n,K\sqrt{n}}_t = \{ \mathbf{a}(\cdot) \in \mathcal{A}_t^n : \mathbf{a}(\cdot) \text{ has values in }B_{K\sqrt{n}}(0) \}$.
	
	Without loss of generality, let $t<s$. By the dynamic programming principle, we have
	\begin{equation}
		\begin{split}
			&u_n(t,\mathbf{x})\\
			&= \inf_{\mathbf{a}(\cdot)\in \mathcal{A}^{n,K\sqrt{n}}_t} \mathbb{E} \left [ \int_t^s \left ( \frac{1}{n} \sum_{i=1}^n (l_1(X_i(s^{\prime}),\mu_{\mathbf{X}(s^{\prime})}) + l_2(a_i(s^{\prime}))) \right ) \mathrm{d}s^{\prime} + u_n(s,\mathbf{X}(s)) \right ].
		\end{split}
	\end{equation}
	Thus,
	\begin{equation}
		\begin{split}
			&|u_n(t,\mathbf{x}) - u_n(s,\mathbf{x})|\\
			&\leq \sup_{a(\cdot)\in \mathcal{A}^{n,K\sqrt{n}}_t} \mathbb{E} \Bigg [ \int_t^s \left ( \frac{1}{n} \sum_{i=1}^n \left ( \left | l_1(X_i(s^{\prime}),\mu_{\mathbf{X}(s^{\prime})}) \right | + \left | l_2(a_i(s^{\prime})) \right | \right ) \right ) \mathrm{d}s^{\prime} \\
			&\qquad\qquad\qquad\qquad\qquad\qquad\qquad\qquad\qquad\qquad + \left |u_n(s,\mathbf{X}(s)) - u_n(s,\mathbf{x}) \right | \Bigg ].
		\end{split}
	\end{equation}
	Due to the Lipschitz continuity of $l_1$, we have
	\begin{equation}
		\frac{1}{n} \sum_{i=1}^n \left | l_1(X_i(s^{\prime}),\mu_{\mathbf{X}(s^{\prime})}) \right | \leq C \left (1+|\mathbf{X}(s^{\prime})|_1 + \mathcal{M}^{\frac1r}_r(\mu_{\mathbf{X}(s^{\prime})}) \right )\leq C (1 + | \mathbf{X}(s^{\prime}) |_r ).
	\end{equation}
	Furthermore, since $\mathbf{a}(\cdot) \in \mathcal{A}^{n,K\sqrt{n}}_t$, using \eqref{assumption_growth_l_2} we obtain for some constant $C>0$
	\begin{equation}
		\frac{1}{n} \sum_{i=1}^n  \left | l_2(a_i(s^{\prime})) \right | \leq \frac{1}{n} \sum_{i=1}^n \left ( C + C|a_i(s^{\prime})|^2 \right ) \leq C \left ( 1 + \frac{1}{n} | \mathbf{a}(s^{\prime}) |^2 \right ) \leq C(1+K^2).
	\end{equation}
	By Proposition \ref{Lipschitz_Continuity_u_n}, we have
	\begin{equation}
		\left |u_n(s,\mathbf{X}(s)) - u_n(s,\mathbf{x}) \right | \leq C | \mathbf{X}(s) - \mathbf{x} |_r.
	\end{equation}
	Altogether, we obtain
	\begin{equation}
		\begin{split}
			&|u_n(t,\mathbf{x}) - u_n(s,\mathbf{x})| \\
			&\leq C \left ( 1+ \sup_{\mathbf{a}(\cdot)\in \mathcal{A}^{n,K\sqrt{n}}_t} \mathbb{E} \left [ \sup_{s^{\prime}\in[t,s]} | \mathbf{X}(s^{\prime})|_r \right ] \right ) |s-t| + \sup_{\mathbf{a}(\cdot)\in \mathcal{A}^{n,K\sqrt{n}}_t} \mathbb{E} \left [ |\mathbf{X}(s) - \mathbf{x}|_r \right ].
		\end{split}
	\end{equation}
	Applying Proposition \ref{Prop_Finite_A_Priori}(i) and (ii) concludes the proof. 
\end{proof}

\begin{remark}
	Note that we did not use the $C^{1,1}$ regularity of the coefficients to obtain this result.
\end{remark}

\subsection{Convergence of \texorpdfstring{$u_n$}{un}}

Once Propositions \ref{Lipschitz_Continuity_u_n} and \ref{Lipschitz_Time_u_n} are established, the construction of the functions $\mathcal{V}_n$ and their limit $\mathcal{V}$ follows exactly the additive noise case, see the beginning of Section 4 in \cite{mayorga_swiech_2023}. We streamline the presentation here for the convenience of the reader.

Since the functions $u_n$ are invariant under permutations of $\mathbf{x}= (x_1,\dots,x_n)\in (\mathbb{R}^d)^n$, we can define the function $\tilde{\mathcal{V}}_n(t,\mu_{\mathbf{x}}) := u_n(t,\mathbf{x})$, where $\mu_{\mathbf{x}} = \frac1n \sum_{i=1}^n \delta_{x_i}$. Each $\tilde{\mathcal{V}}_n$ is now defined on the subset $\mathcal{D}_n \subset \mathcal{P}_r(\mathbb{R}^d)$ consisting of averages of $n$ Dirac point masses. For $R>0$ and $\bar{r}\in [r,2]$, we denote $\mathfrak{M}^{\bar{r}}_{R} := \{ \mu\in \mathcal{P}_r(\mathbb{R}^d) : \int_{\mathbb{R}^d} |x|^{\bar{r}} \mu(\mathrm{d}x) \leq R \}$. By Propositions \ref{Lipschitz_Continuity_u_n} and \ref{Lipschitz_Time_u_n}, for every $R>0$, there is a constant $C_R$ such that
\begin{equation}
	| \tilde{\mathcal{V}}_n(t,\mu_{\mathbf{x}}) - \tilde{\mathcal{V}}_n(s,\mu_{\mathbf{y}}) | \leq C d_r(\mu_{\mathbf{x}},\mu_{\mathbf{y}}) + C_R |t-s|^{\frac12}
\end{equation}
for all $t,s\in [0,T]$ and $\mathbf{x},\mathbf{y} \in (\mathbb{R}^d)^n$ such that $\mu_{\mathbf{x}},\mu_{\mathbf{y}}\in \mathfrak{M}^r_R$, with $C>0$ being the constant from Proposition \ref{Lipschitz_Continuity_u_n}. In particular, there is a constant $C$ such that
\begin{equation}
	| \tilde{\mathcal{V}}_n(t,\mu_{\mathbf{x}}) | \leq C(1+|\mathbf{x}|_r).
\end{equation}
We extend $\tilde{\mathcal{V}}_n$ to a function $\mathcal{V}_n$ on $[0,T]\times \mathcal{P}_2(\mathbb{R}^d)$ by setting
\begin{equation}
	\mathcal{V}_n(t,\mu) := \sup_{\beta\in \mathcal{D}_n} \{ \tilde{\mathcal{V}}_n(t,\beta) - 2C d_r(\mu,\beta) \}.
\end{equation}
Then $\mathcal{V}_n(t,\cdot)$ has Lipschitz constant $2C$ for any $t\in [0,T]$, see \cite{mcshane_1934}. Moreover, one can show that
\begin{equation}\label{uniform_continuity_mathcal_V_n}
	|\mathcal{V}_n(t,\mu) - \mathcal{V}_n(s,\beta)| \leq 2Cd_r(\mu,\beta) + C|t-s|^{\frac12}
\end{equation}
for all $t,s\in [0,T]$ and $\mu,\beta\in \mathfrak{M}_R^r$, and in particular
\begin{equation}
	|\mathcal{V}_n(t,\mu)| \leq C \left ( 1+ \mathcal{M}^{\frac{1}{r}}_r(\mu) \right ).
\end{equation}
Thus, the sequence $\mathcal{V}_n$ is equicontinuous and bounded on bounded sets in $[0,T]\times\mathcal{P}_r(\mathbb{R}^d)$. Since the sets $\mathfrak{M}^2_R$ are relatively compact in $\mathcal{P}_r(\mathbb{R}^d)$ (see e.g. \cite[Proposition 7.1.5]{ambrosio_gigli_savare_2005}), it follows from the Arzel\`{a}--Ascoli theorem that a subsequence of $(\mathcal{V}_n)_{n\in\mathbb{N}}$, still denoted by $(\mathcal{V}_n)_{n\in\mathbb{N}}$, converges uniformly on every set $[0,T]\times \mathfrak{M}^2_R$ to a function $\mathcal{V}:[0,T]\times \mathcal{P}_2(\mathbb{R}^d) \to \mathbb{R}$ which also satisfies estimate \eqref{uniform_continuity_mathcal_V_n}.

Now, we define $V:[0,T]\times E\to \mathbb{R}$ by
\begin{equation}\label{Definition_of_V}
	V(t,X) := \mathcal{V}(t,X_{\texttt{\#}}\mathcal{L}^1).
\end{equation}
Recall that we call $\mathcal{V}$ an $L$-viscosity solution of equation \eqref{intro:HJB_on_Wasserstein_space}, if its lift $V$ is a viscosity solution of equation \eqref{intro:lifted_HJB}. Here, the notion of viscosity solution for equations on Hilbert spaces is the natural extension of the one used in finite dimensional spaces, see Appendix A.

\begin{theorem}\label{theorem:convergence}
	Let Assumptions \ref{Assumption_b_sigma_lipschitz} and \ref{Assumption_running_terminal_cost} be satisfied and let $V$ be defined by \eqref{Definition_of_V}. Then, for every bounded set $B$ in $\mathcal{P}_2(\mathbb{R}^d)$, we have
	\begin{multline}
		\lim_{n\to \infty} \sup \Bigg \{ \left | u_n(t,x_1,\dots,x_n) - \mathcal{V}\left ( t, \frac{1}{n} \sum_{i=1}^n \delta_{x_i} \right ) \right |\\
		: (t,x_1,\dots,x_n)\in (0,T]\times (\mathbb{R}^d)^n, \frac{1}{n} \sum_{i=1}^n \delta_{x_i} \in B \Bigg \} =0
	\end{multline}
	and $\mathcal{V}$ is the unique L-viscosity solution of equation \eqref{intro:HJB_on_Wasserstein_space} in the class of functions
	$\mathcal{W}$ whose lifts are uniformly continuous on bounded subsets of $[0,T]\times E$ and satisfy 
	\begin{equation}
		|W(t,X)-W(t,Y)| \leq C\|X-Y\|_E\quad\text{for all}\,\,t\in[0,T],\,\,X,Y\in E.
	\end{equation}
	Moreover, $V=U$, where $U$ is given by \eqref{Lifted_Value_Function}.
\end{theorem}

\begin{proof}
	First, we show that the limit $\mathcal{V}$ is an L-viscosity solution of equation \eqref{intro:HJB_on_Wasserstein_space}. The proof of convergence then follows from the discussion preceding the theorem and the uniqueness of the L-viscosity solution in the above class, which can be found in\footnote{Note that in the present case, there is no unbounded operator, i.e., the operator $A$ in \cite[Theorem 3.54]{fabbri_gozzi_swiech_2017} is zero, and the operator $B$ for the strong $B$-condition there can be chosen as the identity. In fact, the strong and weak $B$-conditions are then equivalent and either \cite[Theorem 3.50]{fabbri_gozzi_swiech_2017} or \cite[Theorem 3.54]{fabbri_gozzi_swiech_2017} can be used for comparison.} \cite[Theorem 3.54]{fabbri_gozzi_swiech_2017}. This part of the argument follows exactly the proof of \cite[Theorem 4.1]{mayorga_swiech_2023}.
	
	For the proof that $\mathcal{V}$ is an L-viscosity solution of equation \eqref{intro:HJB_on_Wasserstein_space}, the only difference in the case of multiplicative noise, is the different structure of the second order term in equation \eqref{intro:lifted_HJB}. Thus, we will streamline the presentation and only highlight the additional difficulties due to the multiplicative noise.
	
	We will show that $V$ is a viscosity subsolution of \eqref{intro:lifted_HJB}; the proof that it is a supersolution is similar. To this end, let $\varphi \in C^{1,2}((0,T)\times E)$ be such that $V-\varphi$ has a strict, global maximum at $(t,X)\in (0,T)\times E$. We remind that strict maximum means that whenever $(V-\varphi)(t_n,X_n)\to (V-\varphi)(t,X)$, then $(t_n,X_n)\to (t,X)$. Let $P= D\varphi(t,X)$. We want to show that
	\begin{equation}\label{Convergence_Proof_To_Show}
		\begin{cases}
			\partial_t \varphi(t,X) + \frac12 \sum_{m=1}^{d^{\prime}} \langle D^2 \varphi(t,X) \Sigma(X) e^{\prime}_m, \Sigma(X)e^{\prime}_m \rangle_E \\
			\quad + \langle B(X), P \rangle_E + \int_{\Omega} l_1(X(\omega),X_{\texttt{\#}}\mathcal{L}^1) \mathrm{d}\omega - \int_{\Omega} l_2^{\ast}(P(\omega)) \mathrm{d}\omega \geq 0,
		\end{cases}
	\end{equation}
	where $l_2^{\ast}(p) := \sup_{q\in \mathbb{R}^d} \{ q \cdot p - l_2(q) \}$ denotes the convex conjugate of $l_2$. For $\mathbf{x}\in (\mathbb{R}^d)^n$ and $t\in (0,T)$, define $\varphi_n(t,\mathbf{x}) := \varphi(t,X^{\mathbf{x}}_n)$, where $X^{\mathbf{x}}_n := \sum_{i=1}^n x_i \mathbf{1}_{A_i^n}$. Since $\sup \{ |u_n(t,\mathbf{x}) - V(t,X^{\mathbf{x}}_n)|:t\in [0,T], \mathbf{x}\in (\mathbb{R}^d)^n \text{ such that }X^{\mathbf{x}}_n \in B_1(X) \}\to 0$ as $n\to \infty$ and the maximum of $V-\varphi$ at $(t,X)$ is strict, there must be a sequence of points $(t_n,\mathbf{x}(n))$ such that the functions $u_n- \varphi_n$ have a local maximum over $\{(s,\mathbf{x}) \in [0,T]\times (\mathbb{R}^d)^n : X^{\mathbf{x}}_n \in B_1(X) \}$ at these points and $t_n \to t$, $X_n^{\mathbf{x}(n)} \to X$.
	
	We are going to use the fact that $u_n$ is a viscosity solution of equation \eqref{intro:finite_dimensional_HJB}. Recall that $A_n(\mathbf{x},\mu_{\mathbf{x}})$ consists of $n^2$ block matrices, each of dimension $d\times d$, given by $\sigma(x_i,\mu_{\mathbf{x}}) \sigma^{\top}(x_j,\mu_{\mathbf{x}})$, $i,j=1,\dots,n$. Thus, we have
	\begin{equation}
		\begin{split}
			\text{Tr}\left ( A_n(\mathbf{x},\mu_{\mathbf{x}}) D^2 \varphi_n(t,\mathbf{x}) \right ) &= \sum_{i,j=1}^n \text{Tr} \left ( \sigma(x_i,\mu_{\mathbf{x}}) \sigma^{\top}(x_j,\mu_{\mathbf{x}}) D^2_{x_jx_i} \varphi_n(t,\mathbf{x}) \right )\\
			&= \sum_{m=1}^{d^{\prime}}\sum_{i,j=1}^n 
			(D^2_{x_jx_i} \varphi_n(t,\mathbf{x})\sigma(x_i,\mu_{\mathbf{x}})e^{\prime}_m)\cdot(\sigma(x_j,\mu_{\mathbf{x}})e^{\prime}_m)
			\\
			&= \sum_{m=1}^{d^{\prime}}\sum_{i,j=1}^n \sum_{k=1}^d
			(D^2_{x_jx_i} \varphi_n(t,\mathbf{x})\sigma(x_i,\mu_{\mathbf{x}}))_{km}\sigma_{km}(x_j,\mu_{\mathbf{x}})
			\\
			&= \sum_{m=1}^{d^{\prime}}\sum_{i,j=1}^n \sum_{k,l=1}^d
			D^2_{x_j^kx_i^l} \varphi_n(t,\mathbf{x})\sigma_{lm}(x_i,\mu_{\mathbf{x}})\sigma_{km}(x_j,\mu_{\mathbf{x}}).
		\end{split}
	\end{equation}
	For the derivatives of $\varphi_n$, we have
	\begin{equation}
		D_{x_j} \varphi_n (t,\mathbf{x}) = \int_{\Omega} D\varphi(t,X^{\mathbf{x}}_n) \mathbf{1}_{A^n_j} \mathrm{d}\omega,
	\end{equation}
	as well as
	\begin{equation}\label{Varphi_Second_Derivative}
		D^2_{x_j^kx_i^l} \varphi_n(t,\mathbf{x}) = \int_{\Omega} D^2 \varphi(t,X^{\mathbf{x}}_n) ( e_k \mathbf{1}_{A^n_j})\cdot e_l \mathbf{1}_{A^n_i} \mathrm{d}\omega,
	\end{equation}
	for $i,j=1,\dots,n$ and $k,l=1,\dots,d$, where $(e_k)_{k=1,\dots,d}$ denotes the standard basis of $\mathbb{R}^d$. Therefore,
	\begin{equation}\label{computation_second_derivative}
		\begin{split}
			&\text{Tr}\left ( A_n(\mathbf{x},\mu_{\mathbf{x}}) D^2 \varphi_n(t,\mathbf{x}) \right )\\
			&= \sum_{m=1}^{d^{\prime}}\sum_{i,j=1}^n \sum_{k,l=1}^d
			\left( \int_{\Omega} D^2 \varphi(t,X^{\mathbf{x}}_n) ( e_k \mathbf{1}_{A^n_j})\cdot e_l \mathbf{1}_{A^n_i} \mathrm{d}\omega\right)\sigma_{lm}(x_i,\mu_{\mathbf{x}})\sigma_{km}(x_j,\mu_{\mathbf{x}})
			\\
			&= \sum_{m=1}^{d^{\prime}}\sum_{i,j=1}^n
			\int_{\Omega} D^2 \varphi(t,X^{\mathbf{x}}_n) \left( \sum_{k=1}^de_k \sigma_{km}(x_j,\mu_{\mathbf{x}})\mathbf{1}_{A^n_j}\right)\cdot \left(\sum_{l=1}^de_l \sigma_{lm}(x_i,\mu_{\mathbf{x}})\mathbf{1}_{A^n_i}\right)\mathrm{d}\omega
			\\
			&= \sum_{m=1}^{d^{\prime}} \sum_{i,j=1}^n \int_{\Omega} D^2 \varphi(t,X^{\mathbf{x}}_n) ( \sigma(x_j,\mu_{\mathbf{x}}) e^{\prime}_m \mathbf{1}_{A^n_j} ) \cdot (\sigma(x_i,\mu_{\mathbf{x}}) e^{\prime}_m \mathbf{1}_{A^n_i})\mathrm{d}\omega\\
			&=\sum_{m=1}^{d^{\prime}} \int_{\Omega} D^2 \varphi(t,X^{\mathbf{x}}_n) (\sigma(X^{\mathbf{x}}_n,(X^{\mathbf{x}}_n)_{\texttt{\#}}\mathcal{L}^1) e^{\prime}_m)(\omega) \cdot (\sigma(X^{\mathbf{x}}_n(\omega),(X^{\mathbf{x}}_n)_{\texttt{\#}}\mathcal{L}^1) e^{\prime}_m) \mathrm{d}\omega\\
			&=\sum_{m=1}^{d^{\prime}} \langle D^2 \varphi(t,X^{\mathbf{x}}_n) \Sigma(X^{\mathbf{x}}_n) e^{\prime}_m, \Sigma(X^{\mathbf{x}}_n) e^{\prime}_m \rangle_E.
		\end{split}
	\end{equation}
	Thus, since $u_n - \varphi_n$ have local maxima at $(t_n,\mathbf{x}(n))$ and $u_n$ is a viscosity solution of equation \eqref{intro:finite_dimensional_HJB}, it holds
	\begin{equation}\label{varphi_n_test_function}
		\begin{cases}
			\partial_t \varphi(t_n,X^{\mathbf{x}(n)}_n) + \frac12 \sum_{m=1}^{d^{\prime}} \langle D^2 \varphi(t_n,X_n^{\mathbf{x}(n)}) \Sigma(X_n^{\mathbf{x}(n)}) e^{\prime}_m, \Sigma(X_n^{\mathbf{x}(n)})e^{\prime}_m \rangle_E \\
			\quad + \sum_{i=1}^n b(x_i(n),\mu_{\mathbf{x}(n)}) \cdot \int_{A^n_i} D \varphi(t_n,X^{\mathbf{x}(n)}_n) \mathrm{d}\omega\\
			\quad + \frac{1}{n} \sum_{i=1}^n l_1(x_i(n),\mu_{\mathbf{x}(n)}) - \frac{1}{n} \sum_{i=1}^n l_2^{\ast}\left (n \int_{A^n_i} D \varphi(t_n,X_n^{\mathbf{x}(n)}) \mathrm{d}\omega \right ) \geq 0.
		\end{cases}
	\end{equation}
	The remainder of the proof follows along the same lines as the proof of \cite[Theorem 4.1]{mayorga_swiech_2023} since it is obvious that
	\begin{equation}
		\sum_{m=1}^{d^{\prime}} \langle D^2 \varphi(t_n,X_n^{\mathbf{x}(n)}) \Sigma(X_n^{\mathbf{x}(n)}) e^{\prime}_m, \Sigma(X_n^{\mathbf{x}(n)})e^{\prime}_m \rangle_E\to \sum_{m=1}^{d^{\prime}} \langle D^2 \varphi(t,X )\Sigma(X) e^{\prime}_m, \Sigma(X)e^{\prime}_m \rangle_E
	\end{equation}
	as $n\to\infty$. The argument that $V=U$ and that $V$ is the unique viscosity solution in the given class of functions is the same as that in the proof of \cite[Theorem 4.1]{mayorga_swiech_2023} if we use the Lipschitz regularity estimate of Lemma \ref{Value_Function_Lipschitz}.
\end{proof}

\section{\texorpdfstring{$C^{1,1}$}{C^{1,1}} Regularity of the Infinite Dimensional Value Function}\label{section:C11_regularity}

In this section, we are going to prove $C^{1,1}$ regularity in the spatial variable for the value function $U$ of the lifted infinite dimensional control problem introduced in Subsection \ref{Section_Infinite_Dimensional_Problem}.

\subsection{Estimates for the Infinite Dimensional Problem}

First, we derive a priori estimates for the solution of the lifted state equation \eqref{Lifted_State_Equation}.

\begin{lemma}\label{Infinite_Dimensional_A_Priori}
	Let Assumption \ref{Assumption_b_sigma_lipschitz} be satisfied. Let $X,X_0,X_1$ be the solutions of equation \eqref{Lifted_State_Equation} with initial conditions $Y,Y_0,Y_1\in E$, and controls $a(\cdot),a_0(\cdot),a_1(\cdot)\in \mathcal{A}_t$, respectively. Then, there is a constant $C$ such that
	\begin{equation}\label{Apriori_0_Infinite_State}
		\mathbb{E} \left [ \int_{\Omega} \sup_{s\in [t,T]} | X(s,\omega) |^2 \mathrm{d}\omega \right ] \leq C \left ( 1 + \|Y\|_E^2 + \| a(\cdot) \|^2_{M^2(t,T;E)} \right ),
	\end{equation}
	for all $Y\in E$, and $a(\cdot)\in \mathcal{A}_t$. Moreover, there is a constant $C$ such that
	\begin{equation}\label{Apriori_1_Infinite_State}
		\mathbb{E} \left [ \int_{\Omega} \sup_{s\in [t,T]} | X_1(s,\omega) - X_0(s,\omega) |^2 \mathrm{d}\omega \right ] \leq C \left ( \|Y_1-Y_0\|_E^2 + \| a_1(\cdot) - a_0(\cdot) \|^2_{M^2(t,T;E)} \right ),
	\end{equation}
	for all $Y_0,Y_1\in E$, and $a_0(\cdot),a_1(\cdot)\in \mathcal{A}_t$.
\end{lemma}

\begin{proof} 
	We are only going to prove inequality \eqref{Apriori_1_Infinite_State}. The proof of inequality \eqref{Apriori_0_Infinite_State} follows along the same lines. For every $s^{\prime} \in [t,s]$, we have
	\begin{equation}
		\begin{split}
			| X_1(s^{\prime},\omega)-X_0(s^{\prime},\omega)|^2 &\leq C |Y_1(\omega) - Y_0(\omega)|^2 + C \int_t^s |a_1(t^{\prime},\omega) - a_0(t^{\prime},\omega) |^2 \mathrm{d}t^{\prime} \\
			&\quad + C \int_t^{s} |B(X_1(t^{\prime}))(\omega) - B(X_0(t^{\prime}))(\omega) |^2 \mathrm{d}t^{\prime}\\
			&\quad + C \left | \int_t^{s^{\prime}} \left ( \Sigma(X_1(t^{\prime}))(\omega) - \Sigma(X_0(t^{\prime}))(\omega) \right ) \mathrm{d}W(t^{\prime}) \right |^2.
		\end{split}
	\end{equation}
	Thus, taking the supremum over $s^{\prime} \in[t,s]$, the expectation, and the integral over $\Omega$, we obtain
	\begin{equation}\label{Infinite_Dimensional_A_Priori_Inequality_1}
		\begin{split}
			&\int_{\Omega} \mathbb{E} \left [ \sup_{s^{\prime} \in [t,s]} | X_1(s^{\prime},\omega)-X_0(s^{\prime},\omega)|^2 \right ] \mathrm{d}\omega\\
			&\leq C \|Y_1 - Y_0\|_E^2 + C  \| a_1(\cdot) - a_0(\cdot) \|_{M^2(t,T;E)}^2 \\
			&\quad + C \int_{\Omega} \mathbb{E} \left [ \int_t^s |B(X_1(s'))(\omega) - B(X_0(s'))(\omega) |^2 \mathrm{d}s' \right ] \mathrm{d} \omega\\
			&\quad + C \int_{\Omega} \mathbb{E} \left [ \sup_{s^{\prime} \in [t,s]} \left | \int_t^{s^{\prime}} \left ( \Sigma(X_1(t^{\prime}))(\omega) - \Sigma(X_0(t^{\prime}))(\omega) \right ) \mathrm{d}W(t^{\prime}) \right |^2 \right ] \mathrm{d}\omega.
		\end{split}
	\end{equation}
	Due to the Lipschitz continuity of $B$, we have
	\begin{equation}
		\begin{split}
			&\int_{\Omega} \mathbb{E} \left [ \int_t^s |B(X_1(s'))(\omega) - B(X_0(s'))(\omega) |^2 \mathrm{d}s' \right ] \mathrm{d} \omega \\
			&\leq C \int_{\Omega} \mathbb{E} \left [ \int_t^s \sup_{t^{\prime}\in [t,s^{\prime}]} | X_1(t^{\prime},\omega) - X_0(t^{\prime},\omega) |^2 \mathrm{d}s^{\prime} \right ] \mathrm{d}\omega.
		\end{split}
	\end{equation}
	Moreover, by Burkholder--Davis--Gundy inequality and the Lipschitz continuity of $\Sigma$, we have
	\begin{equation}
		\begin{split}
			&\int_{\Omega} \mathbb{E} \left [ \sup_{s^{\prime} \in [t,s]} \left | \int_t^{s^{\prime}} \left ( \Sigma(X_1(t^{\prime}))(\omega) - \Sigma(X_0(t^{\prime}))(\omega) \right ) \mathrm{d}W(t^{\prime}) \right |^2 \right ] \mathrm{d}\omega\\
			&\leq C \int_{\Omega} \mathbb{E} \left [ \int_t^s | \Sigma(X_1(s^{\prime}))(\omega) - \Sigma(X_0(s^{\prime}))(\omega) |^2 \mathrm{d}s^{\prime} \right ] \mathrm{d}\omega\\
			&\leq C \int_{\Omega} \mathbb{E} \left [ \int_t^s \sup_{t^{\prime} \in [t,s^{\prime}]} | X_1(t^{\prime},\omega) - X_0(t^{\prime},\omega) |^2 \mathrm{d}s^{\prime} \right ] \mathrm{d}\omega.
		\end{split}
	\end{equation}
	Thus, applying Gr\"onwall's inequality in equation \eqref{Infinite_Dimensional_A_Priori_Inequality_1} concludes the proof.
\end{proof}

For the following lemma, we are going to use the same notation as in the previous lemma. Moreover, for $\lambda\in [0,1]$, $s\in [t,T]$, we set
\begin{align}
	&a_{\lambda}(s) := \lambda a_1(s) + (1-\lambda) a_0(s), && Y_{\lambda} := \lambda Y_1 + (1-\lambda)Y_0\\
	&X^{\lambda}(s) := \lambda X_1(s) + (1-\lambda)X_0(s), && X_{\lambda}(s) := X(s;Y_{\lambda},a_{\lambda}(\cdot)),
\end{align}
where $X(s;Y_{\lambda},a_{\lambda}(\cdot))$ denotes the solution of equation \eqref{Lifted_State_Equation} at time $s$ with initial condition $Y_{\lambda}$ and control $a_{\lambda}(\cdot)$.

\begin{lemma}\label{lem:Xlambda}
	Let Assumption \ref{Assumption_b_sigma_lipschitz} be satisfied with $r=1$. Then, there is a constant $C$, such that
	\begin{equation}
		\begin{split}
			&\int_{\Omega} \mathbb{E} \left [ \sup_{s\in [t,T]} |X^{\lambda}(s,\omega) - X_{\lambda}(s,\omega) | \right ] \mathrm{d}\omega\\
			&\leq C \lambda (1-\lambda) \left ( \|Y_1-Y_0\|_E^2 + \| a_1(\cdot) - a_0(\cdot) \|_{M^2(t,T;E)}^2 \right )
		\end{split}
	\end{equation}
	for all $\lambda\in [0,1]$, $Y_0,Y_1\in E$, and $a_0(\cdot),a_1(\cdot)\in \mathcal{A}_t$.
\end{lemma}

\begin{proof}
	Let
	\begin{equation}
		\begin{split}
			&Z_0(\theta) = X^{\lambda}(s) + \theta \lambda (X_0(s) - X_1(s)),\\
			&Z_1(\theta) = X^{\lambda}(s) + \theta (1-\lambda) (X_1(s) - X_0(s)).
		\end{split}
	\end{equation}
	First, since $\tilde{b}$ and $\tilde{\sigma}$ are $C^{1,1}$, we have for $s\in [t,T]$ and $\omega\in \Omega$,
	\begin{equation}\label{X_lambda_Estimate_B}
		\begin{split}
			&| \lambda B(X_1(s))(\omega) + (1-\lambda) B(X_0(s))(\omega) - B(X^{\lambda}(s))(\omega) | \\
			&= | \lambda \tilde{b}(X_1(s,\omega),X_1(s)) + (1-\lambda) \tilde{b}(X_0(s,\omega),X_0(s)) - \tilde{b}(X^{\lambda}(s,\omega),X^{\lambda}(s)) | \\
			&\leq \lambda(1-\lambda) \int_0^1 \Big | ( D_x \tilde{b}(Z_1(\theta,\omega), Z_1(\theta)) - D_x \tilde{b}(Z_0(\theta,\omega), Z_0(\theta)))\\
			&\qquad\qquad\qquad\qquad\qquad\qquad\qquad\qquad\qquad\qquad (X_1(s,\omega) - X_0(s,\omega)) \Big | \mathrm{d}\theta\\
			&\quad + \lambda(1-\lambda) \int_0^1 \Big | ( D_X \tilde{b}(Z_1(\theta,\omega), Z_1(\theta)) - D_X \tilde{b}(Z_0(\theta,\omega), Z_0(\theta)) )\\
			&\qquad\qquad\qquad\qquad\qquad\qquad\qquad\qquad\qquad\qquad (X_1(s,\omega) - X_0(s,\omega)) \Big |\mathrm{d}\theta\\
			&\leq C \lambda (1-\lambda) \left ( |X_1(s,\omega) - X_0(s,\omega)|^2 + \| X_1(s) - X_0(s) \|_E^2 \right )
		\end{split}
	\end{equation}
	and similarly
	\begin{equation}\label{X_lambda_Estimate_Sigma}
		\begin{split}
			&| \lambda \Sigma(X_1(s))(\omega) + (1-\lambda) \Sigma(X_0(s))(\omega) - \Sigma(X^{\lambda}(s))(\omega) | \\
			&\leq C \lambda (1-\lambda) \left ( |X_1(s,\omega) - X_0(s,\omega)|^2 + \| X_1(s) - X_0(s) \|_E^2 \right ).
		\end{split}
	\end{equation}
	Now, for every $s^{\prime} \in [t,s]$,
	\begin{equation}
		\begin{split}
			&| X_{\lambda}(s^{\prime},\omega) - X^{\lambda}(s^{\prime},\omega)|\\
			&\leq \int_t^s | \lambda B(X_1(t^{\prime}))(\omega) + (1-\lambda) B(X_0(t^{\prime}))(\omega) - B(X^{\lambda}(t^{\prime}))(\omega) | \mathrm{d}t^{\prime}\\
			&\quad + \int_t^s | B(X^{\lambda}(t^{\prime}))(\omega) - B(X_{\lambda}(t^{\prime}))(\omega) | \mathrm{d}t^{\prime}\\
			&\quad + \left | \int_t^{s^{\prime}} (\lambda \Sigma(X_1(t^{\prime}))(\omega) +(1-\lambda) \Sigma(X_0(t^{\prime}))(\omega) - \Sigma(X^{\lambda}(t^{\prime}))(\omega) ) \mathrm{d}W(t^{\prime}) \right | \\
			&\quad + \left | \int_t^{s^{\prime}} ( \Sigma(X^{\lambda}(t^{\prime}))(\omega) - \Sigma(X_{\lambda}(t^{\prime}))(\omega) ) \mathrm{d}W(t^{\prime}) \right |.
		\end{split}
	\end{equation}
	Taking the supremum over $[t,s]$, the expectation and the integral over $\Omega$, we obtain
	\begin{equation}\label{X_lambda_Estimate_gronwall}
		\begin{split}
			&\int_{\Omega} \mathbb{E} \left [ \sup_{s^{\prime}\in[t,s]} | X_{\lambda}(s^{\prime},\omega) - X^{\lambda}(s^{\prime},\omega)| \right ] \mathrm{d}\omega\\
			&\leq \int_{\Omega} \mathbb{E} \left [ \int_t^s | \lambda B(X_1(s^{\prime}))(\omega) + (1-\lambda) B(X_0(s^{\prime}))(\omega) - B(X^{\lambda}(s^{\prime}))(\omega) | \mathrm{d}s^{\prime} \right ] \mathrm{d}\omega \\
			&\quad + \int_{\Omega} \mathbb{E} \left [ \int_t^s | B(X^{\lambda}(s^{\prime}))(\omega) - B(X_{\lambda}(s^{\prime}))(\omega) | \mathrm{d}s^{\prime} \right ] \mathrm{d}\omega \\
			&\quad + \int_{\Omega} \mathbb{E} \Bigg [ \sup_{s^{\prime}\in [t,s]} \bigg | \int_t^{s^{\prime}} (\lambda \Sigma(X_1(t^{\prime}))(\omega) +(1-\lambda) \Sigma(X_0(t^{\prime}))(\omega)\\
			&\qquad\qquad\qquad\qquad\qquad\qquad\qquad\qquad\qquad\qquad - \Sigma(X^{\lambda}(t^{\prime}))(\omega) ) \mathrm{d}W(t^{\prime}) \bigg | \Bigg ] \mathrm{d}\omega \\
			&\quad + \int_{\Omega} \mathbb{E} \left [ \sup_{s^{\prime}\in [t,s]} \left | \int_t^{s^{\prime}} ( \Sigma(X^{\lambda}(t^{\prime}))(\omega) - \Sigma(X_{\lambda}(t^{\prime}))(\omega) ) \mathrm{d}W(t^{\prime}) \right | \right ] \mathrm{d}\omega.
		\end{split}
	\end{equation}
	For the first term on the right-hand side, using \eqref{X_lambda_Estimate_B}, we obtain
	\begin{equation}
		\begin{split}
			&\int_{\Omega} \mathbb{E} \left [ \int_t^s | \lambda B(X_1(s^{\prime}))(\omega) + (1-\lambda) B(X_0(s^{\prime}))(\omega) - B(X^{\lambda}(s^{\prime}))(\omega) | \mathrm{d}s^{\prime} \right ] \mathrm{d}\omega\\
			&\leq C\lambda (1-\lambda) \int_{\Omega} \mathbb{E} \left [ \int_t^s |X_1(s^{\prime},\omega) - X_0(s^{\prime},\omega)|^2 \mathrm{d}s^{\prime} \right ] \mathrm{d}\omega.
		\end{split}
	\end{equation}
	For the second term, using the Lipschitz continuity of $b$, we obtain
	\begin{equation}
		\begin{split}
			&\int_{\Omega} \mathbb{E} \left [ \int_t^s | B(X^{\lambda}(s^{\prime}))(\omega) - B(X_{\lambda}(s^{\prime}))(\omega) | \mathrm{d}s^{\prime} \right ] \mathrm{d}\omega\\
			&= \int_{\Omega} \mathbb{E} \left [ \int_t^s | b(X^{\lambda}(s^{\prime},\omega), X^{\lambda}(s^{\prime})_{\texttt{\#}}\mathcal{L}^1) - b(X_{\lambda}(s^{\prime},\omega), X_{\lambda}(s^{\prime})_{\texttt{\#}}\mathcal{L}^1) | \mathrm{d}s^{\prime} \right ] \mathrm{d}\omega\\
			&\leq C \int_{\Omega} \mathbb{E} \left [ \int_t^s \left ( |X^{\lambda}(s^{\prime},\omega) - X_{\lambda}(s^{\prime},\omega) | + d_1(X_{\lambda}(s^{\prime})_{\texttt{\#}} \mathcal{L}^1, X^{\lambda}(s^{\prime})_{\texttt{\#}} \mathcal{L}^1 ) \right ) \mathrm{d}s^{\prime} \right ] \mathrm{d}\omega\\
			&\leq 2 C \int_{\Omega} \mathbb{E} \left [ \int_t^s |X^{\lambda}(s^{\prime},\omega) - X_{\lambda}(s^{\prime},\omega) | \mathrm{d}s^{\prime} \right ] \mathrm{d}\omega,
		\end{split}
	\end{equation}
	where in the last step, we used the fact that
	\begin{equation}\label{estimtate_d1}
		d_1(X_{\lambda}(s^{\prime})_{\texttt{\#}} \mathcal{L}^1, X^{\lambda}(s^{\prime})_{\texttt{\#}} \mathcal{L}^1 ) \leq \int_{\Omega} |X^{\lambda}(s^{\prime},\omega) - X_{\lambda}(s^{\prime},\omega) | \mathrm{d}\omega.
	\end{equation}
	For the third term, using Burkholder--Davis--Gundy inequality and \eqref{X_lambda_Estimate_Sigma}, we obtain
	\begin{equation}\label{Third_Term}
		\begin{split}
			&\int_{\Omega} \mathbb{E} \left [ \sup_{s^{\prime} \in [t,s]} \left | \int_t^{s^{\prime}} (\lambda \Sigma(X_1(t^{\prime}))(\omega) +(1-\lambda) \Sigma(X_0(t^{\prime}))(\omega) - \Sigma(X^{\lambda}(t^{\prime}))(\omega) ) \mathrm{d}W(t^{\prime}) \right | \right ] \mathrm{d}\omega\\
			&\leq \int_{\Omega} \mathbb{E} \left [ \left ( \int_t^s \left | \lambda \Sigma(X_1(s^{\prime}))(\omega) +(1-\lambda) \Sigma(X_0(s^{\prime}))(\omega) - \Sigma(X^{\lambda}(s^{\prime}))(\omega) \right |^2 \mathrm{d}s^{\prime} \right )^{\frac12} \right ] \mathrm{d}\omega\\
			&\leq C \lambda (1-\lambda) \int_{\Omega} \mathbb{E} \left [ \left ( \int_t^s \left ( |X_1(s^{\prime},\omega) - X_0(s^{\prime},\omega)|^4 + \| X_1(s^{\prime}) - X_0(s^{\prime}) \|_E^4 \right ) \mathrm{d}s^{\prime} \right )^{\frac12} \right ] \mathrm{d}\omega\\
			&\leq C \lambda (1-\lambda) \int_{\Omega} \mathbb{E} \left [ \left ( \int_t^s |X_1(s^{\prime},\omega) - X_0(s^{\prime},\omega)|^4 \mathrm{d}s^{\prime} \right )^{\frac12} \right ] \mathrm{d}\omega\\
			&\quad+ C \lambda (1-\lambda) \mathbb{E} \left [ \left ( \int_t^s \| X_1(s^{\prime}) - X_0(s^{\prime}) \|_E^4 \mathrm{d}s^{\prime} \right )^{\frac12} \right ].
		\end{split}
	\end{equation}
	Noting that
	\begin{equation}
		\begin{split}
			&\mathbb{E} \left [ \left ( \int_t^s \| X_1(s^{\prime}) - X_0(s^{\prime}) \|_E^4 \mathrm{d}s^{\prime} \right )^{\frac12} \right ]\\
			&= \mathbb{E} \left [ \left ( \int_t^s \left ( \int_{\Omega} |X_1(s^{\prime},\omega) - X_0(s^{\prime},\omega)|^2 \mathrm{d}\omega \right )^2 \mathrm{d}s^{\prime} \right )^{\frac12} \right ]\\
			&\leq C\mathbb{E} \left [ \sup_{s^{\prime} \in [t,s]} \int_{\Omega} | X_1(s^{\prime},\omega) - X_0(s^{\prime},\omega) |^2 \mathrm{d}\omega \right ]\\
			&\leq C\int_{\Omega} \mathbb{E} \left [ \sup_{s^{\prime} \in [t,s]} |X_1(s^{\prime},\omega) - X_0(s^{\prime},\omega) |^2 \right ] \mathrm{d}\omega,
		\end{split}
	\end{equation}
	we obtain from \eqref{Third_Term}
	\begin{equation}
		\begin{split}
			&\int_{\Omega} \mathbb{E} \left [ \sup_{s^{\prime} \in [t,s]} \left | \int_t^{s^{\prime}} (\lambda \Sigma(X_1(t^{\prime}))(\omega) +(1-\lambda) \Sigma(X_0(t^{\prime}))(\omega) - \Sigma(X^{\lambda}(t^{\prime}))(\omega) ) \mathrm{d}W(t^{\prime}) \right | \right ] \mathrm{d}\omega\\
			&\leq 2 C \lambda (1-\lambda) \int_{\Omega} \mathbb{E} \left [ \sup_{s^{\prime} \in [t,s]} |X_1(s^{\prime},\omega) - X_0(s^{\prime},\omega)|^2 \right ] \mathrm{d}\omega.
		\end{split}
	\end{equation}
	Finally, for the fourth term, due to the Lipschitz continuity of $\sigma$, we have
	\begin{equation}
		\begin{split}
			&\int_{\Omega} \mathbb{E} \left [ \sup_{s^{\prime} \in [t,s]} \left | \int_t^{s^{\prime}} ( \Sigma(X^{\lambda}(t^{\prime}))(\omega) - \Sigma(X_{\lambda}(t^{\prime}))(\omega) ) \mathrm{d}W(t^{\prime}) \right | \right ] \mathrm{d}\omega\\
			&\leq \int_{\Omega} \mathbb{E} \left [  \left ( \int_t^s | \Sigma(X^{\lambda}(s^{\prime}))(\omega) - \Sigma(X_{\lambda}(s^{\prime}))(\omega) |^2 \mathrm{d}s^{\prime} \right )^{\frac12} \right ] \mathrm{d}\omega\\
			&= \int_{\Omega} \mathbb{E} \left [ \left ( \int_t^s | \sigma(X^{\lambda}(s^{\prime},\omega), X^{\lambda}(s^{\prime})_{\texttt{\#}}\mathcal{L}^1) - \sigma(X_{\lambda}(s^{\prime},\omega), X_{\lambda}(s^{\prime})_{\texttt{\#}}\mathcal{L}^1) |^2 \mathrm{d}s^{\prime} \right )^{\frac12} \right ] \mathrm{d}\omega\\
			&\leq C \int_{\Omega} \mathbb{E} \left [ \left ( \int_t^s \left ( |X^{\lambda}(s^{\prime},\omega) - X_{\lambda}(s^{\prime},\omega) |^2 + d^2_1(X_{\lambda}(s^{\prime})_{\texttt{\#}} \mathcal{L}^1, X^{\lambda}(s^{\prime})_{\texttt{\#}} \mathcal{L}^1 ) \right ) \mathrm{d}s^{\prime} \right )^{\frac12} \right ] \mathrm{d}\omega\\
			&\leq \frac12 \int_{\Omega} \mathbb{E} \left [ \sup_{s^{\prime} \in [t,s]} |X^{\lambda}(s^{\prime},\omega) - X_{\lambda}(s^{\prime},\omega)| \right ] \mathrm{d}\omega\\
			&\quad + C \int_{\Omega} \mathbb{E} \left [ \int_t^s |X^{\lambda}(s^{\prime},\omega) - X_{\lambda}(s^{\prime},\omega)| \mathrm{d}s^{\prime} \right ] \mathrm{d}\omega,
		\end{split}
	\end{equation}
	where in the last step we used \eqref{estimtate_d1} again. Combining these estimates with inequality \eqref{X_lambda_Estimate_gronwall}, applying Gr\"onwall's inequality, and using Lemma \ref{Infinite_Dimensional_A_Priori} concludes the proof. 
\end{proof}

\subsection{\texorpdfstring{$C^{1,1}$}{C^{1,1}} Regularity of the Value Function}\label{subsection:C11_regularity}

\begin{proposition}[Lipschitz Continuity]\label{Value_Function_Lipschitz}
	Let Assumptions \ref{Assumption_b_sigma_lipschitz} and \ref{Assumption_running_terminal_cost} be satisfied. Then, there is a constant $C$, depending only on the Lipschitz constants of $b$, $\sigma$, $l_1$ and $\mathcal{U}_T$ as well as on $T$, such that
	\begin{equation}
		| U(t,X) - U(t,Y) | \leq C \|X-Y\|_E
	\end{equation}
	for all $t\in [0,T]$ and $X,Y\in E$.
\end{proposition}

\begin{proof}
	The proof follows from the Lipschitz continuity of the coefficients of the cost functional as well as Lemma \ref{Infinite_Dimensional_A_Priori}.
\end{proof}

\begin{proposition}[Semiconcavity: Case 1]\label{Value_Function_Semiconcave_1}
	Let Assumptions \ref{Assumption_b_sigma_lipschitz} and \ref{Assumption_running_terminal_cost} be satisfied with $r=1$. Then, for every $t\in [0,T]$, $U(t,\cdot)$ is semiconcave, i.e., there is a constant $C\geq 0$ such that
	\begin{equation}
		\lambda U(t,X) + (1-\lambda) U(t,Y) - U(t,\lambda X + (1-\lambda) Y) \leq C \lambda (1-\lambda) \| X- Y \|_E^2
	\end{equation}
	for all $\lambda \in [0,1]$ and $X,Y\in E$. Moreover, the semiconcavity constant $C$ is independent of $t\in [0,T]$.
\end{proposition}

\begin{proof}
	The proof uses the same techniques as the proof of semiconcavity in the case of additive noise, see \cite[Proposition 6.1(ii)]{mayorga_swiech_2023}, now using Lemmas \ref{Infinite_Dimensional_A_Priori} and \ref{lem:Xlambda}.
\end{proof}

\begin{proposition}[Semiconcavity: Case 2]\label{Value_Function_Semiconcave_2}
	Let Assumptions \ref{Assumption_running_terminal_cost} and \ref{Assumption_Linear_State_Equation} be satisfied. Then, for every $t\in [0,T]$, $U(t,\cdot)$ is semiconcave
	with the semiconcavity constant independent of $t\in [0,T]$.
\end{proposition}

\begin{proof}
	The proof uses again the same techniques as the corresponding result in \cite[Proposition 6.2]{mayorga_swiech_2023}. We note that now we have $X_\lambda=X^\lambda$, so Lemma \ref{lem:Xlambda} is not needed and we do not need $r=1$.
\end{proof}

\begin{proposition}[Semiconvexity: Case 1]\label{Value_Function_Semiconvex_1}
	Let Assumptions \ref{Assumption_b_sigma_lipschitz} and \ref{Assumption_running_terminal_cost} be satisfied with $r=1$. There is a constant $\nu_0\geq 0$ (depending on $T$ and other constants in Assumptions \ref{Assumption_b_sigma_lipschitz} and \ref{Assumption_running_terminal_cost}) such that if $\nu$ in Assumption \ref{Assumption_running_terminal_cost}(ii) satisfies $\nu \geq \nu_0$, then, for every $t\in [0,T]$, $U(t,\cdot)$ is semiconvex, i.e., there is a constant $C\geq 0$ such that
	\begin{equation}
		\lambda U(t,X) + (1-\lambda) U(t,Y) - U(t,\lambda X + (1-\lambda) Y) \geq - C \lambda (1-\lambda) \| X- Y \|_E^2
	\end{equation}
	for all $\lambda \in [0,1]$ and $X,Y\in E$. Moreover, the semiconvexity constant $C$ is independent of $t\in [0,T]$. 
\end{proposition}

\begin{proof}
	The proof follows along the same lines as the proof in the case of additive noise, see \cite[Proposition 6.1(iii)]{mayorga_swiech_2023}, now using Lemmas \ref{Infinite_Dimensional_A_Priori} and \ref{lem:Xlambda}.
\end{proof}

\begin{proposition}[Semiconvexity: Case 2]\label{Value_Function_Semiconvex_2}
	Let Assumptions \ref{Assumption_running_terminal_cost} and \ref{Assumption_Linear_State_Equation} be satisfied. Then, for every $t\in [0,T]$, $U(t,\cdot)$ is convex.
\end{proposition}

\begin{proof}
	The proof again repeats the one in the case of additive noise, see \cite[Proposition 6.2]{mayorga_swiech_2023}. We note again that we have $X_\lambda=X^\lambda$, so Lemma \ref{lem:Xlambda} is not needed and we do not need $r=1$.
\end{proof}

\begin{remark}\label{remark_joint_continuity}
	\begin{enumerate}[label=(\roman*)]
		\item If for fixed $t\in [0,T]$, the value function $V(t,\cdot)$ is uniformly continuous, semiconcave and semiconvex, then it is in $C^{1,1}(E)$, see \cite{lasry_lions_1986}. Therefore, if either the assumptions of Propositions \ref{Value_Function_Semiconcave_1} and \ref{Value_Function_Semiconvex_1} are satisfied, or the assumptions of Propositions \ref{Value_Function_Semiconcave_2} and \ref{Value_Function_Semiconvex_2} are satisfied, then $V(t,\cdot)\in C^{1,1}(E)$, for all $t\in [0,T]$ and the Lipschitz constant of $DV(t,\cdot)$ is independent of $t$.
		\item If $V$ is uniformly continuous on bounded subsets of $[0,T]\times E$, $V(t,\cdot)\in C^{1,1}(E)$, for all $t\in [0,T]$ and the Lipschitz constant of $DV(t,\cdot)$ is independent of $t$, then by \cite[Lemma 4.8]{defeo_swiech_wessels_2023} (more precisely by its proof which here can be done on $[0,T]\times E$), $DV$ is uniformly continuous on bounded subsets of $[0,T]\times E$.
	\end{enumerate}
\end{remark}

\begin{remark}\label{rem:Lipint}
	We remark that since $U$ is a viscosity solution of \eqref{intro:lifted_HJB}, if $U(t,\cdot)\in C^{1,1}(E)$ for all $t\in [0,T]$ and the semiconvexity and semiconcavity constants are independent of $t$, then for every $R>0$ there is a constant $C_R$ such that
	\begin{equation}
		|U(t,X)-U(s,X)|\leq C_R|t-s|\quad\text{for all}\,\,t,s\in[0,T],\,\, \|X\|_E\leq R.
	\end{equation}
	We sketch the argument. Let $\eta<T/2$. Define for $\varepsilon>0$, the partial sup-convolution functions
	\begin{equation}
		U^\varepsilon(t,X):=\sup_{s\in[0,T]}\left\{U(s,X)-\frac{(t-s)^2}{\varepsilon}\right\}.
	\end{equation}
	We note that $U^{\varepsilon}$ is Lipschitz continuous in both variables and
	\begin{equation}
		(t,X)\to U^\varepsilon(t,X)+\frac{t^2}{\varepsilon}+C\|X\|_E^2
	\end{equation}
	is convex for some universal constant $C$.
	It is easy to see by standard viscosity solution arguments that for every $R>0$, for sufficiently small $\varepsilon$, the functions $U^\varepsilon$ are viscosity subsolutions of 
	\begin{equation}
		\partial_t U^\varepsilon=-C_R\quad\text{in}\,\,(\eta,T-\eta)\times B_R(0)
	\end{equation}
	for some constant $C_R>0$ independent of $\varepsilon$. (We remark that the above equation is considered to be a terminal value problem.) To obtain this, we used the Lipschitz continuity of $U$ in $X$, the growth bounds on the coefficient functions in \eqref{intro:lifted_HJB}, and the fact that if $U-\varphi$ has a maximum at some point $(t,X)$ for a smooth function $\varphi$, then $D^2\varphi(t,X)\geq -CI$ for some universal constant $C>0$. Hence, we obtained that on a dense set of points $(t,X)$ such that $U^\varepsilon-\varphi$ has a maximum at $(t,X)$ for a smooth function $\varphi$, which in particular implies that $U^\varepsilon$ is differentiable at $(t,X)$, we have $\partial_tU^\varepsilon(t,X)\geq -C_R$. By density, this implies that this inequality holds for every point $(t,X)$ of differentiability of $U^\varepsilon$ in $(\eta,T-\eta)\times B_R(0)$. Now, since for a semiconvex function $u$ in a Hilbert space
	\begin{equation}
		D^-u(x)=\overline{\rm conv}\{p: Du(y_n)\rightharpoonup p, y_n\to x\},
	\end{equation}
	where $D^-u(x)$ above is the subdifferential of $u$ at $x$ and $y_n$ are points of differentiability of $u$, we conclude that for every $(t,X)\in (\eta,T-\eta)\times B_R(0)$ and $(p_t,P_X)\in D^-U^\varepsilon(t,X)$, we have $p_t\geq -C_R$. Thus, if $X\in B_R(0)$ is fixed and $\eta<t<s<T-\eta$,
	\begin{equation}
		U^\varepsilon(s,X)-U^\varepsilon(t,X)=\int_t^s\partial_tU^\varepsilon(s^{\prime},X)\mathrm{d}s^{\prime}\geq -C_R(s-t),
	\end{equation}
	where we used that if $\partial_tU^\varepsilon(s^{\prime},X)$ exists, then $\partial_tU^\varepsilon(s^{\prime},X)=p_t$ for every $(p_t,P_X)\in D^-U^\varepsilon(t,X)$, in particular the $p_t$ element of the subdifferential is unique.
	Since $C_R$ does neither depend on $\eta$ nor on $\varepsilon$, by sending first $\varepsilon\to 0$ and then $\eta\to 0$, we obtain
	\begin{equation}
		U(s,X)-U(t,X)\geq -C_R(s-t)\quad\text{for all}\,\,0\leq t\leq s\leq T,\,\, X\in B_R(0).
	\end{equation}
	The opposite inequality is obtained by considering the partial inf-convolutions and arguing as above.
\end{remark}

\section{Projection of \texorpdfstring{$V$}{V} when \texorpdfstring{$DV$}{DV} is Continuous}\label{section:projection_C11}

In Theorem \ref{theorem:convergence}, we showed that the finite dimensional value functions $u_n$ converge to the infinite dimensional value function $V$. In this section, we are going to show that when $DV$ is continuous, $V$ actually projects precisely onto $u_n$. Note that the value function satisfies this regularity assumption in the case of $C^{1,1}$ regularity discussed in Section \ref{section:C11_regularity}, see Remark \ref{remark_joint_continuity}(ii).

For $n\in\mathbb{N}$, $i \in \{1,\dots,n\}$, recall that $A_i^n = \left (\frac{i-1}{n},\frac{i}{n} \right ) \subset \Omega$. Now, we introduce $V_n : [0,T]\times (\mathbb{R}^d)^n \to \mathbb{R}$,
\begin{equation}\label{Projection_V_n}
	V_n(t,x_1,\dots,x_n) := V\left ( t, \sum_{i=1}^n x_i \mathbf{1}_{A^n_i} \right ).
\end{equation}

\subsection{\texorpdfstring{$V_n\leq u_n$}{Vn<un}}

In this section, we are going to show that $V_n\leq u_n$, where $u_n$ is given by \eqref{finite_dimensional_value_function}. First, we prove the following result.
\begin{lemma}\label{lemma_subsolution}
	Let Assumptions \ref{Assumption_b_sigma_lipschitz} and \ref{Assumption_running_terminal_cost} be satisfied. Then, for any control $\mathbf{a}(\cdot)\in \mathcal{A}^n_t$, we have
	\begin{equation}
		J_n(t,\mathbf{x};\mathbf{a}(\cdot)) = J(t,X^{\mathbf{x}}_n;a^n(\cdot)),
	\end{equation}
	where $X^{\mathbf{x}}_n = \sum_{i=1}^n x_i \mathbf{1}_{A^n_i}$ and $a^n(\cdot) = \sum_{i=1}^n a_i(\cdot) \mathbf{1}_{A^n_i}$.
\end{lemma}

\begin{proof}
	Let $X^{n}(\cdot)$ be the solution of the lifted state equation \eqref{Lifted_State_Equation} with initial condition $X^{\mathbf{x}}_n$ and control $a^n(\cdot)$. First note that $\mathbb{P}$-a.s., for every $s\in [t,T]$,
	\begin{equation}
		X^{n}(s) = \sum_{i=1}^n x_i \mathbf{1}_{A^n_i} - \int_t^s \sum_{i=1}^n a_i(r) \mathbf{1}_{A^n_i} \mathrm{d}r +\int_t^s B(X^{n}(r)) \mathrm{d}r + \int_t^s \Sigma(X^{n}(r)) \mathrm{d}W(r),
	\end{equation}
	and therefore, for almost every $\omega\in A^n_i = (\frac{i-1}{n},\frac{i}{n})$, we have
	\begin{equation}
		\begin{split}
			&X^{n}(s,\omega)\\
			&= x_i - \int_t^s a_i(r) \mathrm{d}r + \int_t^s b(X^{n}(r,\omega),X^{n}(r)_{\texttt{\#}}\mathcal{L}^1) \mathrm{d}r + \int_t^s \sigma(X^{n}(r,\omega),X^{n}(r)_{\texttt{\#}}\mathcal{L}^1) \mathrm{d}W(r).
		\end{split}
	\end{equation}
	Noting that for the lift $X^{\mathbf{X}(s)}_n= \sum_{i=1}^n X_i(s) \mathbf{1}_{A^n_i}$ of the solution $\mathbf{X}(s)$ of equation \eqref{state_equation}, we have $\mu_{\mathbf{X}(s)} = (X_n^{\mathbf{X}(s)})_{\texttt{\#}}\mathcal{L}^1$, we see that $X^{\mathbf{X}(s)}_n$ is a solution of the previous equation, and therefore, by uniqueness, we have $\mathbb{P}$-a.s., for every $s\in [t,T]$, $X^{n}(s,\omega) = X^{\mathbf{X}(s)}_n(\omega)$ for almost every $\omega\in \Omega$. Thus, we have
	\begin{equation}
		\begin{split}
			&J(t,X^{\mathbf{x}}_n;a^n(\cdot))\\
			&= \mathbb{E} \left [ \int_t^T \left ( L_1(X^{\mathbf{X}(s)}_n) + L_2(a^n(s)) \right ) \mathrm{d}s + U_T(X^{\mathbf{X}(T)}_n) \right ]\\
			&= \mathbb{E} \left [ \int_t^T \left ( \int_{\Omega} l_1(X_n^{\mathbf{X}(s)}(\omega),(X_n^{\mathbf{X}(s)})_{\texttt{\#}}\mathcal{L}^1) \mathrm{d}\omega + \int_{\Omega} l_2(a^n(s,\omega)) \mathrm{d}\omega \right ) \mathrm{d}s + U_T(X^{\mathbf{X}(T)}_n) \right ]\\
			&= \mathbb{E} \left [ \int_t^T \left ( \frac{1}{n} \sum_{i=1}^n \left ( l_1(X_i(s),\mu_{\mathbf{X}(s)}) + l_2(a_i(s)) \right ) \right ) \mathrm{d}s + \mathcal{U}_T(\mu_{\mathbf{X}(T)}) \right ]\\
			&= J_n(t,\mathbf{x};\mathbf{a}(\cdot)),
		\end{split}
	\end{equation}
	which concludes the proof.
\end{proof}
From this result, we immediately derive the following inequality.
\begin{proposition}\label{proposition:V_n_leq_u_n}
	Let Assumptions \ref{Assumption_b_sigma_lipschitz} and \ref{Assumption_running_terminal_cost} be satisfied. Then, it holds
	\begin{equation}
		V_n(t,\mathbf{x}) \leq u_n(t,\mathbf{x})
	\end{equation}
	for all $(t,\mathbf{x}) \in [0,T] \times (\mathbb{R}^d)^n$.
\end{proposition}

Note that the proof of the inequality $V_n\leq u_n$ does not require any differentiability of $V$.

\subsection{\texorpdfstring{$V_n \geq u_n$}{Vn>un}}

We now discuss the opposite inequality. For the proof, we are going to show that $V_n$ is a viscosity supersolution of equation \eqref{intro:finite_dimensional_HJB}. The desired inequality then follows from the comparison principle. Note that for the proof, we need to assume that $DV$ is continuous. However, as it was pointed out in Remark \ref{remark_joint_continuity}, this holds if either the assumptions of Propositions \ref{Value_Function_Semiconcave_1} and \ref{Value_Function_Semiconvex_1} are satisfied, or the assumptions of Propositions \ref{Value_Function_Semiconcave_2} and \ref{Value_Function_Semiconvex_2} are satisfied.

\begin{proposition}\label{proposition:V_n_geq_u_n}
	Let Assumptions \ref{Assumption_b_sigma_lipschitz} and \ref{Assumption_running_terminal_cost} be satisfied. Moreover, let $DV$ be continuous on $[0,T]\times E$. Then, it holds
	\begin{equation}
		V_n(t,\mathbf{x}) \geq u_n(t,\mathbf{x})
	\end{equation}
	for all $(t,\mathbf{x}) \in [0,T] \times (\mathbb{R}^d)^n$.
\end{proposition}

\begin{proof}
	Let $\varphi_n\in C^{1,2}((0,T) \times (\mathbb{R}^d)^n)$ and let $V_n-\varphi_n$ have a global minimum at $(t_0,\mathbf{x}_0) \in (0,T) \times (\mathbb{R}^d)^n$. In order to simplify notation, let us assume without loss of generality that $\mathbf{x}_0 = \mathbf{0}$, $(V_n-\varphi_n)(t_0,\mathbf{0})=0$.
	
	Let $f_i:= \sqrt{n} \mathbf{1}_{A^n_i}$, $i=1,\dots,n$, and extend this orthonormal set in $L^2(\Omega)$ to an orthonormal basis denoted by $(f_i)_{i\in \mathbb{N}}$. Let $F_n \subset L^2(\Omega)$ be the subspace spanned by $f_i$, $i=1,\dots n$, and let $E_n := F_n \otimes \mathbb{R}^d \subset E$. Note that $E_n$ is $nd$-dimensional, where $d\in \mathbb{N}$ is fixed. Let $E_n^{\perp}$ denote its orthogonal complement, i.e., $E=E_n\oplus E_n^{\perp}$. Moreover, let $P_n$ denote the orthogonal projection onto $E_n$ and let $P_n^{\perp} := I - P_n$. Each element $X\in E$ can be written as
	\begin{equation}
		X = P_n X + P_n^{\perp} X = \sum_{i=1}^n y_i f_i + \sum_{i=n+1}^{\infty} y_i f_i = \sum_{i=1}^n \sum_{k=1}^d y_i^k f_i e_k + \sum_{i=n+1}^{\infty} \sum_{k=1}^d y_i^k f_i e_k,
	\end{equation}
	where $y_i = (y_i^1,\dots, y_i^d) \in \mathbb{R}^d$, $i\geq 1$. Let us denote $\mathbf{y}=(y_1,\dots,y_n) \in (\mathbb{R}^d)^n$ and $\bar{\mathbf{y}} = (y_i)_{i\geq n+1}$. 
	For $(t,X) \in [0,T]\times E$, note that
	\begin{equation}
		\begin{split}
			V(t,P_nX) = V\left ( t, \sum_{i=1}^n y_i f_i \right ) = V \left ( t, \sum_{i=1}^n \sqrt{n} y_i \mathbf{1}_{A^n_i} \right ) = V_n(t, \sqrt{n} \mathbf{y}).
		\end{split}
	\end{equation}
	In particular, for $\mathbf{x} \in (\mathbb{R}^d)^n$, and $X^{\mathbf{x}}_n = \sum_{i=1}^n x_i \mathbf{1}_{A^n_i}$, the coefficients $\mathbf{y}$ in the basis representation are given by $\mathbf{y}= \mathbf{x}/\sqrt{n}$ and $\bar{\mathbf{y}}=\mathbf{0}$. We note that for $X\in E_n$, we have almost everywhere $DV(t,X)(\omega)=\partial_\mu \mathcal{V}(t,X_{\texttt{\#}} \mathcal{L}^1)(X(\omega))$, thus $DV(t,X)\in E_n$. 
	
	Now, we define the function $\varphi:(0,T)\times E \to \mathbb{R}$ by
	\begin{equation}
		\varphi(t,X) := \varphi_n(t,\sqrt{n}\mathbf{y}).
	\end{equation}
	Let $\delta<\min(t_0, T-t_0)$. Consider the function
	\begin{equation}\label{test_function}
		(0,T)\times E \ni (t,X) \mapsto V(t,X) - \varphi(t,X) + \varepsilon((t-t_0)^2+\|P_n X\|_E^2) + \frac{\varepsilon}{\delta^2} \|P_n^{\perp}X\|_E^2.
	\end{equation}
	Note that for $(t,X)\in (0,T)\times E$ such that $\|(t-t_0,P_n X)\|_{\mathbb{R} \times E} = \delta$ and $\|P_n^{\perp}X\|_E \leq \delta^2$, we have
	\begin{equation}
		\begin{split}
			&V(t,X) - \varphi(t,X) + \varepsilon((t-t_0)^2+\|P_n X\|_E^2) + \frac{\varepsilon}{\delta^2} \|P^{\perp}_n X\|_E^2\\
			&= V_n(t,\sqrt{n}\mathbf{y}) - \varphi_n(t,\sqrt{n}\mathbf{y}) + \varepsilon((t-t_0)^2+\|P_n X\|_E^2)\\
			&\quad + \frac{\varepsilon}{\delta^2} \|P^{\perp}_n X\|_E^2 + V(t,X) - V_n(t,\sqrt{n}\mathbf{y})\\
			&\geq \varepsilon \delta^2 + V(t,X) - V(t,P_n X).
		\end{split}
	\end{equation}
	Since $DV$ is continuous and $P_n^{\perp}DV(t_0,0)=0$, we have
	\begin{equation}
		| V(t,X) - V(t,P_n X) | \leq \rho(\|(t-t_0,X)\|_{\mathbb{R}\times E}) \|P_n^{\perp} X \|_E \leq \rho(\delta) \delta^2
	\end{equation}
	for some modulus of continuity $\rho$. Thus, we have
	\begin{equation}
		V(t,X) - \varphi(t,X) + \varepsilon((t-t_0)^2+\|P_n X\|_E^2) + \frac{\varepsilon}{\delta^2} \|P^{\perp}_n X\|_E^2 \geq \varepsilon \delta^2 - \rho(\delta) \delta^2.
	\end{equation}
	Now let us consider $(t,X)\in (0,T)\times E$ such that $\|(t-t_0,P_n X)\|_{\mathbb{R}\times E} \leq \delta$ and $\|P_n^{\perp} X \|_E=\delta^2$. We have
	\begin{equation}
		\begin{split}
			&V(t,X) - \varphi(t,X) + \varepsilon((t-t_0)^2+\|P_n X\|_E^2) + \frac{\varepsilon}{\delta^2} \|P^{\perp}_n X\|_E^2\\
			&= V_n(t,\sqrt{n}\mathbf{y}) - \varphi_n(t,\sqrt{n}\mathbf{y}) + \varepsilon((t-t_0)^2+\|P_n X\|_E^2)\\
			&\quad + \frac{\varepsilon}{\delta^2} \|P^{\perp}_n X\|_E^2 + V(t,X) - V_n(t,\sqrt{n}\mathbf{y})\\
			&\geq \varepsilon\delta^2 + V(t,X) - V(t,P_n X)
			\geq \varepsilon\delta^2 - \rho(\delta)\delta^2.
		\end{split}
	\end{equation}
	Thus, for any $0<\varepsilon<1$, there is a positive
	$\delta=\delta_\varepsilon<\varepsilon$ sufficiently small such that we have
	\begin{equation}
		V(t,X) - \varphi(t,X) + \varepsilon((t-t_0)^2+\|P_n X\|_E^2) + \frac{\varepsilon}{\delta^2} \|P^{\perp}_n X\|_E^2 \geq \frac{\varepsilon\delta^2}{2}=:\gamma 
	\end{equation}
	on the boundary of the set
	\begin{equation}
		K_{\delta} := \{ (t,X) \in \mathbb{R} \times E : \|(t-t_0,P_n X)\|_{\mathbb{R}\times E} \leq \delta, \|P_n^{\perp} X \|_E \leq \delta^2 \}.
	\end{equation}
	By the Ekeland--Lebourg Theorem, see e.g. \cite[Theorem 3.25]{fabbri_gozzi_swiech_2017}, we know that there are $a\in \mathbb{R}$ and $Z\in E$ with $\|(a,Z)\|_{\mathbb{R}\times E} < \gamma/(4\delta)$ such that the function
	\begin{equation}
		(t,X) \mapsto V(t,X) - \varphi(t,X) + \varepsilon((t-t_0)^2+\|P_n X\|_E^2) + \frac{\varepsilon}{\delta^2} \|P^{\perp}_n X\|_E^2 + at + \langle Z,X \rangle_E
	\end{equation}
	attains a strict minimum over $K_{\delta}$ at some point $(t_{\delta},X_{\delta}) \in K_{\delta}$. Note that
	\begin{equation}
		|at + \langle Z,X \rangle_E| \leq \|(a,Z)\|_{\mathbb{R}\times E} \|(t,X)\|_{\mathbb{R}\times E} < \frac{\gamma}{2}
	\end{equation}
	for all $(t,X)\in K_{\delta}$. Thus,
	\begin{equation}
		V(t,X) - \varphi(t,X) + \varepsilon((t-t_0)^2+\|P_n X\|_E^2) + \frac{\varepsilon}{\delta^2} \|P^{\perp}_n X\|_E^2 + at + \langle Z,X \rangle_E \geq \frac{\gamma}{2}
	\end{equation}
	for all $(t,X) \in \partial K_{\delta}$. Since $(V - \varphi)(t_0,0) = 0$, it follows that $(t_{\delta},X_{\delta})$ must be in the interior of $K_{\delta}$. Therefore, since $V$ is a viscosity solution of \eqref{intro:lifted_HJB}, we have
	\begin{equation}\label{supersolution_inequality}
		\begin{split}
			&\partial_t \varphi(t_{\delta},X_{\delta}) - 2\varepsilon (t_{\delta}-t_0) - a\\
			&+ \frac12 \sum_{m=1}^{d^{\prime}} \langle ( D^2 \varphi(t_{\delta},X_{\delta}) - 2 \varepsilon P_n - 2\frac{\varepsilon}{\delta^2} P_n^{\perp}) \Sigma(X_{\delta}) e^{\prime}_m, \Sigma(X_{\delta}) e^{\prime}_m \rangle_E\\
			& - \tilde{H}(X_{\delta},(X_{\delta})_{\texttt{\#}} \mathcal{L}^1, D\varphi(t_{\delta},X_{\delta}) - 2 \varepsilon P_n X_{\delta} - 2 \frac{\varepsilon}{\delta^2} P_n^{\perp} X_{\delta} - Z) \leq 0.
		\end{split}
	\end{equation}
	
	Now, note that for any $X \in E$, we have
	\begin{equation}
		\Sigma(P_n X)(\omega) = \sigma \left ( \sum_{i=1}^n y_i f_i(\omega),(P_n X)_{\texttt{\#}} \mathcal{L}^1 \right ) = \sum_{i=1}^n \sigma \left ( \sqrt{n} y_i,(P_n X)_{\texttt{\#}} \mathcal{L}^1 \right ) \mathbf{1}_{A^n_i}(\omega),
	\end{equation}
	which shows that $\Sigma(P_n X) e^{\prime}_m \in E_n$. Therefore, we have
	\begin{equation}
		\begin{split}
			&| \langle \frac{\varepsilon}{\delta^2} P_n^{\perp} \Sigma(X_{\delta})e^{\prime}_m, \Sigma(X_{\delta})e^{\prime}_m \rangle_E |\\
			&= | \langle \frac{\varepsilon}{\delta^2} P_n^{\perp} (\Sigma(X_{\delta})- \Sigma(P_n X_{\delta})) e^{\prime}_m, \Sigma(X_{\delta})e^{\prime}_m\rangle_E |\\
			&\leq \frac{C\varepsilon}{\delta^2} \| \Sigma(X_{\delta}) - \Sigma(P_n X_{\delta}) \|_{L_2(\mathbb{R}^{d^{\prime}},E)} \leq C \frac{\varepsilon}{\delta^2} \| P^{\perp}_n X_{\delta} \|_E\leq C\varepsilon.
		\end{split}
	\end{equation}
	Hence, taking the limit $\varepsilon \to 0$ in inequality \eqref{supersolution_inequality}, we obtain
	\begin{equation}\label{supersolution_inequality_2}
		\partial_t \varphi(t_0,0) + \frac12 \sum_{m=1}^{d^{\prime}} \langle D^2 \varphi(t_0,0) \Sigma(0) e^{\prime}_m, \Sigma(0) e^{\prime}_m \rangle_E - \tilde{H}(0,0_{\texttt{\#}} \mathcal{L}^1,D\varphi(t_0,0)) \leq 0.
	\end{equation}
	Recalling that $(e_k)_{k=1,\dots,d}$ denotes the standard basis of $\mathbb{R}^d$, we find that for any $(t,X)\in (0,T]\times E$, we have
	\begin{equation}\label{derivative_varphi}
		D\varphi(t,X) = \sum_{i=1}^n \sum_{k=1}^d \sqrt{n} D_{x_i^k} \varphi_n( t, \sqrt{n} \mathbf{y} ) f_i \otimes e_k
	\end{equation}
	and
	\begin{equation}
		D^2 \varphi(t,X) = \sum_{i,j=1}^n \sum_{k,l=1}^d n D_{x_i^k x_j^l} \varphi_n(t,\sqrt{n} \mathbf{y}) \langle f_i\otimes e_k, \cdot \rangle_E \langle f_j\otimes e_l ,\cdot \rangle_E.
	\end{equation}
	Therefore, we obtain for the second order term in \eqref{supersolution_inequality_2}
	\begin{equation}
		\begin{split}
			&\sum_{m=1}^{d^{\prime}} \langle D^2 \varphi(t_0,0) \Sigma(0) e^{\prime}_m, \Sigma(0) e^{\prime}_m \rangle_E\\
			&= \sum_{m=1}^{d^{\prime}} \sum_{i,j=1}^n \sum_{k,l=1}^d n D_{x^k_ix^l_j} \varphi_n(t_0,\mathbf{0}) \langle f_i\otimes e_k, \Sigma(0) e^{\prime}_m \rangle_E \langle f_j\otimes e_l, \Sigma(0) e^{\prime}_m \rangle_E.
		\end{split}
	\end{equation}
	Noticing that
	\begin{equation}
		\langle f_i\otimes e_k, \Sigma(0) e^{\prime}_m \rangle_E = \frac{1}{\sqrt{n}} \sigma_{km}(\mathbf{0},\mu_{\mathbf{0}}),
	\end{equation}
	we obtain
	\begin{equation}
		\begin{split}
			\sum_{m=1}^{d^{\prime}} \langle D^2 \varphi(t_0,0) \Sigma(0) e^{\prime}_m, \Sigma(0) e^{\prime}_m \rangle_E &= \sum_{m=1}^{d^{\prime}} \sum_{i,j=1}^n \sum_{k,l=1}^d D_{x^k_ix^l_j} \varphi_n(t_0,\mathbf{0}) \sigma_{km}(\mathbf{0},\mu_{\mathbf{0}}) \sigma_{lm}(\mathbf{0},\mu_{\mathbf{0}})\\
			&= \text{Tr} (A_n(\mathbf{0},\mu_{\mathbf{0}}) D^2 \varphi_n(t_0,\mathbf{0})).
		\end{split}
	\end{equation}
	Recalling the definition of $\tilde{H}$ given in \eqref{definition_H_tilde} and using \eqref{derivative_varphi}, a direct computation shows that
	\begin{equation}
		\tilde{H}(0,0_{\texttt{\#}} \mathcal{L}^1,D\varphi(t_0,0)) = \frac{1}{n} \sum_{i=1}^n H(0,\mu_{\mathbf{0}},nD_{x_i}\varphi_n(t_0,\mathbf{0})).
	\end{equation}
	
	Thus, we obtain from \eqref{supersolution_inequality_2} 
	\begin{equation}
		\partial_t \varphi_n(t_0,\mathbf{0}) + \frac12 \text{Tr} (A_n(\mathbf{0},\mu_{\mathbf{0}}) D^2 \varphi_n(t_0,\mathbf{0})) - \frac{1}{n} \sum_{i=1}^n H(0,\mu_{\mathbf{0}},nD_{x_i}\varphi_n(t_0,\mathbf{0})) \leq 0,
	\end{equation}
	which shows that $V_n$ is a viscosity supersolution of equation \eqref{intro:finite_dimensional_HJB}. Since $u_n$ is a viscosity solution of \eqref{intro:finite_dimensional_HJB}, the result now follows from the comparison principle in the class sub/supersolutions which are Lipschitz continuous in the spatial variable.
\end{proof}

Combining Propositions \ref{proposition:V_n_geq_u_n} and \ref{proposition:V_n_leq_u_n}, we obtain the following result.

\begin{theorem}\label{theorem:projection_C11}
	Let Assumptions \ref{Assumption_b_sigma_lipschitz} and \ref{Assumption_running_terminal_cost} be satisfied. Moreover, let $DV$ be continuous on $[0,T]\times E$. Then, it holds
	\begin{equation}
		V_n(t,\mathbf{x}) = u_n(t,\mathbf{x})
	\end{equation}
	for all $(t,\mathbf{x}) \in [0,T] \times (\mathbb{R}^d)^n$.
\end{theorem}

\section{Lifting and Projection of Optimal Controls}\label{section:feedbacks}

In this section, we are going to show that if $DV$ is continuous (or if $V(t,\cdot)\in C^{1,1}(E)$ for the case of optimal feedback controls), optimal (feedback) controls of the finite dimensional control problems correspond to optimal (feedback) controls of the lifted infinite dimensional control problem started at the corresponding initial condition. Conversely, we will show that piecewise constant optimal (feedback) controls of the infinite dimensional control problem project onto optimal (feedback) controls of the finite dimensional control problems.
Throughout this section, we assume that the assumptions of Theorem \ref{theorem:projection_C11} are satisfied.

\subsection{Lifting of Optimal Controls}

Let $n\geq 1$ and let $\mathbf{a}^{\ast}(\cdot) = (a_1^{\ast}(\cdot),\dots,a_n^{\ast}(\cdot))$ be an optimal control of the finite dimensional problem, i.e.,
\begin{equation}
	\begin{split}
		u_n(t,\mathbf{x}) &= J_n(t,\mathbf{x};\mathbf{a}^{\ast}(\cdot))\\
		&:= \mathbb{E} \left [ \int_t^T \left ( \frac1n \sum_{i=1}^n \left ( l_1(X^{\ast}_i(s),\mu_{\mathbf{X}^{\ast}(s)}) + l_2(a^{\ast}_i(s)) \right ) \right ) \mathrm{d}s + \mathcal{U}_T(\mu_{\mathbf{X}^{\ast}(T)}) \right ],
	\end{split}
\end{equation}
where $\mathbf{X}^{\ast}(s) = (X_1^{\ast}(s),\dots,X_n^{\ast}(s))$ is the solution of the system of SDEs
\begin{equation}
	\begin{cases}
		\mathrm{d}X^{\ast}_i(s) = [-a^{\ast}_i(s) + b(X^{\ast}_i(s),\mu_{\mathbf{X}^{\ast}(s)})] \mathrm{d}s + \sigma(X^{\ast}_i(s),\mu_{\mathbf{X}^{\ast}(s)}) \mathrm{d}W(s)\\
		X^{\ast}_i(t) = x_i\in \mathbb{R}^d,
	\end{cases}
\end{equation}
$i=1,\dots,n$. By Lemma \ref{lemma_subsolution}, we obtain
\begin{equation}
	u_n(t,\mathbf{x}) = J_n(t,\mathbf{x};\mathbf{a}^{\ast}(\cdot)) = J(t,X^{\mathbf{x}}_n ; a^{\ast,n}(\cdot)),
\end{equation}
where $a^{\ast,n}(\cdot) = \sum_{i=1}^n a_i(\cdot) \mathbf{1}_{A^n_i}$. Moreover, due to Theorem \ref{theorem:projection_C11}, we have
\begin{equation}
	u_n(t,\mathbf{x}) = V_n(t,\mathbf{x}) = V\left (t, \sum_{i=1}^n x_i \mathbf{1}_{A^n_i} \right ).
\end{equation}
Therefore, the optimal control $\mathbf{a}^{\ast}(\cdot)$ lifts to an optimal control $a^{\ast,n}(\cdot)$ of the infinite dimensional problem started at $(t,\sum_{i=1}^n x_i \mathbf{1}_{A^n_i})$.

\begin{remark}\label{remark:lifting_optimal_feedbacks}
	If in addition to the assumptions of Theorem \ref{theorem:projection_C11}, $\nu >0$ in Assumption \ref{Assumption_running_terminal_cost}(ii), $V(t,\cdot)\in C^{1,1}(E)$ for every $t\in [0,T]$, and the semiconcavity and semiconvexity constants of $V(t,\cdot)$ are independent of $t\in [0,T]$, then it was shown in \cite[Theorem 4.9, Example 4.10]{defeo_swiech_wessels_2023} and \cite[Proposition 7.1]{mayorga_swiech_2023} that $\mathbf{a}^{\ast}(\cdot) = (a^{\ast}_1(\cdot),\dots,a^{\ast}_n(\cdot))$ defined by
	\begin{equation}
		a_i^{\ast}(s):=(Dl_2)^{-1}(nD_{x_i}u_n(s,{\mathbf{X}^{\ast}(s)}))
	\end{equation}
	is an optimal feedback control for the finite dimensional control problem. Here $\mathbf{X}^{\ast}(s)=(X_1^{\ast}(s),\dots,X_n^{\ast}(s))$ is the solution of the system
	\begin{equation}
		\begin{cases}
			\mathrm{d}X^{\ast}_i(s) = [-(Dl_2)^{-1}(nD_{x_i}u_n(s,{\mathbf{X}^{\ast}(s)})) + b(X_i^{\ast}(s),\mu_{\mathbf{X}^{\ast}(s)})] \mathrm{d}s\\
			\qquad\qquad + \sigma(X^{\ast}_i(s),\mu_{\mathbf{X}^{\ast}(s)}) \mathrm{d}W(s)\\
			X^{\ast}_i(t) = x_i,
		\end{cases}
	\end{equation}
	$i=1,...,n$. Now, consider the lifted control
	\begin{equation}\label{lifted_control}
		a^{\ast,n}(s) = \sum_{i=1}^n a^{\ast}_i(s) \mathbf{1}_{A^n_i}.
	\end{equation}
	As in the proof of Lemma \ref{lemma_subsolution}, we see that the solution $X^{\ast,n}(\cdot)$ of the lifted state equation with initial condition $X^{\mathbf{x}}_n$ and control $a^{\ast,n}(\cdot)$ coincides with the lift of the solution $\mathbf{X}^{\ast}(\cdot)$ of the finite dimensional state equation with control $\mathbf{a}^{\ast}(\cdot)$. In particular, $X^{\ast,n}(\cdot)$ is piecewise constant in $\Omega$. Therefore, we observe that
	\begin{equation}
		\begin{split}
			a^{\ast,n}(s) &=\sum_{i=1}^n(Dl_2)^{-1}(nD_{x_i}u_n(s,{\mathbf{X}^{\ast}(s)}))\mathbf{1}_{A^n_i}\\
			&= \sum_{i=1}^n(Dl_2)^{-1}\left(n\int_{A^n_i}DV(s, X^{\ast,n}(s))(\omega) \mathrm{d}\omega\right)\mathbf{1}_{A^n_i}\\
			&=\sum_{i=1}^n (Dl_2)^{-1} \left(n\int_{A^n_i}\partial_\mu \mathcal{V}(s, X^{\ast,n}_{\texttt{\#}} \mathcal{L}^1)( X^{\ast,n}(s)(\omega))
			\mathrm{d}\omega \right)\mathbf{1}_{A^n_i}\\
			& =\sum_{i=1}^n(Dl_2)^{-1}\left(\partial _\mu\mathcal{V}(s, X^{\ast,n}_{\texttt{\#}} \mathcal{L}^1)( X^{\ast,n}(s)) \right)\mathbf{1}_{A^n_i} \\
			& =(DL_2)^{-1}(DV(s, X^{\ast,n}(s))).
		\end{split}
	\end{equation}
	In the last step, we used the fact that $(DL_2)^{-1}(Y)(\omega)=(Dl_2)^{-1}(Y(\omega))$, for $Y\in E$. Applying the results of \cite[Theorem 4.10, Example 4.11]{defeo_swiech_wessels_2023} and \cite[Proposition 7.1]{mayorga_swiech_2023} to the infinite dimensional control problem, we obtain that $(DL_2)^{-1}(DV(s, X^{\ast,n}(s)))$ is an optimal feedback control of the lifted problem started at $(t,\sum_{i=1}^n x_i \mathbf{1}_{A^n_i})$. Therefore, we have shown that the optimal feedback control of the finite dimensional problem lifts to the optimal feedback control of the corresponding lifted control problem.
\end{remark}

\subsection{Projection of Optimal Controls}

Now, let $a^{\ast,n}(\cdot)$ be a piecewise constant optimal control of the infinite dimensional problem started at $(t,X^{\mathbf{x}}_n)$, $\mathbf{x}\in (\mathbb{R}^d)^n$. More precisely, let $a^{\ast,n}(s) = \sum_{i=1}^n a_i^{\ast}(s) \mathbf{1}_{A^n_i}$ for some $a_i^{\ast}(\cdot)$ taking values in $\mathbb{R}^d$, $i=1,\dots,n$, and
\begin{equation}
	V(t,X^{\mathbf{x}}_n) = J(t,X^{\mathbf{x}}_n;a^{\ast,n}(\cdot)) = \mathbb{E} \left [ \int_t^T \left ( L_1(X^{\ast}(s)) + L_2(a^{\ast,n}(s)) \right ) \mathrm{d}s + U_T(X^{\ast}(T)) \right ],
\end{equation}
where
\begin{equation}
	\begin{cases}
		\mathrm{d}X^{\ast}(s) = [-a^{\ast,n}(s) + B(X^{\ast}(s))] \mathrm{d}s + \Sigma(X^{\ast}(s)) \mathrm{d}W(s)\\
		X^{\ast}(t) = X^{\mathbf{x}}_n \in E_n.
	\end{cases}
\end{equation}
Note that in this case, $X^{\ast}(s) \in E_n$ for all $s\in [t,T]$; recall \eqref{E_n} for the definition of $E_n$. Let $\mathbf{a}^{\ast}(\cdot) = (a_1^{\ast}(\cdot),\dots,a_n^{\ast}(\cdot))$ be the finite dimensional projection of $a^{\ast,n}(\cdot)$. Then, we obtain from Lemma \ref{lemma_subsolution} that
\begin{equation}
	V(t,X^{\mathbf{x}}_n) = J(t,X^{\mathbf{x}}_n;a^{\ast,n}(\cdot)) = J_n(t,\mathbf{x};\mathbf{a}^{\ast}(\cdot)).
\end{equation}
Moreover, by Theorem \ref{theorem:projection_C11}, we have
\begin{equation}
	V(t,X^{\mathbf{x}}_n) = V_n(t,\mathbf{x}) = u_n(t,\mathbf{x}).
\end{equation}
Therefore, the piecewise constant optimal control $a^{\ast,n}(\cdot)$ of the infinite dimensional control problem started at $(t,X^{\mathbf{x}}_n)$ projects to an optimal control of the finite dimensional control problem started at $(t,\mathbf{x})$.

\begin{remark}
	As mentioned in Remark \ref{remark:lifting_optimal_feedbacks}, under its stated assumptions, an optimal control for the infinite dimensional control problem is given by
	\begin{equation}
		a^{\ast}(s)=(DL_2)^{-1}(DV(s,X^{\ast}(s))),
	\end{equation}
	where $X^{\ast}(\cdot)$ is the solution of the infinite dimensional state equation with control $a^{\ast}(\cdot)$. Note that $a^{\ast}(\cdot)$ is piecewise constant if $X^{\ast}(s)$ is piecewise constant. Indeed, for every $X\in E$, we have
	\begin{equation}
		DV(t,X)(\omega) = \partial_{\mu} \mathcal{V}(t,X_{\texttt{\#}}\mathcal{L}^1)(X(\omega)).
	\end{equation}
	Thus, $(DL_2)^{-1}(DV(s,X)) \in E_n$ if $X\in E_n$. Therefore, in this case the same calculation as in Remark \ref{remark:lifting_optimal_feedbacks} shows that the optimal feedback control projects onto the optimal feedback control of the finite dimensional control problem.
\end{remark}

\section{Projection of \texorpdfstring{$V$}{V} without Differentiability of \texorpdfstring{$V$}{V} in the Linear Case}\label{section:projection_no_C11}

In Section \ref{section:projection_C11}, we established the projection of $V$ onto $u_n$ when $DV$ was continuous. In this section, we are going to relax the assumptions on the coefficients of the cost functional. Therefore, we may lose the above regularity of the value function. Nevertheless, we are still able to prove the projection property in the case of a linear state equation. This will be done using approximations in the Wasserstein space.

There are different ways to regularize functions in the Wasserstein space, see e.g. \mbox{\cite[Theorems 4.2, 4.4]{cosso_martini_2023}}, \cite[Section 4]{daudin_delarue_jackson_2024}, \cite[Lemmas 2.11, 4.2]{daudin_jackson_seeger_2023}. We are going to use the method introduced in \cite{cosso_martini_2023}.

\subsection{Approximation in the Wasserstein Space}

We are going to approximate a convex and Lipschitz continuous function $\phi:\mathbb{R}^d \times \mathcal{P}_2(\mathbb{R}^d) \to \mathbb{R}$ by functions whose lifts are in $C^{1,1}(\mathbb{R}^d\times E)$. The method uses a similar approximation as the one used in \cite[Theorems 4.2]{cosso_martini_2023}. However, since the approximation in \cite[Theorems 4.2]{cosso_martini_2023} does not preserve the convexity of the lift of $\phi$, we need to make some modifications. Moreover, we will show that the approximation preserves some other properties of $\phi$.

\begin{theorem}\label{Theorem_Approximation_Wasserstein_Space}
	Let $r\in[1,2)$, and let $\phi:\mathbb{R}^d \times \mathcal{P}_2(\mathbb{R}^d) \to \mathbb{R}$ be Lipschitz continuous with respect to the $|\,\cdot\,|\times d_r(\cdot,\cdot)$-distance. Then, there is an approximating sequence $(\phi_k)_{k\in \mathbb{N}}$ in $C^{\infty}(\mathbb{R}^d \times \mathcal{P}_2(\mathbb{R}^d))$ with lifts $\tilde{\phi}_k:\mathbb{R}^d \times E \to \mathbb{R}$ defined by $\tilde{\phi}_k(x,X) := \phi_k(x,X_{\texttt{\#}} \mathcal{L}^1)$ satisfying the following properties:
	\begin{enumerate}[label=(\roman*)]
		\item For each $k\in \mathbb{N}$, $\phi_k$ is Lipschitz continuous with respect to the $|\,\cdot\,|\times d_r(\cdot,\cdot)$-distance with Lipschitz constant independent of $k$;
		\item The sequence $(\phi_k)_{k\in\mathbb{N}}$ converges to $\phi$ uniformly on bounded subsets of $\mathbb{R}^d \times \mathcal{P}_2(\mathbb{R}^d)$, i.e.
		\begin{equation}\label{approximation_property_uniform}
			\lim_{k\to\infty} \sup \left \{ \left | \phi_k(x,\mu) - \phi(x,\mu) \right | : (x,\mu) \in B_R(0) \times \mathfrak{M}^2_R \right \}  = 0
		\end{equation}
		for every $R>0$;
		\item $\tilde{\phi}_k \in C^{1,1}(\mathbb{R}^d \times E)$ for each $k\in\mathbb{N}$;
		\item If the lift of $\phi$ is convex, then the lift of $\phi_k$ is convex for each $k\in\mathbb{N}$.
	\end{enumerate} 
\end{theorem}

\begin{proof}
	\textit{Construction of $\phi_k$:} Let $\eta:\mathbb{R}^d \to \mathbb{R}$ denote the standard mollifier on $\mathbb{R}^d$, i.e.,
	\begin{equation}
		\begin{split}
			\eta(x) := \begin{cases}
				C \exp \left ( \frac{1}{|x|^2-1} \right )& \text{if } |x|<1\\
				0& \text{if } |x| \geq 1,
			\end{cases}
		\end{split}
	\end{equation}
	where $C>0$ is chosen such that $\int_{\mathbb{R}^d} \eta \mathrm{d}x =1$. For $\varepsilon>0$, we set
	\begin{equation}
		\eta_{\varepsilon}(x) := \frac{1}{\varepsilon^d} \eta\left ( \frac{x}{\varepsilon} \right ),
	\end{equation}
	and for $N\in\mathbb{N}$, we define $h_{N,\varepsilon} : \mathbb{R}^d \times (\mathbb{R}^{d})^N \to \mathbb{R}$ as
	\begin{equation}
		\begin{split}
			&h_{N,\varepsilon}(x_0,x_1,\dots,x_N) \\
			&:= \int_{\mathbb{R}^{d(N+1)}} \phi \left ( x_0-y_0,\frac{1}{N} \sum_{i=1}^N \delta_{x_i-y_i} \right ) \eta_{\varepsilon}(y_0) \cdots \eta_{\varepsilon}(y_N) \mathrm{d}y_0 \cdots \mathrm{d}y_N\\
			&= \int_{\mathbb{R}^{d(N+1)}} \phi \left ( z_0,\frac{1}{N} \sum_{i=1}^N \delta_{z_i} \right ) \eta_{\varepsilon}(x_0-z_0) \cdots \eta_{\varepsilon}(x_N-z_N) \mathrm{d}z_0 \cdots \mathrm{d}z_N,
		\end{split}
	\end{equation}
	where $x_i \in \mathbb{R}^d$, and $z_i:=x_i-y_i$, $i=0,\dots,N$. 
	Moreover, let $\psi_{N,{\varepsilon}}: \mathbb{R}^d \times \mathcal{P}_2(\mathbb{R}^d) \to \mathbb{R}$ be given by
	\begin{equation}
		\psi_{N,{\varepsilon}}(x,\mu) := \tilde{\mathbb{E}} \left [ h_{N,{\varepsilon}}(x,\tilde{X}_1,\dots,\tilde{X}_N) \right ], \quad (x,\mu)\in \mathbb{R}^d \times \mathcal{P}_2(\mathbb{R}^d)
	\end{equation}
	where $\tilde{X}_1,\dots, \tilde{X}_N$ are i.i.d. random variables with law $\mu$ defined on some probability space $(\tilde{\Omega},\tilde{\mathcal{F}},\tilde{\mathbb{P}})$. Now we define for $k\in \mathbb{N}$
	\begin{equation}
		\phi_k(x,\mu) := \psi_{k,1/k}(x,\mu) = \tilde{\mathbb{E}} \left [ h_{k,1/k}(x,\tilde{X}_1,\dots,\tilde{X}_k) \right ].
	\end{equation} 
	We note that $h_{k,1/k}$ is of class $C^{\infty}(\mathbb{R}^{d(k+1)})$, and its partial derivatives are given by
	\begin{equation}
		\begin{split}
			&D_{x_i} h_{k,1/k}(x_0,x_1,\dots,x_k)\\
			&= \int_{\mathbb{R}^{d(k+1)}} \phi \left ( z_0,\frac{1}{k} \sum_{i=1}^k \delta_{z_i} \right ) \eta_{\frac1k}(x_0-z_0) \cdots \eta_{\frac1k}(x_k-z_k) \frac{2k^{2}(x_i-z_i)}{(k^2|x_i-z_i|^2-1)^2} \mathrm{d}z_0 \cdots \mathrm{d}z_k
		\end{split}
	\end{equation}
	$i=0,\dots,n$. Moreover, $h_{k,1/k}$ is Lipschitz continuous. Indeed, we have
	\begin{equation}\label{Lipschitz_h_nm}
		\begin{split}
			&| h_{k,1/k}(x_0,\dots,x_k) - h_{k,1/k}(\tilde{x}_0,\dots,\tilde{x}_k) |\\
			&\leq \int_{\mathbb{R}^{d(k+1)}} \left | \phi \left ( x_0-y_0,\frac{1}{k} \sum_{i=1}^k \delta_{x_i-y_i} \right ) - \phi \left ( \tilde{x}_0-y_0,\frac{1}{k} \sum_{i=1}^k \delta_{\tilde{x}_i-y_i} \right ) \right |\\
			&\qquad\qquad\qquad\qquad\qquad\qquad\qquad\qquad\qquad\qquad\qquad\qquad \times \eta_{\frac1k}(y_0) \cdots \eta_{\frac1k}(y_k) \mathrm{d}y_0 \cdots \mathrm{d}y_k\\
			&\leq C \int_{\mathbb{R}^{d(k+1)}} \left ( |x_0-\tilde{x}_0| + d_r\left (\frac{1}{k} \sum_{i=1}^k \delta_{x_i-y_i}, \frac1k \sum_{i=1}^k \delta_{\tilde{x}_i-y_i}\right ) \right ) \eta_{\frac1k}(y_0) \cdots \eta_{\frac1k}(y_k) \mathrm{d}y_0 \cdots \mathrm{d}y_k\\
			&\leq C \left ( |x_0-\tilde{x}_0| + |\mathbf{x}- \tilde{\mathbf{x}}|_r \right ) \leq \frac{C}{\sqrt{k}}(|x_0-\tilde{x}_0| + |\mathbf{x} - \tilde{\mathbf{x}}|),
		\end{split}
	\end{equation}
	where the constant $C$ only depends on the Lipschitz constant of $\phi$. Therefore, $D_{x_i} h_{k,1/k}$, $i=0,\dots,k$, is bounded and, in order to differentiate $\phi_k$, we can differentiate under the expectation. Thus, we have
	\begin{equation}
		D_x \phi_{k}(x,\mu) = \tilde{\mathbb{E}} \left [ D_x h_{k,1/k}(x,\tilde{X}_1,\dots,\tilde{X}_k) \right ]
	\end{equation}
	and, from \cite[Equation (5.37)]{carmona_delarue_2018} it follows that
	\begin{equation}
		\partial_{\mu} \phi_{k}(x,\mu)(x_1) = k \tilde{\mathbb{E}} \left [ D_{x_1} h_{k,1/k}(x,x_1,\tilde{X}_2,\dots,\tilde{X}_k) \right ].
	\end{equation}
	The same argument as in equation \eqref{Lipschitz_h_nm} shows that $D_{x_i} h_{k,1/k}$ is Lipschitz continuous. Therefore, we have
	\begin{equation}
		D_{x_1} \partial_{\mu} \phi_k(x,\mu)(x_1) = k \tilde{\mathbb{E}} \left [ D_{x_1x_1}^2 h_{k,1/k}(x,x_1,\tilde{X}_2,\dots,\tilde{X}_k) \right ]
	\end{equation}
	and
	\begin{equation}
		\partial_{\mu}^2 \phi_k(x,\mu)(x_1,x_2) = k (k-1) \tilde{\mathbb{E}} \left [ D^2_{x_2x_1} h_{k,1/k}(x,x_1,x_2,\tilde{X}_3,\dots,\tilde{X}_k) \right ].
	\end{equation}
	Iterating these arguments shows that $\phi_k\in C^{\infty}(\mathbb{R}^d\times \mathcal{P}_2(\mathbb{R}^d))$ for each $k\in \mathbb{N}$. Moreover, since all the derivatives of $h_{k,1/k}$ are bounded, the derivatives of $\phi_k$ are also bounded.
	
	\textit{Proof of (i): Lipschitz continuity.} Due to the Lipschitz continuity of $\phi$, we have
	\begin{equation}
		\begin{split}
			&|\phi_k(x,\mu) - \phi_k(y,\beta)|\\
			&\leq \int_{\mathbb{R}^{d(k+1)}} \tilde{\mathbb{E}} \left [ \left | \phi\left ( x-y_0, \frac{1}{k} \sum_{i=1}^k \delta_{\tilde{X}_i-y_i} \right ) - \phi\left ( y-y_0, \frac{1}{k} \sum_{i=1}^k \delta_{\tilde{Y}_i-y_i} \right ) \right | \right ]\\
			&\qquad\qquad\qquad\qquad\qquad\qquad\qquad\qquad\qquad\qquad \times \eta_{\frac1k}(y_0) \cdots \eta_{\frac1k}(y_k) \mathrm{d}y_0 \cdots \mathrm{d}y_k\\
			&\leq \int_{\mathbb{R}^{d(k+1)}} C\left ( |x-y| + \tilde{\mathbb{E}} \left [ d_r \left ( \frac{1}{k} \sum_{i=1}^k \delta_{\tilde{X}_i-y_i}, \frac{1}{k} \sum_{i=1}^k \delta_{\tilde{Y}_i-y_i}\right ) \right ] \right )\\
			&\qquad\qquad\qquad\qquad\qquad\qquad\qquad\qquad\qquad\qquad \times \eta_{\frac1k}(y_0) \cdots \eta_{\frac1k}(y_k) \mathrm{d}y_0 \cdots \mathrm{d}y_k,
		\end{split}
	\end{equation}
	where $\tilde{X}_1,\dots,\tilde{X}_k$ are i.i.d. random variables with law $\mu$, and $\tilde{Y}_1,\dots,\tilde{Y}_k$, are i.i.d. random variables with law $\beta$. Next, we observe that
	\begin{equation}
		\begin{split}
			&\tilde{\mathbb{E}} \left [ d_r \left ( \frac{1}{k} \sum_{i=1}^k \delta_{\tilde{X}_i-y_i}, \frac{1}{k} \sum_{i=1}^k \delta_{\tilde{Y}_i-y_i} \right ) \right ] \leq \tilde{\mathbb{E}} \left [ |\tilde{X} - \tilde{Y} |_r \right ] = \frac{1}{k^{1/r}} \tilde{\mathbb{E}} \left [ \left ( \sum_{i=1}^k |\tilde{X}_i - \tilde{Y}_i|^r \right )^{\frac{1}{r}} \right ]\\
			&\leq \frac{1}{k^{1/r}} \left ( \tilde{\mathbb{E}} \left [ \sum_{i=1}^k |\tilde{X}_i - \tilde{Y}_i|^r \right ] \right )^{\frac{1}{r}} = \left ( \tilde{\mathbb{E}} \left [ |\tilde{X}_1 - \tilde{Y}_1|^r \right ] \right )^{\frac{1}{r}}.
		\end{split}
	\end{equation}
	Since this hold for all random variables $\tilde{X}_1$ and $\tilde{Y}_1$ with laws $\mu$ and $\beta$, respectively, not necessarily independent of each other, we have due to \eqref{definition_wasserstein_distance}
	\begin{equation}
		\tilde{\mathbb{E}} \left [ d_r \left ( \frac{1}{k} \sum_{i=1}^k \delta_{\tilde{X}_i-y_i}, \frac{1}{k} \sum_{i=1}^k \delta_{\tilde{Y}_i-y_i} \right ) \right ] \leq d_r(\mu,\beta).
	\end{equation}
	Altogether, we obtain
	\begin{equation}
		|\phi_k(x,\mu) - \phi_k(y,\beta)| \leq C \left ( |x-y| + d_r(\mu,\beta) \right ),
	\end{equation}
	which proves that the approximation is Lipschitz continuous with respect to the $|\,\cdot\,|\times d_r(\cdot,\cdot)$-distance.
	
	\textit{Proof of (ii): Uniform convergence.} First, we show that $\phi_k \to \phi$ pointwise as $k\to\infty$. Indeed, note that
	\begin{equation}\label{u_k_pointwise_convergence}
		\begin{split}
			&|\phi(x,\mu) - \phi_k(x,\mu)|\\
			&\leq \left | \phi(x,\mu) - \tilde{\mathbb{E}} \left [ \phi \left (x, \frac{1}{k} \sum_{i=1}^k \delta_{\tilde{X}_i} \right ) \right ] \right | + \left | \tilde{\mathbb{E}} \left [ \phi \left (x, \frac{1}{k} \sum_{i=1}^k \delta_{\tilde{X}_i} \right ) \right ] - \phi_k(x,\mu) \right |.
		\end{split}
	\end{equation}
	For the first term on the right-hand side, due to the Lipschitz continuity of $\phi$, we have
	\begin{equation}
		\left | \phi(x,\mu) - \tilde{\mathbb{E}} \left [ \phi \left (x, \frac{1}{k} \sum_{i=1}^k \delta_{\tilde{X}_i} \right ) \right ] \right | \leq C \tilde{\mathbb{E}} \left [ d_r\left (\mu,\frac{1}{k} \sum_{i=1}^k \delta_{\tilde{X}_i} \right ) \right ],
	\end{equation}
	where the right-hand side converges to zero as $k\to \infty$, see \cite[Theorem 5.8]{carmona_delarue_2018}. In order to treat the second term on the right-hand side of \eqref{u_k_pointwise_convergence}, we note that due to the Lipschitz continuity of $\phi$, we have for all $x_0,\dots,x_k\in \mathbb{R}^d$,
	\begin{equation}
		\begin{split}
			&\left | \phi \left ( x_0, \frac{1}{k} \sum_{i=1}^k \delta_{x_i} \right ) - h_{k,1/k}(x_0,\dots,x_k) \right |\\
			&\leq \int_{\mathbb{R}^{d(k+1)}} \left | \phi \left ( x_0, \frac{1}{k} \sum_{i=1}^k \delta_{x_i} \right ) - \phi \left ( x_0-y_0, \frac{1}{k} \sum_{i=1}^k \delta_{x_i-y_i} \right ) \right | \eta_{\frac1k}(y_0) \cdots \eta_{\frac1k}(y_k) \mathrm{d}y_0 \cdots \mathrm{d}y_k\\
			&\leq C \int_{\mathbb{R}^{d(k+1)}} \left ( | y_0 | + d_r\left ( \frac{1}{k} \sum_{i=1}^k \delta_{x_i}, \frac{1}{k} \sum_{i=1}^k \delta_{x_i-y_i} \right ) \right ) \eta_{\frac1k}(y_0) \cdots \eta_{\frac1k}(y_k) \mathrm{d}y_0 \cdots \mathrm{d}y_k\\
			&\leq C \int_{\mathbb{R}^{d(k+1)}} \left ( | y_0 | + |\mathbf{y}|_r \right ) \eta_{\frac1k}(y_0) \cdots \eta_{\frac1k}(y_k) \mathrm{d}y_0 \cdots \mathrm{d}y_k\\
			&\leq \frac{C}{k}.
		\end{split}
	\end{equation}
	Therefore, we have
	\begin{equation}
		\begin{split}
			&\left | \tilde{\mathbb{E}} \left [ \phi \left (x, \frac{1}{k} \sum_{i=1}^k \delta_{\tilde{X}_i} \right ) \right ] - \phi_{k}(x,\mu) \right |\\
			&= \left | \tilde{\mathbb{E}} \left [ \phi \left ( x, \frac{1}{k} \sum_{i=1}^k \delta_{\tilde{X}_i} \right ) - h_{k,1/k}(x,\tilde{X}_1,\dots,\tilde{X}_k) \right ] \right | \leq \frac{C}{k},
		\end{split}
	\end{equation}
	where the constant is independent of $k\in \mathbb{N}$. This concludes the proof that the right-hand side of \eqref{u_k_pointwise_convergence} tends to zero as $k\to\infty$, and therefore proves the pointwise convergence. The uniform convergence now follows from part $(i)$: Due to the Lipschitz continuity with uniform Lipschitz constant, the sequence $(\phi_k)_{k\in \mathbb{N}}$ is equicontinuous and bounded on bounded sets of $\mathcal{P}_r(\mathbb{R}^d)$. Therefore, since the sets $\mathfrak{M}^2_R$ are relatively compact in $\mathcal{P}_r(\mathbb{R}^d)$, it follows from the Arzel\`a--Ascoli theorem that $\phi_k$ converges to $\phi$ uniformly on bounded subsets of $\mathbb{R}^d\times \mathcal{P}_2(\mathbb{R}^d)$.
	
	\textit{Proof of (iii): $C^{1,1}$ regularity of the lift.} For the lift $\tilde{\phi}_{k} : \mathbb{R}^d \times E \to \mathbb{R}$ of $\phi_{k}$, we have
	\begin{equation}\label{C11_proof}
		\begin{split}
			&\| D_{(x,X)} \tilde{\phi}_k(x,X) - D_{(x,X)} \tilde{\phi}_k(y,Y) \|_{\mathbb{R}^d\times E}^2\\
			&\leq C | D_x \tilde{\phi}_k(x,X) - D_x \tilde{\phi}_k(y,X) |^2 + C |D_x \tilde{\phi}_k(y,X) - D_x \tilde{\phi}_k(y,Y) |^2\\
			&\qquad + C \| D_X \tilde{\phi}_k(x,X) - D_X \tilde{\phi}_k(y,X) \|_E^2 + C \| D_X \tilde{\phi}_k(y,X) - D_X \tilde{\phi}_k(y,Y) \|_E^2\\
			& \leq C | D_x \phi_k(x,X_{\texttt{\#}}\mathcal{L}^1) - D_x \phi_k(y,X_{\texttt{\#}}\mathcal{L}^1) |^2 + C |D_x \phi_k(y,X_{\texttt{\#}}\mathcal{L}^1) - D_x \phi_k(y,Y_{\texttt{\#}}\mathcal{L}^1) |^2\\
			&\qquad + C \int_{\Omega} | \partial_{\mu} \phi_k(x,X_{\texttt{\#}}\mathcal{L}^1)(X(\omega)) - \partial_{\mu} \phi_k(y,X_{\texttt{\#}}\mathcal{L}^1)(X(\omega)) |^2 \mathrm{d}\omega\\
			&\qquad + C \int_{\Omega} \left | \partial_{\mu} \phi_k(y,X_{\texttt{\#}}\mathcal{L}^1)(X(\omega)) - \partial_{\mu}\phi_k(y,Y_{\texttt{\#}}\mathcal{L}^1)(X(\omega)) \right |^2 \mathrm{d}\omega \\
			&\qquad + C \int_{\Omega} \left | \partial_{\mu} \phi_k(y,Y_{\texttt{\#}}\mathcal{L}^1)(X(\omega)) - \partial_{\mu} \phi_k(y,Y_{\texttt{\#}}\mathcal{L}^1)(Y(\omega)) \right |^2 \mathrm{d}\omega.
		\end{split}
	\end{equation}
	Let us first consider the second to last term. Using the boundedness of $\partial_{\mu}^2 \phi_k$, we obtain
	\begin{equation}
		\begin{split}
			& \int_{\Omega} \left | \partial_{\mu} \phi_k(y,X_{\texttt{\#}}\mathcal{L}^1)(X(\omega)) - \partial_{\mu}\phi_k(y,Y_{\texttt{\#}}\mathcal{L}^1)(X(\omega)) \right |^2 \mathrm{d}\omega\\
			&\leq \int_{\Omega} \sup_{t\in [0,1]} \left | \tilde{\mathbb{E}} \left [ \partial_{\mu}^2 \phi_k(y,(t\tilde X+(1-t) \tilde Y)_{\texttt{\#}}\mathcal{L}^1)(X(\omega),t \tilde{X}+(1-t) \tilde{Y}) (\tilde X - \tilde Y) \right ] \right |^2 \mathrm{d}\omega\\
			&\leq C \tilde{\mathbb{E}} \left [ |\tilde X - \tilde Y |^2 \right ] = C \|X-Y\|_E^2,
		\end{split}
	\end{equation}
	where $(\tilde X,\tilde Y)$ is a copy of the random vector $(X,Y)$. Here, we used the same argument as in \cite[Remark 5.27]{carmona_delarue_2018} combined with the mean value theorem for vector-valued functions. Similarly, for the last term in equation \eqref{C11_proof}, using the boundedness of $D_{x_1}\partial_{\mu} \phi_k$, we obtain
	\begin{equation}
		\begin{split}
			&\int_{\Omega} \left | \partial_{\mu} \phi_k(y,Y_{\texttt{\#}}\mathcal{L}^1)(X(\omega)) - \partial_{\mu} \phi_k(y,Y_{\texttt{\#}}\mathcal{L}^1)(Y(\omega)) \right |^2 \mathrm{d}\omega\\
			&\leq \int_{\Omega} \sup_{t\in [0,1]} \left |D_{x_1}\partial_{\mu} \phi_k(y,Y_{\texttt{\#}}\mathcal{L}^1)(tX(\omega)+(1-t) Y(\omega)) (X(\omega)-Y(\omega)) \right |^2 \mathrm{d}\omega\\
			&\leq C \|X-Y\|_E^2.
		\end{split}
	\end{equation}
	The remaining terms on the right-hand side of equation \eqref{C11_proof} can be treated with similar arguments. This proves the $C^{1,1}$ regularity of $\tilde{\phi}_{k}$\footnote{Note that the Lipschitz constant of the derivative of $\tilde{\phi}_{k}$ depends on $k$.}.
	
	\textit{Proof of (iv): Convexity of the lift.} Let $(\tilde{X}_i,\tilde{Y}_i)$, $i=1,\dots,k$, be independent copies of the random vector $(X,Y)$. Then, we have for $\lambda\in [0,1]$,
	\begin{equation}
		\begin{split}
			&\tilde{\phi}_{k}(\lambda x + (1-\lambda)y,\lambda X + (1-\lambda)Y) = \phi_{k}(\lambda x+(1-\lambda)y, (\lambda X + (1-\lambda) Y)_{\texttt{\#}} \mathcal{L}^1)\\
			&= \tilde{\mathbb{E}} \bigg [ \int_{\mathbb{R}^{d(k+1)}} \phi \bigg (\lambda x + (1-\lambda)y-y_0, \frac{1}{k} \sum_{i=1}^k \delta_{\lambda \tilde{X}_i + (1-\lambda)\tilde{Y}_i-y_i} \bigg )\\
			&\qquad\qquad\qquad\qquad\qquad\qquad\qquad\qquad\qquad\qquad \times \eta_{\frac1k}(y_0)\cdots \eta_{\frac1k}(y_k) \mathrm{d}y_0 \cdots \mathrm{d}y_k \bigg ]\\
			&= \tilde{\mathbb{E}} \bigg [ \int_{\mathbb{R}^{d(k+1)}} \tilde{\phi} \bigg (\lambda x + (1-\lambda)y-y_0, \sum_{i=1}^k (\lambda \tilde{X}_i + (1-\lambda)\tilde{Y}_i-y_i) \mathbf{1}_{A^k_i} \bigg )\\
			&\qquad\qquad\qquad\qquad\qquad\qquad\qquad\qquad\qquad\qquad \times \eta_{\frac1k}(y_0)\cdots \eta_{\frac1k}(y_k) \mathrm{d}y_0 \cdots \mathrm{d}y_k \bigg ].
		\end{split}
	\end{equation}
	Thus, using the convexity of $\tilde{\phi}$, we obtain
	\begin{equation}
		\begin{split}
			&\tilde{\phi}_{k}(\lambda x + (1-\lambda)y,\lambda X + (1-\lambda)Y)\\
			&\geq \tilde{\mathbb{E}} \left [ \int_{\mathbb{R}^{d(k+1)}} \lambda \tilde{\phi} \left (x-y_0, \sum_{i=1}^k (\tilde{X}_i-y_i) \mathbf{1}_{A^k_i} \right ) \eta_{\frac1k}(y_0)\cdots \eta_{\frac1k}(y_k) \mathrm{d}y_0 \cdots \mathrm{d}y_k \right ]\\
			&\quad + \tilde{\mathbb{E}} \left [ \int_{\mathbb{R}^{d(k+1)}} (1-\lambda) \tilde{\phi} \left (y-y_0, \sum_{i=1}^k (\tilde{Y}_i-y_i) \mathbf{1}_{A^k_i} \right ) \eta_{\frac1k}(y_0)\cdots \eta_{\frac1k}(y_k) \mathrm{d}y_0 \cdots \mathrm{d}y_k \right ]\\
			&= \tilde{\mathbb{E}} \left [ \int_{\mathbb{R}^{d(k+1)}} \lambda \phi \left ( x-y_0,\frac{1}{k} \sum_{i=1}^k \delta_{\tilde{X}_i-y_i} \right ) \eta_{\frac1k}(y_0)\cdots \eta_{\frac1k}(y_k) \mathrm{d}y_0 \cdots \mathrm{d}y_k \right ]\\
			&\quad + \tilde{\mathbb{E}} \left [ \int_{\mathbb{R}^{d(k+1)}} (1-\lambda) \phi \left (y-y_0, \frac{1}{k} \sum_{i=1}^k \delta_{\tilde{Y}_i-y_i} \right )\eta_{\frac1k}(y_0)\cdots \eta_{\frac1k}(y_k) \mathrm{d}y_0 \cdots \mathrm{d}y_k \right ]\\
			&= \lambda \tilde{\phi}_{k}(x,X) + (1-\lambda) \tilde{\phi}_{k}(y,Y),
		\end{split}
	\end{equation}
	which proves the convexity of the lift of $\phi_k$.
\end{proof}

\subsection{Convergence of the Approximating Value Functions}

\begin{assumption}\label{assumption_running_terminal_cost_relaxed}
	Let $r\in [1,2)$.
	\begin{enumerate}[label=(\roman*)]
		\item Let $l_1:\mathbb{R}^d \times \mathcal{P}_2(\mathbb{R}^d)\to \mathbb{R}$ be Lipschitz continuous with respect to the $|\,\cdot\,|\times d_r(\cdot,\cdot)$-distance, i.e., there is a constant $C$ such that
		\begin{equation}
			|l_1(x,\mu) - l_1(y,\beta)| \leq C(|x-y|+d_r(\mu,\beta))
		\end{equation}
		for all $x,y\in \mathbb{R}^d$ and $\mu,\beta\in \mathcal{P}_2(\mathbb{R}^d)$.
		\item Let there be constants $C_1, C_2,C_3$ such that $l_2\in C^{1,1}(\mathbb{R}^d)$ satisfies
		\begin{equation}
			-C_1 + C_2|p|^2 \leq l_2(p) \leq C_1 + C_3|p|^2
		\end{equation}
		for all $p\in \mathbb{R}^d$. Moreover, let $l_2$ be convex.
		\item Let $\mathcal{U}_T$ be Lipschitz continuous on $\mathcal{P}_r(\mathbb{R}^d)$.
	\end{enumerate}
\end{assumption}

\begin{theorem}
	Let Assumptions \ref{Assumption_Linear_State_Equation} and \ref{assumption_running_terminal_cost_relaxed} be satisfied. Then the projection $V_n$ of $V$ defined in \eqref{Projection_V_n} coincides with the finite dimensional value function $u_n$ defined in \eqref{finite_dimensional_value_function}.
\end{theorem}

\begin{remark}
	Note that in comparison with Assumption \ref{Assumption_running_terminal_cost}, we dropped the $C^{1,1}$ regularity assumption on the lifts of the coefficients $l_1$ and $\mathcal{U}_T$ of the cost functional. Instead, we consider an approximation as in Theorem \ref{Theorem_Approximation_Wasserstein_Space}.
\end{remark}

\begin{proof}
	For the coefficients $l_1:\mathbb{R}^d \times \mathcal{P}_2(\mathbb{R}^d) \to \mathbb{R}$ and $\mathcal{U}_T:\mathcal{P}_2(\mathbb{R}^d) \to \mathbb{R}$ in the cost functional, we consider their approximations from Theorem \ref{Theorem_Approximation_Wasserstein_Space}, denoted by $l_1^k$ and $\mathcal{U}_T^k$. Thus, the approximation of the lifted value function is given by
	\begin{equation}
		\begin{split}
			V^k(t,X) &:= \inf_{a(\cdot) \in \mathcal{A}_t} J^k(t,X,a(\cdot))\\
			&:= \inf_{a(\cdot) \in \mathcal{A}_t} \mathbb{E} \left [ \int_t^T \left ( L_1^k(X(s)) + L_2(a(s)) \right ) \mathrm{d}s + U_T^k(X(T)) \right ]
		\end{split}
	\end{equation}
	where
	\begin{equation}
		L_1^k(X) = \int_{\Omega} l_1^k(X(\omega),X_{\texttt{\#}}\mathcal{L}^1) \mathrm{d}\omega
	\end{equation}
	and
	\begin{equation}
		U_T^k(X) = \mathcal{U}^k_T(X_{\texttt{\#}}\mathcal{L}^1).
	\end{equation}
	Furthermore, the approximation of the value function for the finite dimensional problem is given by
	\begin{equation}
		\begin{split}
			u_n^k(t,\mathbf{x}) &:= \inf_{\mathbf{a}(\cdot)\in \mathcal{A}^n_t} J_n^k(t,\mathbf{x};\mathbf{a}(\cdot))\\
			&:= \mathbb{E} \left [ \int_t^T \left ( \frac1n \sum_{i=1}^n \left ( l_1^k(X_i(s),\mu_{\mathbf{X}(s)}) + l_2(a_i(s)) \right ) \right ) \mathrm{d}s + \mathcal{U}_T^k(\mu_{\mathbf{X}(T)}) \right ].
		\end{split}
	\end{equation}
	
	From Subsection \ref{subsection:C11_regularity} and Theorem \ref{theorem:projection_C11}, we know that the projection of $V^k$ coincides with $u_n^k$. Now, it remains to show that $V^k$ converges to $V$ as $k\to \infty$ and $u_n^k$ converges to $u_n$ as $k\to \infty$.

	Consider
	\begin{equation}
		\tilde{H}^k(X,\beta,P) := \int_{\Omega} \left ( -b(X(\omega),\beta) \cdot P(\omega) - l^k_1(X(\omega),\beta) + \sup_{q\in \mathbb{R}^d} (q\cdot P(\omega) - l_2(q)) \right ) \mathrm{d}\omega.
	\end{equation}
	The functions $V^k$ are viscosity solutions of \eqref{intro:lifted_HJB} with $\tilde{H}, U_T$ replaced by $\tilde{H}^k,U_T^k$. Since the $V^k(t,\cdot)$ are Lipschitz continuous with Lipschitz constants independent of $k$, using the same arguments as in the proof of \cite[Proposition 3.1]{mayorga_swiech_2023} we obtain that there is a constant $K>0$ such that
	\begin{equation}
		V^k(t,X) = \inf_{a(\cdot) \in \mathcal{A}^K_t} \mathbb{E} \left [ \int_t^T \left ( L_1^k(X(s)) + L_2(a(s)) \right ) \mathrm{d}s + U_T^k(X(T)) \right ],
	\end{equation}
	where $\mathcal{A}_t^K = \{ a(\cdot) \in \mathcal{A}_t : a(\cdot) \text{ has values in } B_K(0) \}$ and the same set $\mathcal{A}_t^K$ can be used in the definition of $V$. Therefore,
	\begin{equation}
		\begin{split}
			|V(t,X) - V^k(t,X)| &= \left | \inf_{a(\cdot)\in \mathcal{A}_t^K} J(t,X;a(\cdot)) - \inf_{a(\cdot)\in \mathcal{A}_t^K} J^k(t,X;a(\cdot)) \right |\\
			&\leq \sup_{a(\cdot)\in \mathcal{A}^K_t} \left | J(t,X;a(\cdot)) - J^k(t,X;a(\cdot)) \right |.
		\end{split}
	\end{equation}
	Next, we are going to show that the right-hand side converges to zero as $k$ tends to infinity. We first consider the terms involving $L_1^k$ and $L_1$. For any control $a(\cdot)\in \mathcal{A}_t^K$ and any $R>0$, we have
	\begin{equation}\label{convergence_J_k_Step_1}
		\begin{split}
			&\mathbb{E} \left [ \int_t^T \int_{\Omega} \left | l_1^k(X(s,\omega),X(s)_{\texttt{\#}} \mathcal{L}^1) - l_1(X(s,\omega),X(s)_{\texttt{\#}} \mathcal{L}^1) \right | \mathrm{d}\omega \mathrm{d}s \right ]\\
			&\leq \mathbb{E} \bigg [ \int_t^T \mathbf{1}_{\{\|X(s) \|_E \leq R\}} \\
			&\qquad\qquad \int_{\Omega} \mathbf{1}_{\{|X(s,\omega) | \leq R\}} \left | l_1^k(X(s,\omega),X(s)_{\texttt{\#}} \mathcal{L}^1) - l_1(X(s,\omega),X(s)_{\texttt{\#}} \mathcal{L}^1) \right | \mathrm{d}\omega \mathrm{d}s \bigg ]\\
			&\quad + \mathbb{E} \left [ \int_t^T \int_{\Omega} \mathbf{1}_{\{|X(s,\omega) | > R\}} \left | l_1^k(X(s,\omega),X(s)_{\texttt{\#}} \mathcal{L}^1) - l_1(X(s,\omega),X(s)_{\texttt{\#}} \mathcal{L}^1) \right | \mathrm{d}\omega \mathrm{d}s \right ]\\
			&\quad +\mathbb{E} \left [ \int_t^T \mathbf{1}_{\{\|X(s) \|_E > R\}} \int_{\Omega} \left | l_1^k(X(s,\omega),X(s)_{\texttt{\#}} \mathcal{L}^1) - l_1(X(s,\omega),X(s)_{\texttt{\#}} \mathcal{L}^1) \right | \mathrm{d}\omega \mathrm{d}s \right ].
		\end{split}
	\end{equation}
	For the first term on the right-hand side, we have
	\begin{equation}
		\begin{split}
			&\mathbb{E} \bigg [ \int_t^T \mathbf{1}_{\{\|X(s) \|_E \leq R\}}\\
			&\qquad\qquad\qquad \int_{\Omega} \mathbf{1}_{\{|X(s,\omega) | \leq R\}} \left ( l_1^k(X(s,\omega),X(s)_{\texttt{\#}} \mathcal{L}^1) - l_1(X(s,\omega),X(s)_{\texttt{\#}} \mathcal{L}^1) \right ) \mathrm{d}\omega \mathrm{d}s \bigg ]\\
			&\leq \mathbb{E} \left [ \int_t^T \int_{\Omega} \sup_{(x,\mu)\in B_R(0) \times \mathfrak{M}^2_{R^2}} \left | l_1^k(x,\mu) - l_1(x,\mu) \right | \mathrm{d}\omega \mathrm{d}s \right ],
		\end{split}
	\end{equation}
	which converges to zero as $k$ tends to infinity due to Theorem \ref{Theorem_Approximation_Wasserstein_Space}. For the second term on the right-hand side of inequality \eqref{convergence_J_k_Step_1}, we obtain using H\"older's inequality and the Lipschitz continuity of $l_1^k$ and $l_1$
	\begin{equation}
		\begin{split}
			&\mathbb{E} \left [ \int_t^T \int_{\Omega} \mathbf{1}_{\{|X(s,\omega) | > R\}} \left ( l_1^k(X(s,\omega),X(s)_{\texttt{\#}} \mathcal{L}^1) - l_1(X(s,\omega),X(s)_{\texttt{\#}} \mathcal{L}^1) \right ) \mathrm{d}\omega \mathrm{d}s \right ]\\
			&\leq (\mathbb{P} \otimes \mathcal{L}^1 \otimes \mathcal{L}^1)( \{ (\omega^{\prime},s,\omega)\in \Omega^{\prime}\times [t,T]\times\Omega: |X(s,\omega)|>R \})^{\frac12}\\
			&\quad \times C \left ( \mathbb{E} \left [ \int_t^T \int_{\Omega} \left ( 1 + |X(s,\omega)|^2 + \mathcal{M}_2(X(s)_{\texttt{\#}}\mathcal{L}^1) \right ) \mathrm{d}\omega \mathrm{d}s \right ] \right )^{\frac12}.
		\end{split}
	\end{equation}
	For the first term, we have due to Lemma \ref{Infinite_Dimensional_A_Priori}
	\begin{equation}
		\begin{split}
			&(\mathbb{P} \otimes \mathcal{L}^1 \otimes \mathcal{L}^1)( \{ (\omega^{\prime},s,\omega)\in \Omega^{\prime}\times [t,T]\times\Omega: |X(s,\omega)|>R \})\\
			&\leq \frac{\| X(\cdot) \|_{L^2(\Omega^{\prime}\times [t,T]\times \Omega)}^2}{R^2}\\
			&\leq \frac{C \left ( 1+\|X\|_E^2 + \|a(\cdot)\|_{M^2(t,T;E)}^2 \right )}{R^2} \leq \frac{C \left ( 1+\|X\|_E^2 \right )}{R^2},
		\end{split}
	\end{equation}
	where in the last inequality we used the fact that $a(\cdot)\in \mathcal{A}_t^K$. For the second term, using again Lemma \ref{Infinite_Dimensional_A_Priori}, we obtain
	\begin{equation}\label{a_priori_estimate}
		\begin{split}
			&\mathbb{E} \left [ \int_t^T \int_{\Omega} \left ( 1 + |X(s,\omega)|^2 + \mathcal{M}_2(X(s)_{\texttt{\#}}\mathcal{L}^1) \right ) \mathrm{d}\omega \mathrm{d}s \right ]\\
			&\leq C(1+\|X\|_E^2 + \|a(\cdot)\|_{M^2(t,T;E)}^2 ) \leq C (1 + \|X\|_E^2).
		\end{split}
	\end{equation}
	For the third term on the right-hand side of inequality \eqref{convergence_J_k_Step_1}, we have
	\begin{equation}
		\begin{split}
			&\mathbb{E} \left [ \int_t^T \mathbf{1}_{\{\|X(s) \|_E > R\}} \int_{\Omega} \left ( l_1^k(X(s,\omega),X(s)_{\texttt{\#}} \mathcal{L}^1) - l_1(X(s,\omega),X(s)_{\texttt{\#}} \mathcal{L}^1) \right ) \mathrm{d}\omega \mathrm{d}s \right ]\\
			&\leq (\mathbb{P} \otimes \mathcal{L}^1)( \{ (\omega^{\prime},s)\in \Omega^{\prime}\times [t,T]: \|X(s)\|_E>R \})^{\frac12}\\
			&\quad \times C \left ( \mathbb{E} \left [ \int_t^T \int_{\Omega} \left ( 1 + |X(s,\omega)|^2 + \mathcal{M}_2(X(s)_{\texttt{\#}}\mathcal{L}^1) \right ) \mathrm{d}\omega \mathrm{d}s \right ] \right )^{\frac12}.
		\end{split}
	\end{equation}
	For the first term, we have due to Lemma \ref{Infinite_Dimensional_A_Priori}
	\begin{equation}
		(\mathbb{P} \otimes \mathcal{L}^1)( \{ (\omega^{\prime},s)\in \Omega^{\prime}\times [t,T]: \|X(s)\|_E>R \}) \leq \frac{\|X(\cdot)\|_{L^2(\Omega^{\prime}\times[t,T];E)}^2}{R^2} \leq \frac{C \left ( 1+\|X\|_E^2 \right )}{R^2}. 
	\end{equation}
	The second term can be estimated as in \eqref{a_priori_estimate}.
	
	We note that the estimates above do not depend on $a(\cdot)\in  \mathcal{A}^K_t$. The terms involving $U_T^k$ and $U_T$ are estimated similarly. Therefore, taking the limit superior as $k\to\infty$ and then sending $R\to\infty$ 
	proves the pointwise convergence of $V^k$ to $V$. The convergence of $u^k_n$ to $u_n$ follows along the same lines.
\end{proof}

\appendix
\section{Definition of Viscosity Solution of Equations in Hilbert Spaces}

We briefly recall the notion of viscosity solution for equations in Hilbert spaces as introduced in \cite{crandall_lions_1985} for first order equations and generalized to second order equations in \cite{lions_1989}. See also \cite[Section 3.3.1]{fabbri_gozzi_swiech_2017}. Let us emphasize that the equations in the present manuscript do not involve an unbounded operator.

Let $\mathcal{E}$ be a real Hilbert space and let $F:(0,T)\times \mathcal{E} \times \mathbb{R} \times \mathcal{E} \times S(\mathcal{E}) \to \mathbb{R}$, where $S(\mathcal{E})$ denotes the set of bounded, self-adjoint operators on $\mathcal{E}$. We assume that
$F(t,X,s,P,Z)\leq F(t,X,s,P,Z')$ for all $t\in (0,T),X\in \mathcal{E},s\in\mathbb{R},P\in \mathcal{E},Z,Z'\in S(\mathcal{E})$ such that $Z\leq Z'$. Let $G:\mathcal{E} \to \mathbb{R}$. We consider terminal value problems of the type
\begin{equation}\label{appendix_pde}
	\begin{cases}
		\partial_t v + F(t,X,v,Dv,D^2v) = 0, \quad (t,X) \in (0,T)\times \mathcal{E}\\
		v(T,X) = G(X), \quad X\in \mathcal{E}.
	\end{cases}
\end{equation}

\begin{definition}\label{appendix_definition_pde}
	An upper semicontinuous function $v : (0,T]\times \mathcal{E} \to \mathbb{R}$ that is bounded on bounded sets is a viscosity subsolution of \eqref{appendix_pde} if $v(T,X) \leq G(X)$, $X\in \mathcal{E}$, and whenever $u-\varphi$ has a local maximum at a point $(t,X) \in (0,T)\times \mathcal{E}$ for a function $\varphi \in C^{1,2}((0,T)\times \mathcal{E})$ then
	\begin{equation}
		\partial_t \varphi(t,X) + F(t,X,v(t,X),D\varphi(t,X),D^2\varphi(t,X)) \geq 0.
	\end{equation}
	A lower semicontinuous function $v : (0,T]\times \mathcal{E} \to \mathbb{R}$ that is bounded on bounded sets is a viscosity supersolution of \eqref{appendix_pde} if $v(T,X) \geq G(X)$, $X\in \mathcal{E}$, and whenever $u-\varphi$ has a local minimum at a point $(t,X) \in (0,T)\times \mathcal{E}$ for a function $\varphi \in C^{1,2}((0,T)\times \mathcal{E})$ then
	\begin{equation}
		\partial_t \varphi(t,X) + F(t,X,v(t,X),D\varphi(t,X),D^2\varphi(t,X)) \leq 0.
	\end{equation}
	A viscosity solution of \eqref{appendix_pde} is a function which is both a viscosity subsolution and a viscosity supersolution of \eqref{appendix_pde}.
\end{definition}

\end{document}